\documentclass[reqno, 11pt]{amsart} 
\usepackage{amsfonts, amsmath, amssymb, amsthm}
\usepackage[margin=2.71cm, heightrounded]{geometry}
\usepackage{hhline, float}
\usepackage{nccmath}
\usepackage[colorlinks=true]{hyperref}

\newtheorem{theorem}{Theorem}[section]
\newtheorem{definition}[theorem]{Definition}
\newtheorem{prop}[theorem]{Proposition}
\newtheorem{lemma}[theorem]{Lemma}
\newtheorem{cor}[theorem]{Corollary}
\newtheorem{conv}[theorem]{Convention}
\newtheorem{conj}[theorem]{Conjecture}

\theoremstyle{definition}
\newtheorem{remark}[theorem]{Remark}

\newcommand{\ga}{\gamma}
\newcommand{\ep}{\epsilon}
\newcommand{\ka}{\kappa}
\newcommand{\vph}{\varphi}
\newcommand{\vth}{\vartheta}
\newcommand{\pa}{\partial}

\newcommand{\bg}{\bar{g}}
\newcommand{\bh}{\bar{h}}
\newcommand{\bw}{\bar{w}}
\newcommand{\bx}{\bar{x}}
\newcommand{\bz}{\bar{z}}
\newcommand{\bth}{\bar{\theta}}

\newcommand{\hg}{\hat{g}}
\newcommand{\hv}{\hat{v}}

\newcommand{\tf}{\tilde{f}}
\newcommand{\tg}{\tilde{g}}
\renewcommand{\th}{\tilde{h}}
\newcommand{\tu}{\tilde{u}}

\newcommand{\ty}{\tilde{y}}

\newcommand{\trh}{\tilde{\rho}}

\newcommand{\B}{\mathbb{B}}

\newcommand{\N}{\mathbb{N}}
\newcommand{\R}{\mathbb{R}}
\renewcommand{\S}{\mathbb{S}}

\newcommand{\mca}{\mathcal{A}}
\newcommand{\mcb}{\mathcal{B}}
\newcommand{\mcc}{\mathcal{C}}
\newcommand{\mcd}{\mathcal{D}}
\newcommand{\mce}{\mathcal{E}}
\newcommand{\mch}{\mathcal{H}}
\newcommand{\mci}{\mathcal{I}}
\newcommand{\mcj}{\mathcal{J}}
\newcommand{\mck}{\mathcal{K}}

\newcommand{\mcm}{\mathcal{M}}

\newcommand{\mcq}{\mathcal{Q}}
\newcommand{\mcr}{\mathcal{R}}
\newcommand{\mcs}{\mathcal{S}}
\newcommand{\mct}{\mathcal{T}}
\newcommand{\mcx}{\mathcal{X}}

\newcommand{\mz}{\mathfrak{m}_0}

\newcommand{\whv}{\widehat{V}}
\newcommand{\wtf}{\widetilde{F}}
\newcommand{\wtk}{\widetilde{K}}
\newcommand{\wtu}{\widetilde{U}}
\newcommand{\wtv}{\widetilde{V}}
\newcommand{\wmcd}{\widetilde{\mathcal{D}}}
\newcommand{\wmck}{\widetilde{\mathcal{K}}}

\newcommand{\ox}{\overline{X}}

\newcommand{\m}{{\textup{M}}}
\newcommand{\pb}{{\textup{PB}}}
\newcommand{\ph}{{\textup{PH}}}
\newcommand{\geo}{{\textup{g}}}
\newcommand{\bgg}{\bg_{\geo}}
\newcommand{\rhog}{\rho_{\geo}}
\newcommand{\pp}{{\textup{pp}}}

\newcommand{\la}{\left\langle}
\newcommand{\ra}{\right\rangle}
\renewcommand{\(}{\left(}
\renewcommand{\)}{\right)}

\allowdisplaybreaks
\numberwithin{equation}{section}

\begin{document}
\title[Weighted isoperimetric ratios]{Weighted isoperimetric ratios and extension problems \\ for fractional conformal Laplacians}

\author[S. Jin]{Sangdon Jin}
\address[Sangdon Jin]{Department of Mathematics Education, Chungbuk National University, 1 Chungdaero Seowon-Gu, Cheongju 28644, Republic of Korea}
\email{sangdonjin@cbnu.ac.kr}

\author[S. Kim]{Seunghyeok Kim}
\address[Seunghyeok Kim]{Department of Mathematics and Research Institute for Natural Sciences, College of Natural Sciences, Hanyang University, 222 Wangsimni-ro Seongdong-gu, Seoul 04763, Republic of Korea,
School of Mathematics, Korea Institute for Advanced Study, 85 Hoegiro Dongdaemun-gu, Seoul 02455, Republic of Korea.}
\email{shkim0401@hanyang.ac.kr shkim0401@gmail.com}

\begin{abstract}
We investigate a novel connection between the weighted isoperimetric problems and the weighted Poisson integrals of the extension problems for nonlocal elliptic operators.

We first derive sharp inequalities for the weighted Poisson integrals associated with degenerate elliptic equations on the half-space and the unit ball, and classify their extremizers.
The equations arise from the Caffarelli-Silvestre extension for the fractional Laplacian on the Euclidean space and its conformal transformation via the M\"obius transformation.

We next interpret the above sharp inequalities in a conformal geometric viewpoint.
For this aim, we formulate a variational problem involving a weighted isoperimetric ratio on a smooth metric measure space induced by a conformally compact Einstein (CCE) manifold.
Then, we prove that the variational problem is closely linked to the Chang-Gonz\'alez extension for a fractional conformal Laplacian on the conformal infinity of the CCE manifold,
and is reduced to the sharp inequality if the CCE manifold is either the Poincar\'e half-space or ball model.
We also find a criterion that ensures the existence of a smooth extremizer of the variational problem, and present a relevant conjecture.
\end{abstract}

\date{\today}
\maketitle


\section{Introduction}
The isoperimetric ratio is a fundamental concept in mathematics that has been a subject of investigation for centuries.
A number of researchers have developed various ideas, viewpoints, and techniques to comprehend it, and it remains a vibrant area of research.

This paper is motivated by the works of Hang, Wang, and Yan \cite{HWY, HWY2} who explored the isoperimetric ratio from a conformal perspective:
In \cite{HWY}, they derived a sharp inequality involving a function $f$ on the $n$-dimensional Euclidean space $\R^n$ ($n \ge 2$) and its harmonic extension $\mck_nf$ on the upper half-space $\R^{n+1}_+ := \R^n \times (0,\infty)$.
In \cite{HWY2}, they proved that the sharp inequality could be interpreted as a variational problem involving an isoperimetric ratio,
which can be defined in any $(n+1)$-dimensional smooth compact Riemannian manifold $\ox$ with boundary $M$.

It is well-known that the normal derivative of $\mck_nf$ equals the half-Laplacian $(-\Delta)^{1/2}f$ of $f$ on $\R^n$, up to a constant factor.
According to Caffarelli and Silvestre \cite{CS}, such a principle continues to hold in a broader context: For all $\ga \in (0,1)$, the fractional Laplacian $(-\Delta)^{\ga}f$ of $f$ on $\R^n$ equals a weighted normal derivative of the $\ga$-harmonic extension $\mck_{n,\ga}f$ on $\R^{n+1}_+$.
Chang and Gonz\'alez \cite{CG} established its further generalization: Let $P^{\ga}_{g^+,\bh}$ be the fractional conformal Laplacian on the conformal infinity $(M^n,[\bh])$ of a conformally compact Einstein (CCE for short) manifold $(X^{n+1},g^+)$ studied by Graham and Zworski \cite{GZ}.
Then $P^{\ga}_{g^+,\bh}f$ equals a weighted normal derivative of the solution $\mck^{1-2\ga}_{\bg,\rho}f$ to a certain degenerate elliptic equation in $X$ with Dirichlet data $f$ on $M$; here, $\rho$ denotes a defining function of $M$ and $\bg := \rho^2g^+$.

In light of the theorems in \cite{CS, CG} mentioned above, it is natural to ask whether there exists a sharp inequality involving both $f$ and $\mck^{1-2\ga}_{\bg,\rho}f$
and, if so, whether one can interpret it as a variational problem involving a version of the isoperimetric ratio. In this paper, we affirmatively answer these questions.
By appealing a geometric interpretation of the Chang-Gonz\'alez extension theorem due to Case and Chang \cite{CC}, we will introduce a weighted isoperimetric ratio on smooth metric measure space induced by the CCE manifold $(X,g^+)$.
Then we will formulate a variational problem involving the ratio and examine its connection to the sharp inequality.
We will also investigate the existence of a smooth extremizer of the variational problem.
Certainly, our main results have analogies to those of Yamabe-type problems such as the Yamabe problem \cite{Au}, the CR Yamabe problem \cite{JL}, the boundary Yamabe problem \cite{Es2, Es}, and the fractional Yamabe problem \cite{GQ}, among others.

For several decades, researchers have investigated the relationship between CCE manifolds and their conformal infinity
because of its fundamental importance in geometry and theoretical physics, as manifested by the AdS/CFT correspondence.
Our study on the weighted isoperimetric ratios may provide a new way to understand such a relationship.

A few studies relevant to our work include those on the resolvent operator \cite{MM}, the relationship between scattering theory and the fractional GJMS operators \cite{GZ}, and the regularity of CCE metrics \cite{CDLS}.
Additionally, studies on the extension theorem of the fractional conformal Laplacian and fractional GJMS operators \cite{CC, CG}, the weighted renormalized volume \cite{Go2},
the fractional Yamabe problem and flow \cite{DSV, GQ, GW, KMW, KMW2, MN, MN2}, and energy inequalities involving the fractional GJMS operator \cite{Ca2} are closely linked to our results; see also the survey \cite{Go}.

\subsection{Isoperimetric ratio over scalar-flat conformal classes}\label{subsec:iso}
In this subsection, we provide a more detailed explanation of the works by Hang et al. \cite{HWY, HWY2} and their follow-ups.

\medskip
Let $n \in \N$, $N := n+1 \ge 3$, and $|\B^N|$ be for the volume of the open unit ball $\B^N$ in $\R^N$. We define the harmonic extension of $f \in L^{\frac{2(N-1)}{N-2}}(\R^n)$ on the upper half-space $\R^N_+$ by
\[(\mck_nf)(x) = k_n \int_{\R^n} \frac{x_N}{(|\bx-\bw|^2+x_N^2)^{\frac{N}{2}}} f(\bw) d\bw \quad \textup{for } x=(\bx,x_N) \in \R^N_+\]
where $k_n := \pi^{-\frac{N}{2}}\Gamma(\frac{N}{2}) = \frac{2}{N|\B^N|}$ is the normalizing constant such that $\mck_n 1 = 1$. In \cite{HWY}, Hang et al. derived a sharp inequality
\begin{equation}\label{eq:hwy}
\|\mck_n f\|_{L^{\frac{2N}{N-2}}(\R^N_+)} \le \left[N^{-\frac{N}{N-1}} \left|\B^N\right|^{-\frac{1}{N-1}}\right]^{\frac{N-2}{2N}} \|f\|_{L^{\frac{2(N-1)}{N-2}}(\R^n)} \quad \textup{for } f \in L^{\frac{2(N-1)}{N-2}}(\R^n),
\end{equation}
and classified its extremizers. Inequality \eqref{eq:hwy} is a higher-dimensional analog of Carleman's inequality in \cite{Car}, which appears in his proof of the isoperimetric inequality on a minimal surface.

In \cite{HWY2}, Hang et al. examined a variational problem involving the isoperimetric ratio over scalar-flat conformal classes on a smooth compact Riemannian manifold $(\ox^N,\bg)$ with boundary $M^n$:
\begin{equation}\label{eq:thetan}
\Theta_n\(\ox,[\bg]\) = \sup\left\{I_n\(\ox,\tg\): \tg \in \mca_{[\bg]}\right\}
\end{equation}
with
\[I_n\(\ox,\tg\) := \frac{\textup{Volume}_{\tg}\(\ox\)}{\textup{Area}_{\th}(M)^{\frac{N}{N-1}}}
\quad \textup{and} \quad \mca_{[\bg]} := \left\{\tg \in [\bg]: R_{\tg} = 0 \textup{ on } \ox\right\}.\]
Here, $[\bg]$ is the conformal class of $\bg$, $\th := \tg|_{TM}$, $TM$ is the tangent bundle of $M$, and $R_{\tg}$ is the scalar curvature on $(\ox,\tg)$. If we set the M\"obius transformation $\phi_\m: \R^N_+ \to \B^N$ as
\begin{equation}\label{eq:Mobius}
\phi_\m(x) = \frac{2(x+e_N)}{|x+e_N|^2}-e_N \quad \textup{where } e_N := (0,\ldots,0,1) \in \R^N,
\end{equation}
then $\phi_\m^{-1}(y) = \phi_\m(y)$ for $y \in \B^N$ and given any positive function $u \in C^{\infty}(\B^N)$,
\begin{equation}\label{eq:Mobius2}
\phi_\m^*\(u^{\frac{4}{N-2}} g_{\R^N}\) = v^{\frac{4}{N-2}} g_{\R^N} \quad \textup{where } v(x) = \frac{2^{\frac{N-2}{2}}u(\phi_\m(x))}{|x+e_N|^{N-2}},\ x \in \R^N_+.
\end{equation}
In light of \eqref{eq:Mobius2}, inequality \eqref{eq:hwy} is recast as
\[\Theta_n\(\overline{\R^N_+} \cup \{\infty\},\left[4|x+e_N|^{-4}g_{\R^N}\right]\) = \Theta_n\big(\overline{\B^N},[g_{\R^N}]\big) = N^{-\frac{N}{N-1}} \left|\B^N\right|^{-\frac{1}{N-1}}\]
where $\infty$ is a point at infinity and $g_{\R^N}$ is the Euclidean metric. Hence \eqref{eq:thetan} extends \eqref{eq:hwy} to any smooth compact Riemannian manifolds $(\ox,\bg)$, providing a geometric interpretation of \eqref{eq:hwy} as well.
As observed in \cite{HWY2}, $\mca_{[\bg]} \ne \emptyset$ if and only if $\lambda_1(L_{\bg}) > 0$,
where $\lambda_1(L_{\bg})$ is the first eigenvalue of the conformal Laplacian $L_{\bg} := -\Delta_{\bg} + \frac{N-2}{4(N-1)} R_{\bg}$ in $X$ with zero Dirichlet boundary condition on $M$.
Besides, the sign of $\lambda_1(L_{\tg})$ is independent of the choice of a metric $\tg$ in $[\bg]$. Assuming $\lambda_1(L_{\bg}) > 0$, Hang et al. proved that
\begin{equation}\label{eq:hwy20}
\Theta_n\big(\overline{\B^N},[g_{\R^N}]\big) \le \Theta_n\(\ox,[\bg]\) < \infty
\end{equation}
and $\Theta_n(\ox,[\bg])$ is achieved by a metric in $\mca_{[\bg]}$ provided
\begin{equation}\label{eq:hwy2}
\Theta_n\big(\overline{\B^N},[g_{\R^N}]\big) < \Theta_n\(\ox,[\bg]\).
\end{equation}
They also conjectured that if $(\ox,\bg)$ is not conformally diffeomorphic to $(\B^N,g_{\R^N})$, then \eqref{eq:hwy2} holds. So far, \eqref{eq:hwy2} has been confirmed for few cases:
\begin{enumerate}
\item[(i)] $N \ge 9$ and $M$ has a nonumbilic point (\cite{JX, CJR});
\item[(ii)] $N \ge 7$, $M$ is umbilic, and the Weyl tensor $W_{\bg}$ is nonzero at some point on $M$ (\cite{JX, CJR});
\item[(iii)] $N \ge 3$ and $(\ox,\bg) = (\{x \in \R^N: \ep<|x|<1\}, g_{\R^N})$ for sufficiently small $\ep > 0$ (\cite{GlZ}).
\end{enumerate}
Verifying \eqref{eq:hwy2} for the remaining cases is an interesting open question.

\subsection{Main results}
Let $n \in \N$, $n > 2\ga$, $\ga \in (0,1)$, and $\mz := 1-2\ga \in (-1,1)$. We define the $\ga$-harmonic extension of $f \in L^{\frac{2n}{n-2\ga}}(\R^n)$ on the upper half-space $\R^N_+$ by
\begin{equation}\label{eq:gaharext}
\(\mck_{n,\ga}f\)(x) = \int_{\R^n} \mck_{n,\ga}(x,\bw)f(\bw) d\bw \quad \textup{for } x=(\bx,x_N) \in \R^N_+
\end{equation}
where
\begin{equation}\label{eq:gaharext2}
\mck_{n,\ga}(x,\bw) := \ka_{n,\ga}\frac{x_N^{2\ga}}{\(|\bx-\bw|^2+x_N^2\)^{\frac{n+2\ga}{2}}}
\end{equation}
and $\ka_{n,\ga} := \pi^{-\frac{n}{2}} \Gamma(\frac{n+2\ga}{2})/\Gamma(\ga)$ is the normalizing constant such that $\mck_{n,\ga}1 = 1$.
According to \cite[Theorem 1.2]{CM} (see also \cite[Section 2.4]{CS}), if $f \in \dot{H}^{\ga}(\R^n) \subset L^{\frac{2n}{n-2\ga}}(\R^n)$, then $U = \mck_{n,\ga}f$ is the unique solution to the degenerate local elliptic equation
\begin{equation}\label{eq:degeq}\begin{cases}
-\textup{div}\(x_N^{\mz} \nabla U\) = 0 &\textup{in } \R^N_+,\\
U = f &\textup{on } \R^n,\\
U \in \dot{H}^{1,2}\(\R^N_+;x_N^{\mz}\),
\end{cases}
\end{equation}
and it holds that
\begin{equation}\label{eq:degeq2}
\frac{d_{\ga}}{2\ga} \lim_{x_N \to 0+} x_N^{\mz}{\frac{\pa U}{\pa x_N}} = (-\Delta)^{\ga} f \quad \textup{on } \R^N
\quad \textup{where } d_{\ga} := 2^{2\ga} \frac{\Gamma(\ga)}{\Gamma(-\ga)} < 0
\end{equation}
in the weak sense; see Definition \ref{def:fs1} for the definition of the function spaces.
The following theorem provides a sharp inequality \eqref{eq:main11} involving both $f$ and $\mck_{n,\ga}f$, and a classification on its extremizers. If $\ga = \frac{1}{2}$, then \eqref{eq:main11} is reduced to \eqref{eq:hwy}.
\begin{theorem}\label{thm:main11}
Assume that $n \in \N$, $n > 2\ga$, and $\ga \in (0,1)$. Then there exists the optimal constant $C_{n,\ga} > 0$ depending only on $n$ and $\ga$ such that
\begin{equation}\label{eq:main11}
\|\mck_{n,\ga}f\|_{L^{\frac{2(n-2\ga+2)}{n-2\ga}}(\R^N_+;x_N^{\mz})} \le C_{n,\ga}\|f\|_{L^{\frac{2n}{n-2\ga}}(\R^n)} \quad \text{for } f \in L^{\frac{2n}{n-2\ga}}(\R^n).\ \footnotemark
\end{equation}
\footnotetext{For all $\ga \in (0,1)$, it holds that
\[C_{n,\ga} = \|W_{1,0}\|_{L^{\frac{2(n-2\ga+2)}{n-2\ga}}(\R^N_+;x_N^{\mz})} \|w_{1,0}\|_{L^{\frac{2n}{n-2\ga}}(\R^n)}^{-1}
\quad \text{where }\begin{cases}
w_{1,0}(\bw) := (1+|\bw|^2)^{-\frac{n-2\ga}{2}} &\text{for } \bw \in \R^n,\\
W_{1,0} := \mck_{n,\ga}w_{1,0} &\text{in } \R^N_+.
\end{cases}\]
If $\ga = \frac{1}{2}$, then $W_{1,0}(\bx,x_N) = ((|\bx|^2+(x_N+1)^2)^{-\frac{N-2}{2}}$ for $(\bx,x_N) \in \R^N_+$,
so one can easily compute $C_{n,1/2} = N^{-\frac{N-2}{2(N-1)}} |\B^N|^{{-\frac{N-2}{2N(N-1)}}}$ as given in \eqref{eq:hwy}.
If $\ga \in (0,1) \setminus \{\frac{1}{2}\}$, it is difficult to find the exact value of $C_{n,\ga}$. \label{fnlabel}}Furthermore,
the equality holds if and only if there are numbers $c \in \R$, $\lambda > 0$, and $\bx_0 \in \R^n$ such that
\begin{equation}\label{eq:bubble}
f(\bx) = c\(\frac{\lambda}{\lambda^2 + |\bx-\bx_0|^2}\)^{\frac{n-2\ga}{2}} \quad \text{for all } \bx \in \R^n.
\end{equation}
\end{theorem}
\begin{remark}
We have three remarks for Theorem \ref{thm:main11}.

\medskip \noindent
(1) One can further extend Theorem \ref{thm:main11} by treating $f \in L^p(\R^n)$ with any $p \in (1,\infty)$
and the case when $\ga \in (-\frac{n}{2p},1)$. See Lemma \ref{lemma:Poiest1} and Theorem \ref{thm:main11g} below.

\medskip \noindent
(2) Similar results to Theorem \ref{thm:main11} were established by Chen \cite[Remark 2]{Ch},
who considered the map $f \in L^{\frac{2n}{n-2\ga}}(\R^n) \to \mck_{n,\ga}f \in L^{\frac{2(n+1)}{n-2\ga}}(\R^N_+)$ for $n \in \N$ and $\ga \in (0,1)$.

In fact, inequalities \eqref{eq:main11} and \eqref{eq:Poiest11} below are already known and their abstract generalization is available in \cite[Theorem 3.1]{LSHZ}. In Lemma \ref{lemma:Poiest1}, we will provide a direct and quick proof. The classification result is new.

\medskip \noindent
(3) The choice of the norms in \eqref{eq:main11} are natural in view of the weighted Sobolev inequality
\begin{equation}\label{eq:Sob}
\|U\|_{L^{\frac{2(n-2\ga+2)}{n-2\ga}}(\R^N_+;x_N^{\mz})} \le C\|U\|_{\dot{H}^{1,2}(\R^N_+;x_N^{\mz})} \quad \textup{for all } U \in \dot{H}^{1,2}\(\R^N_+;x_N^{\mz}\)
\end{equation}
and the weighted Sobolev trace inequality
\begin{equation}\label{eq:Sob2}
\|u\|_{L^{\frac{2n}{n-2\ga}}\(\R^n\)} \le C\|U\|_{\dot{H}^{1,2}(\R^N_+;x_N^{\mz})} \quad \textup{for all } U \in \dot{H}^{1,2}\(\R^N_+;x_N^{\mz}\),\, u = U|_{\R^n}
\end{equation}
for some $C > 0$ depending only on $n$ and $\ga$. As verified in Lemma \ref{lemma:embedSob}, inequality \eqref{eq:Sob} is valid for $\ga \in (0,\frac{1}{2}]$.
However, there is no universal constant $C$ for \eqref{eq:Sob}, and only a weaker inequality holds for $\ga \in (\frac{1}{2},1)$.
In contrast, inequality \eqref{eq:Sob2} is valid for all $\ga \in (0,1)$, as shown by Xiao \cite[Theorem 1.1]{Xi}.

A notable fact is that the best constant for \eqref{eq:Sob2} with $\ga \in (0,1)$ is achieved by $u = f$ in \eqref{eq:bubble} and $U = \mck_{n,\ga}f$ (see \cite[Theorem 1.1]{Xi}),
whilst the best constant for \eqref{eq:Sob} with $\ga \in (0,\frac{1}{2}]$ is attained by
\[U(x) = c \(\frac{\lambda}{\lambda^2 + |x-(\bx_0,0)|^2}\)^{\frac{n-2\ga}{2}} \quad \textup{for all } x \in \R^N_+\]
where $c \in \R$, $\lambda > 0$, and $\bx_0 \in \R^n$ (see \cite[Theorem 1.3]{CR} and \cite[Theorem 1.6]{DSWZ}). This function is certainly not $\ga$-harmonic on $\R^N_+$.
\end{remark}

We next transform Theorem \ref{thm:main11} to a result for functions on the unit sphere $\S^n = \pa\B^N$, by employing the M\"obius transformation $\phi_\m: \R^N_+ \to \B^N$ defined in \eqref{eq:Mobius}, along with additional conformal geometric techniques and the use of spherical harmonics.
We also find an equation in $\B^N$ corresponding to \eqref{eq:degeq}--\eqref{eq:degeq2}.

For future use, we introduce the canonical hyperbolic metric on $\B^N$,
\begin{equation}\label{eq:gpb}
g_\pb^+(y) = \frac{4\, |dy|^2}{(1-|y|^2)^2} \quad \textup{for } y \in \B^N.
\end{equation}
In addition, let
\begin{equation}\label{eq:rhobgpb}
\rho_\pb(y) = \frac{1-|y|}{1+|y|} \quad \textup{and} \quad \bg_\pb(y) = \frac{4\, |dy|^2}{(1+|y|)^4} \quad \textup{for } y \in \overline{\B^N}.
\end{equation}
Denoting the standard metric on $\S^n$ by $g_{\S^n}$, we define the metric $\bh_\pb = \bg_\pb|_{T\S^n} = \frac{1}{4} g_{\S^n}$ on $\S^n$ and the Riemannian volume form $dv_{\bh_\pb} = \frac{1}{2^n} dv_{g_{\S^n}}$ on $(\S^n,\bh_\pb)$.
We also express $dv_{\bg_\pb}(y) = \frac{2^N dy}{(1+|y|)^{2N}}$ for $y \in \overline{\B^N}$.
In Definition \ref{def:fs2}, we set function spaces appearing in the following corollary.
\begin{cor}\label{cor:main12}
Assume that $n \in \N$, $n > 2\ga$, $\ga \in (0,1)$, and $\tf \in C^{\infty}(\S^n)$. Let also
\begin{equation}\label{eq:wmcp}
\(\wmck_{n,\ga}\tf\)(y) = \ka_{n,\ga} (1+|y|)^{n-2\ga} \int_{\S^n} \frac{(1-|y|^2)^{2\ga}}{|y-\zeta|^{n+2\ga}} \tf(\zeta) (dv_{\bh_\pb})_{\zeta} \quad \text{for } y \in \B^N
\end{equation}
where $\ka_{n,\ga} > 0$ is the constant in \eqref{eq:gaharext2} and the subscript $\zeta$ signifies the variable of integration.

\medskip \noindent \textup{(a)} If $C_{n,\ga} > 0$ is the optimal constant of \eqref{eq:main11}, then
\begin{equation}\label{eq:main12}
\left\|\wmck_{n,\ga} \tf\right\|_{L^{\frac{2(n-2\ga+2)}{n-2\ga}}(\B^N;\rho_\pb^{\mz},\bg_\pb)} \le C_{n,\ga} \big\|\tf\big\|_{L^{\frac{2n}{n-2\ga}}(\S^n,\bh_\pb)} \quad \text{for } f \in L^{\frac{2n}{n-2\ga}}(\S^n,\bh_\pb).
\end{equation}
Furthermore, the equality holds if and only if there are numbers $c \in \R$, $\lambda > 0$, and $\zeta_0 = (\bar{\zeta_0},\zeta_{0N}) \in \S^n$ such that
\[\tf(y) = c\left[\frac{2\lambda(1+\zeta_{0N})}{\lambda^2(1+\zeta_{0N})(1+y_N)+2(1-\zeta_0 \cdot y)}\right]^{\frac{n-2\ga}{2}} \quad \text{for all } y = (\bar{y},y_N) \in \S^n.\]
When $\zeta_0 = (0,1) = \phi_\m(0) \in \S^n$ and $\lambda = 1$, the function $\tf$ is reduced to $\tf(y) = 1$ for all $y \in \S^n$.

\medskip \noindent \textup{(b)} The function $\wtv = \wmck_{n,\ga}\tf$ is the unique solution to
\begin{equation}\label{eq:main12eq}
\begin{cases}
-\textup{div}_{\bg_\pb}\(\rho_\pb^{\mz}(y) \nabla_{\bg_\pb} \wtv(y)\) + \dfrac{n(n-2\ga)}{4} \dfrac{(1+|y|)^2}{|y|} \rho_\pb^{\mz}(y) \wtv(y) = 0 &\text{for } y \in \B^N, \\
\wtv = \tf &\text{on } \S^n, \\
\wtv \in H^{1,2}\(\B^N;\rho_\pb^{\mz},\bg_\pb\).
\end{cases}
\end{equation}
In particular, $\wtv$ and $\rho_\pb^{\mz}\frac{\pa \wtv}{\pa \rho_\pb}$ are H\"older continuous on $\overline{\B^N}$.
\end{cor}
\begin{remark}
We have two remarks for Corollary \ref{cor:main12}.

\medskip \noindent
(1) Assume that $\ga = \frac{1}{2}$. If we define $\wtv = \wmck_{n,1/2}\tf$ and
\[\whv(y) = \frac{2^{N-2}}{(1+|y|)^{N-2}}\wtv(y)
= \frac{1}{\left|\S^n\right|} \int_{\S^n} \frac{1-|y|^2}{|y-\zeta|^N} \tf(\zeta) (dv_{g_{\S^n}})_{\zeta} \quad \textup{for } y \in \overline{\B^N}\]
where $|\S^n|$ is the surface area of the unit sphere $\S^n$, then \eqref{eq:main12eq} reads
\[\begin{cases}
-\(\dfrac{1+|y|}{2}\)^{N-2} \Delta \whv(y) = -\Delta_{\bg_\pb} \wtv(y) + \dfrac{n(n-1)}{4} \dfrac{(1+|y|)^2}{|y|} \wtv(y) = 0 &\textup{for } y \in \B^N,\\
\whv = \tf &\textup{on } \S^n.
\end{cases}\]
Recall that $\frac{1-|y|^2}{|\S^n||y-\zeta|^N}$ is the Poisson kernel on $\B^N$.

\medskip \noindent
(2) The reader may want to compare Corollary \ref{cor:main12}(a) with its variants in \cite[Theorem 1]{Ch} and \cite[Theorem 2]{Ti}.
It may also be worthwhile to compare Corollary \ref{cor:main12}(b) with \cite[Propositions 1--3]{Ti}.
Notably, Tian in \cite{Ti} used the defining function $\rho_0(y) := \frac{1-|y|^2}{2}$ and the Euclidean metric $g_{\R^N}$, whereas we employ $\rho_\pb$ and $\bg_\pb$.
The choice of $\rho_0$ by Tian can be supported by the work of Ache and Chang \cite{AC}.
In contrast, our selection of $\rho_\pb$ can be justified by well-known facts that $\rho_\pb$ is a geodesic defining function of $(\S^n,\bh_\pb)$ (see \eqref{eq:rhopbgeo})
and the value of $-\log \rho_\pb(y)$ is the hyperbolic distance from $0$ to $y \in \B^N$.
Also, the following discussions concerning Theorem \ref{thm:main21} further validate our selection.
\end{remark}

\begin{conv}\label{conv}
In the sequel, unless stated otherwise, we always assume that $n = N-1 \in \N$, $\ga \in (0,1)$, $\mz = 1-2\ga \in (-1,1)$, $(X^N,g^+)$ is an $N$-dimensional CCE manifold, $(M^n,[\bh])$ is its conformal infinity,
$\rho$ is a smooth defining function of $M$, and $\bg = \rho^2g^+$ is a smooth metric on $\ox$ such that $(\ox,\bg)$ is compact and $\bh = \bg|_{TM}$.
These notions induce a smooth metric measure space $(\ox,\bg,\rho^{\mz}dv_{\bg},-2\ga)$ where $dv_{\bg}$ is the Riemannian volume form on $(\ox,\bg)$.
Let $L^{\mz}_{\bg,\rho}$ be the weighted conformal Laplacian on $(\ox,\bg,\rho^{\mz}dv_{\bg},-2\ga)$, $\rhog$ a geodesic defining function of $(M,\bh)$, and $\bgg = \rhog^2g^+$; see Subsections \ref{subsec:fcl} and \ref{subsec:smms} for terminologies.
\end{conv}

For any fixed metric $\th \in [\bh]$, let $\trh_\geo$ be a geodesic defining function of $(M,\th)$. We define the admissible class $\Xi_{g^+,\th}$ of defining functions of $M$ by
\begin{multline}\label{eq:Xi}
\Xi_{g^+,\th} = \left\{\trh: \trh \textup{ is a defining function of } M \textup{ such that } \trh,\, \nabla\trh \in L^{\infty}_{\text{loc}}(X) \right. \\
\left. \trh = \trh_\geo \left[1 + \Phi \trh_\geo^{2\ga} + o'\(\trh_\geo^{2\ga}\)\right] \textup{ for } \trh_\geo > 0 \textup{ small and } \Phi \in C^{\infty}(M)\right\}.\ \footnotemark
\end{multline}
\footnotetext{This class naturally appears in the study of fractional conformal Laplacians.
It can be suitably extended for the study of fractional GJMS operators with $\ga \in (1,\frac{n}{2}) \setminus \N$. See \cite{CC, Ca2}.}Here, $o'(\trh_\geo^{2\ga})$
denotes a function on $\ox$ such that both the function and its tangential derivatives of any order are bounded by $C_{\trh_\geo} \trh_\geo^{2\ga}$ where $C_{\trh_\geo} \to 0$ as $\trh_\geo \to 0$.
Moreover, its normal derivative is bounded by $C\trh_\geo^{2\ga-1}$ for some $C > 0$.
Let $\tg = \trh^2g^+$. Then, $(\ox,\tg)$ is compact and $\tg|_{TM} = \th$. Given any $\trh \in \Xi_{g^+,\th}$,
we set the total weighted volume and the weighted isoperimetric ratio of the smooth metric measure space $(\ox,\tg,\trh^{\mz}dv_{\tg},-2\ga)$ as
\[\textup{Volume}_{\tg,\trh}\(\ox\) = \int_{\ox} \trh^{\mz} dv_{\tg} \quad \textup{and} \quad
I_{n,\ga}\(\ox,\tg,\trh\) = \frac{\textup{Volume}_{\tg,\trh}\(\ox\)}{\textup{Area}_{\th}(M)^{\frac{n-2\ga+2}{n}}},\]
respectively.\footnote{Let $\trh^*$ be the adapted defining function of $(M,\th)$ (see Subsection \ref{subsec:ext}), and $\ep > 0$ a sufficiently small number.
The integral $\int_{\{\trh > \ep\}} (\trh^*)^{n/2-\ga} dv_{g^+}$ provides a different type of the total weighted volume known as the renormalized weighted volume \cite{Go2}.
The leading and subsequent order terms of its expansion in $\ep$ are determined by $\textup{Area}_{\th}(M)$ and the total fractional scalar curvature $\int_M Q^{\ga}_{g^+,\th} dv_{\th}$, respectively.}
We consider a variational problem involving the ratio $I_{n,\ga}(\ox,\tg,\trh)$ over weighted scalar-flat conformal classes:
\begin{equation}\label{eq:thetanw}
\Theta_{n,\ga}\(X,g^+,[\bh]\) = \sup\left\{I_{n,\ga}(\ox,\tg,\trh): \tg \in \mcb_{g^+,\,[\bh]}\right\}
\end{equation}
with
\begin{equation}\label{eq:mcb}
\mcb_{g^+,\,[\bh]} := \left\{\tg = \trh^2g^+: \th \in [\bh],\, \trh \in \Xi_{g^+,\th}, R^{\mz}_{\tg,\trh} = 0 \textup{ in } X\right\}.
\end{equation}
Here, $R^{\mz}_{\tg,\trh}$ is the weighted scalar curvature on $(\ox,\tg,\trh^{\mz}dv_{\tg},-2\ga)$ defined in Subsection \ref{subsec:smms}.

As the following theorem depicts, we have an extension of \eqref{eq:hwy20}. Let $\lambda_1(-\Delta_{g^+})$ be the bottom of the $L^2$-spectrum of the Laplace-Beltrami operator $-\Delta_{g^+}$.
\begin{theorem}\label{thm:main21}
Assume that $n \in \N$, $n > 2\ga$, and $\ga \in (0,1)$. If $\lambda_1(-\Delta_{g^+}) > \frac{n^2}{4}-\ga^2$, then
\begin{equation}\label{eq:thetaest}
\infty > \Theta_{n,\ga}\(X,g^+,[\bh]\) \ge \Theta_{n,\ga}\(\B^N,g_\pb^+,[g_{\S^n}]\)
\end{equation}
and $\Theta_{n,\ga}(X,g^+,[\bh])$ is attained by a metric in $\mcb_{g^+,\,[\bh]}$ provided
\begin{equation}\label{eq:thetaest2}
\Theta_{n,\ga}(X,g^+,[\bh]) > \Theta_{n,\ga}(\B^N,g_\pb^+,[g_{\S^n}]).
\end{equation}
\end{theorem}
\begin{remark}
We present three remarks concerning Theorem \ref{thm:main21}. An additional remark will be given in Remark \ref{remark:var}.

\medskip \noindent
(1) Theorem \ref{thm:main21} with $\ga = \frac{1}{2}$ corresponds to \cite[Theorem 1.1]{HWY2} in the analytical sense.
However, to settle Theorem \ref{thm:main21} in a geometrically robust framework, we must introduce a spectral condition involving a CCE metric $g^+$ in $X$ in place of $\lambda_1(L_{\bg}) > 0$ in \cite{HWY2}.
A natural candidate turns out to be $\lambda_1(-\Delta_{g^+}) > \frac{n^2}{4}-\ga^2$, because it guarantees $\mcb_{g^+,\,[\bh]} \ne \emptyset$ as shown in \cite[Lemma 6.1]{CC}.
Refer to Subsection \ref{subsec:ext} for more explanation.

Note that a conformal change of a metric on $\ox$ was allowed in \cite{HWY2}, whereas we are constrained to perform a conformal change of a metric only on its boundary $M$.
This geometric restriction compels us to carefully select the admissible class $\Xi_{g^+,\th}$ of defining functions of $M$ and the admissible class $\mcb_{g^+,\,[\bh]}$ of metrics in $X$
(as given in \eqref{eq:Xi} and \eqref{eq:mcb}, respectively)
so that the variational problem \eqref{eq:thetanw} can be reformulated as a sharp inequality involving the weighted Poisson kernel $\mck^{\mz}_{\bg,\rho}$; refer to Definition \ref{def:wPk} and Subsection \ref{subsec:wir}.

\medskip \noindent
(2) Similar to Theorem \ref{thm:main11}, we will investigate a more general family \eqref{eq:thetanwg} of variational problems than \eqref{eq:thetanw}. The above result can be obtained as a special case $p = \frac{2n}{n-2\ga}$.
Refer to Theorem \ref{thm:main21g} and Proposition \ref{prop:reg} below.

\medskip \noindent
(3) We have opted to work on CCE manifolds in Theorem \ref{thm:main21} to apply the results in \cite{GZ, CC} directly. The description of the results will be given in the subsequent paragraphs.
We expect that our arguments, combined with the idea in \cite[Section 5]{CG}, will apply almost verbatim to general asymptotically hyperbolic manifolds, provided their conformal infinities are minimal for $\ga \in (\frac{1}{2},1)$.
\end{remark}

In \eqref{eq:wir2}, we will prove that $\Theta_{n,\ga}(\B^N,g_\pb^+,[g_{\S^n}]) = C_{n,\ga}^{\frac{2(n-2\ga+2)}{n-2\ga}}$ where $C_{n,\ga}$ is the constant in \eqref{eq:main11} and \eqref{eq:main12},
which manifests the linkage between problem \eqref{eq:thetanw} and inequalities \eqref{eq:main11} and \eqref{eq:main12}.
Since $\mck_{n,\ga}f$ and $\wmck_{n,\ga}\tf$ solve equations \eqref{eq:degeq} and \eqref{eq:main12eq}, respectively, one may ask if there is an equation in $X$ naturally associated with \eqref{eq:thetanw}.
To answer this question, we need to recall the Chang-Gonz\'alez extension theorem \cite{CG, CC}; see Proposition \ref{prop:ext}.

Suppose that $(X,g^+)$ is a CCE manifold with conformal infinity $(M,[\bh])$, $\ga \in (0,\frac{n}{2})$, and $\frac{n^2}{4}-\ga^2$ is not a pure point spectrum of $-\Delta_{g^+}$ (i.e., $\frac{n^2}{4}-\ga^2 \notin \sigma_{\pp}(-\Delta_{g^+})$).
In \cite{GZ}, Graham and Zworski proved the existence of a conformally covariant pseudo-differential operator $P^{\ga}_{g^+,\bh}$ on $M$ whose principal symbol is that of the fractional Laplacian $(-\Delta)^{\ga}$.
It is called the fractional GJMS operator (or the fractional conformal Laplacian for $\ga \in (0,1)$); see Subsection \ref{subsec:fcl} for more explanation.

The analysis of properties of solutions to equations involving the operator $P^{\ga}_{g^+,\bh}$ is challenging due to its abstractness and nonlocality.
The Chang-Gonz\'alez extension provides an alternative description of $P^{\ga}_{g^+,\bh}$ for $n > 2\ga$ and $\ga \in (0,1)$, which makes the analysis more tractable: Choose $f \in C^{\infty}(M)$. According to Proposition \ref{prop:ext2}, the boundary value problem
\begin{equation}\label{eq:degeq3}
\begin{cases}
L^{\mz}_{\bgg,\,\rhog} U = 0 &\textup{in } X,\\
U = f &\textup{on } M
\end{cases}
\end{equation}
has the unique solution $U_0$ in $H^{1,2}(X;\rho^{\mz},\bg)$ provided $\frac{n^2}{4}-\ga^2 \notin \sigma_{\pp}(-\Delta_{g^+})$. Then it holds that
\begin{equation}\label{eq:degeq4}
P^{\ga}_{g^+,\bh} f = \frac{d_{\ga}}{2\ga} \lim_{\rhog \to 0} \rhog^{\mz}{\frac{\pa U_0}{\pa \rhog}} \quad \textup{on } M
\end{equation}
where $d_{\ga}$ is the constant in \eqref{eq:degeq2}. As shown in Lemma \ref{lemma:wir}, the value of $\Theta_{n,\ga}(X,g^+,[\bh])$ equals
the supremum of integrals $\int_X \rhog^{\mz} U_0^{\frac{2(n-2\ga+2)}{n-2\ga}} dv_{\bgg}$ over the set of $L^{\frac{2n}{n-2\ga}}$-normalized positive functions $f$ in $C^{\infty}(M)$.
This fact is crucial in deriving \eqref{eq:wir2}.

In addition, if $(X,g^+)$ is the Poincar\'e half-space model $(\R^N_+,g_\ph^+)$ with the hyperbolic metric
\[g_\ph^+(x) := \frac{|dx|^2}{x_N^2} = \frac{|d\bx|^2 \oplus dx_N^2}{x_N^2} \quad \textup{for } x=(\bx,x_N) \in \R^N_+,\]
then \eqref{eq:degeq3}--\eqref{eq:degeq4} is reduced to the Caffarelli-Silvestre extension \eqref{eq:degeq}--\eqref{eq:degeq2}, which returns us to the setting of Theorem \ref{thm:main11}.

\begin{remark}\label{remark:var}
Assume that $(\ox,\bg)$ is a smooth compact Riemannian manifold with boundary $M$. Let $\mck_{\bg}$ be the Poisson kernel for the conformal Laplacian $L_{\bg}$ on $(\ox,\bg)$.
The following correspondence was originally noted by Jin and Xiong \cite{JX}: Let $\bh = \bg|_{TM}$.
\bgroup
\def\arraystretch{1.2}
\begin{table}[H]
\small
\begin{tabular}{|c|c|}
\hline
Variational problems & Corresponding maps \\ \hhline{|=|=|}
Problem \eqref{eq:thetan} & $\mck_{\bg}: L^{\frac{2(N-1)}{N-2}}(M,\bh) \to L^{\frac{2N}{N-2}}(X,\bg)$ \\ \hline
Yamabe problem \cite{Che, Es} & Sobolev embedding: $H^{1,2}(X,\bg) \to L^{\frac{2N}{N-2}}(X,\bg)$ \\ \hline
\begin{tabular} {@{}c@{}} Boundary Yamabe problem \cite{Che, Es2} \\ (= Riemann mapping problem of Escobar) \end{tabular}
&\begin{tabular} {@{}c@{}} Sobolev trace embedding: \\ $H^{1,2}(X,\bg) \to L^{\frac{2(N-1)}{N-2}}(M,\bh)$ \end{tabular}
\\ \hline
\end{tabular}
\end{table}
\egroup

This viewpoint still makes sense under our setting, at least for $\ga \in (0,\frac{1}{2}]$:
Let $\mck^{\mz}_{\bg,\rho}$ be the weighted Poisson kernel for the weighted conformal Laplacian $L^{\mz}_{\bg,\rho}$, which will be studied in Subsection \ref{subsec:wpk}.
\bgroup
\def\arraystretch{1.2}
\begin{table}[H]
\small
\begin{tabular}{|c|c|}
\hline
Variational problems & Corresponding maps \\ \hhline{|=|=|}
Problem \eqref{eq:thetanw} & $\mck_{\bg,\rho}^{\mz}: L^{\frac{2n}{n-2\ga}}(M,\bh) \to L^{\frac{2(n-2\ga+2)}{n-2\ga}}(X;\rho^{\mz},\bg)$ \\ \hline
\begin{tabular} {@{}c@{}} Yamabe problem \\ on smooth measure metric spaces\footnotemark \end{tabular}
& \begin{tabular} {@{}c@{}} weighted Sobolev embedding: \\ $H^{1,2}(X;\rho^{\mz},\bg) \to L^{\frac{2(n-2\ga+2)}{n-2\ga}}(X;\rho^{\mz},\bg)$ \end{tabular} \\ \hline
\begin{tabular} {@{}c@{}} Chang-Gonz\'alez extension \\ of the fractional Yamabe problem \cite{GQ, GW, KMW, DSV, MN2} \end{tabular}
& \begin{tabular} {@{}c@{}} weighted Sobolev trace embedding: \\ $H^{1,2}(X;\rho^{\mz},\bg) \to L^{\frac{2n}{n-2\ga}}(M,\bh)$ \end{tabular} \\ \hline
\end{tabular}
\end{table}
\egroup
\footnotetext{Case \cite{Ca3} introduced a Yamabe-type problem on compact smooth measure metric spaces without boundary.
Posso \cite{Po} generalized it to the non-empty boundary case. They studied when the characteristic constant $\mu$ (in Subsection \ref{subsec:smms}) is $0$, while in our case $\mu = -2\ga$.
As far as the authors know, the Yamabe-type problem has yet to be treated under our setting.}
\end{remark}

Before closing this subsection, we propose the following conjecture.
\begin{conj}
If $n > 2\ga$ and $(\ox,\bg)$ is not conformally diffeomorphic to $(\B^N,g_{\R^N})$, then \eqref{eq:thetaest2} holds. Consequently, the variational problem $\Theta_{n,\ga}(X,g^+,[\bh])$ is attainable.
\end{conj}
\noindent The above conjecture is inspired by \cite[Conjecture 1.1]{HWY2}, which treats the unweighted case $\ga = \frac{1}{2}$.
To tackle the conjecture, a delicate qualitative analysis of the quotient $I_{n,\ga}(\ox,\tg,\trh)$ (or equivalently, the rightmost side of \eqref{eq:wir0}) is necessary.
In the existence results \cite{JX, CJR, GlZ} for $\ga = \frac{1}{2}$, the authors crucially used the fact that $W_{1,0}$ is a rational function.
It seems that $W_{1,0}$ cannot be expressed explicitly in terms of elementary functions for $\ga \in (0,1) \setminus \{\frac{1}{2}\}$ (refer to Footnote \ref{fnlabel}),
so a new approach for evaluating integrals that involve $W_{1,0}$ has to be devised. This technical issue adds further complexity to the conjecture, which we hope to address in future research.

\subsection{Organization of the paper}
We provide an overview of the paper, highlighting the novelty and interesting aspects.

\medskip
In Section \ref{sec:back}, we review the necessary geometric and analytic backgrounds and introduce the functional spaces to work with.
Then, we establish the unique solvability and regularity results of degenerate elliptic equations (see Propositions \ref{prop:ext2} and \ref{prop:ext3}), using only a finite energy condition instead of conventional pointwise conditions.
To achieve them, we largely depend on the equivalence between two degenerate equations \eqref{eq:degeq31} and \eqref{eq:degeq32}, the construction of a formal solution to \eqref{eq:degeq32}, and an approximation argument with elliptic regularity theory.
These findings are noteworthy in their own right, because they are deduced under an analytically natural framework that covers a broader scope than what is typically treated in conformal geometry.

\medskip
In Section \ref{sec:weight}, we first construct the weighted Poisson kernel $\mck^{\mz}_{\bg,\rho}$ of the weighted conformal Laplacian $L^{\mz}_{\bg,\rho}$ on a general CCE manifold $(\ox,\bg,\rho^{\mz}dv_{\bg},-2\ga)$ and derive various qualitative estimates; refer to Proposition \ref{prop:Poi}.
This process is more complicated than one might expect because of the degeneracy of the differential operator $L^{\mz}_{\bg,\rho}$.
For example, obtaining pointwise upper bounds for the weighted Poisson kernel $\mck^{\mz}_{\bg,\rho}$ and its derivatives is not straightforward for $\ga \in (0,1) \setminus \{\frac{1}{2}\}$, in contrast to \cite[Lemma 2.2]{KMW2} that handles the case $\ga = \frac{1}{2}$.
During the proof, we must carefully treat the non-integer exponents $2\ga$ when $\ga \in (0,1) \setminus \{\frac{1}{2}\}$, which appear in equation \eqref{eq:Poibg} and the main-order term \eqref{eq:Poi02} of $\mck^{\mz}_{\bg,\rho}$,
and the logarithmic singularities of the underlying metric $\bgg$ on $\ox$ that occur when $n \ge 4$ is even.

Subsequently, we reformulate the variational problem \eqref{eq:thetanw} by referring to the degenerate elliptic equations appearing in the extension theorem and their associated weighted Poisson integrals, as depicted in Lemma \ref{lemma:wir}.
This standpoint enables us to regard inequalities \eqref{eq:main11} and \eqref{eq:main12} as specific examples of \eqref{eq:thetanw}.

\medskip
In Sections \ref{sec:main11}--\ref{sec:main2}, we prove the main theorems; Theorem \ref{thm:main11}, Corollary \ref{cor:main12}, and Theorem \ref{thm:main21}.
The proof of Theorems \ref{thm:main11} and \ref{thm:main21} is primarily based on the arguments in Hang et al. \cite{HWY, HWY2}.
During the proof, we develop a regularity theory for integral equations (given in \eqref{eq:main11g2} and \eqref{eq:main21g2}) involving the weighted Poisson kernels $\mck^{\mz}_{\bg,\rho}$,
introducing the concept ``tangential Schauder estimates" in the proof of Lemma \ref{lemma:reg3}.
Our approach, based on the qualitative estimates of $\mck^{\mz}_{\bg,\rho}$ and potential analysis, is flexible and can be easily adapted to other relevant contexts.

On the other hand, the proof of Corollary \ref{cor:main12} crucially uses the fact that $(\R^N_+,|dx|^2,x_N^{\mz}dx,-2\ga)$
and $(\B^N,\bg_\pb,\rho_\pb^{\mz}dv_{\bg_\pb},-2\ga)$ are pointwise conformally equivalent smooth metric measure spaces, shown in Lemma \ref{lemma:smms}.
We also need the pointwise boundary behavior of the function $\wtv = \wmck_{n,\ga}\tf$ and its weighted normal derivative (refer to Lemma \ref{lemma:smms5}).
We achieve them with the help of the theory of spherical harmonics and the concrete expression of $P^{\ga}_{g_\pb^+,\bh_\pb} \tf$ presented in \eqref{eq:smms56}.

\medskip
Appendix \ref{sec:app} presents several auxiliary results along with their proofs. These results encompass weighted Sobolev embedding theorems,
the regularity theory for degenerate elliptic equations with Dirichlet boundary conditions, the theory of weighted spherical harmonics, and expansions of the squares of Riemannian distances.
Many of these are lesser-known or entirely new results that may be useful in other literature.

\section{Background}\label{sec:back}
\subsection{Fractional conformal Laplacians}\label{subsec:fcl}
We recall the definition of the fractional conformal Laplacians on the conformal infinity of a CCE manifold. The main references are \cite{GZ, CDLS, An, CG, FG, CC}.

\medskip
Let $N = n+1$ and $X^N$ be an $N$-dimensional smooth (i.e., infinitely differentiable) manifold with smooth boundary $M^n$.
A function $\rho$ in $X$ is called a defining function of $M$ if $\rho > 0$ in $X$, $\rho = 0$ on $M$, and $d\rho \ne 0$ on $M$.
We say that a metric $g^+$ in $X$ is conformally compact if there is a defining function $\rho \in C^{\infty}(\ox)$ of $M$
such that the conformal metric $\bg = \rho^2g^+$ extends smoothly to the closure $\ox$ of $X$ and $(\ox,\bg)$ is compact.
Because the defining function $\rho$ of $M$ is determined up to a multiplication by a smooth positive function on $\ox$,
only the conformal class $[\bh]$ with $\bh = \bg|_{TM}$ is well-defined on $M$. We call $(M,[\bh])$ the conformal infinity of $(X,g^+)$.

A Riemannian manifold $(X,g^+)$ is called asymptotically hyperbolic if $g^+$ is conformally compact and $|d\rho|_{\bg} \to 1$ as $\rho \to 0$.
The latter condition is equivalent to $R_{g^+} \to -n(n+1)$ as $\rho \to 0$ where $R_{g^+}$ is the scalar curvature on $(X,g^+)$.
Moreover, if $(X,g^+)$ is conformally compact and its Ricci curvature tensor $\textup{Ric}_{g^+}$ satisfies $\textup{Ric}_{g^+} = -ng^+$ in $X$,
then it is called conformally compact Einstein (CCE) or Poincar\'e-Einstein. All CCE manifolds are asymptotically hyperbolic.

\medskip
Suppose that $(X,g^+)$ is a CCE manifold and $\bh$ is a representative of the conformal class $[\bh]$ on $M$.
Then there exists a unique defining function $\rhog$ of $M$, called the geodesic defining function of $(M,\bh)$,
such that $|d\rhog|_{\rhog^2g^+} = 1$ in a collar neighborhood $M \times [0,\delta_0]$ of $M$ for small $\delta_0 > 0$.\footnote{One
can suitably extend $\rhog$ to become a positive smooth function in $\ox$. For the rest of the paper, we will assume the smoothness of $\rhog$.} In particular, the metric $g^+$ takes the form
\begin{equation}\label{eq:g+ ext}
g^+ = \rhog^{-2}\bgg = \rhog^{-2}\(h_{\rhog} \oplus (d\rhog)^2\) \quad \textup{on } \mct_{\delta_0} := M \times (0,\delta_0)
\end{equation}
where $h_{\rhog}$ is a one-parameter family of smooth metrics on $M$ such that $h_0 = \bh$.
It is known that $\text{Trace}_{h_{\rhog}} \pa_{\rhog} h_{\rhog} \in \rhog C^{\infty}(\ox)$ and
the asymptotic expansion of $h_{\rhog}$ contains solely even powers of $\rhog$ up to order $n-1$.
When $n \ge 4$ is even, the next-order term is $h_* \rhog^n |\log\rhog|$ and higher-order terms look like
$h_{**} \rhog^{n+l} |\log \rhog|^m$ for $l \in \N$ and $m \in \N \cup \{0\}$, where $h_*$ and $h_{**}$ represent smooth two-tensors on $M$.
When $n$ is odd or $2$, the expansion of $h_{\rhog}$ has no log-order term and $h_{\rhog}$ is smooth on $\overline{\mct_{\delta_0}}$; refer to \cite{CDLS, An}.
A notable consequence is that $(M,\bh)$ is a totally geodesic submanifold of $(\ox,\bg)$.

In addition, if $\frac{n^2}{4}-\ga^2 \notin \sigma_{\pp}(-\Delta_{g^+})$, $f \in C^{\infty}(M)$, and $s = \frac{n}{2}+\ga$ with $\ga \in (0,\frac{n}{2}) \setminus \N$,
the eigenvalue problem $-\Delta_{g^+}u - s(n-s)u = 0$ in $X$ possesses a unique solution of the form $u = F\rhog^{n-s} + G\rhog^s$.
Here, $F \in C^{\infty}(\ox)$ if $n$ is odd or $2$, $F \in C^{n+1,\beta}(\ox)$ for any $\beta \in (0,1)$ if $n \ge 4$ is even, $G \in C^{\infty}(\ox)$, $F|_M = f$, and $\pa_{\rhog} F|_M = \pa_{\rhog} G|_M = 0$; cf. the proof of Proposition \ref{prop:ext3}.
As proved in \cite{GZ}, if one sets $P^{\ga}_{g^+,\bh}f = d_{\ga} G|_M$, then it becomes a formally self-adjoint pseudo-differential operator on $M$
with a principal symbol identical to that of $(-\Delta)^{\ga}$. It also satisfies the conformal covariance property
\begin{equation}\label{eq:fcl}
P^{\ga}_{g^+, \vph^{\frac{4}{n-2\ga}}\bh}f = \vph^{-{\frac{n+2\ga}{n-2\ga}}} P^{\ga}_{g^+, \bh}(\vph f) \quad \textup{for any positive smooth function } \vph \textup{ on } M,
\end{equation}
which led to it being referred to as the fractional GJMS (Paneitz) operator for all $\ga \in (0,\frac{n}{2}) \setminus \N$, and the fractional conformal Laplacian for $\ga \in (0,1)$.

\subsection{Smooth metric measure spaces}\label{subsec:smms}
We briefly describe some basic features of smooth metric measure spaces.
For more details, including their geometric significance, the reader is advised to consult \cite{Ca, Ca3, CC, Ca2}.

\medskip
A smooth metric measure space is a four-tuple $(\mcx^{n+1},g,v^m dv_g,\mu)$ determined by an $(n+1)$-dimensional Riemannian manifold $(\mcx,g)$ with or without boundary,
the Riemannian volume form $dv_g$ on $(\mcx,g)$, a smooth positive function $v$ in $\textup{Int}(\mcx)$, $m \in \R \setminus \{-n\}$, and $\mu \in \R$.
If the boundary $\pa \mcx$ of $\mcx$ is nonempty, we further assume that $v$ is a sufficiently regular nonnegative function on $\mcx$ with $v^{-1}(0) = \pa\mcx$.

Let $R_g$ be the scalar curvature, $\Delta_g$ the Laplace-Beltrami operator, $\nabla_g$ the gradient, $\la \cdot, \cdot \ra_g$ the metric, and $|\cdot|_g$ the norm on $(\mcx,g)$.
On a smooth metric measure space $(\mcx,g,v^m dv_g,\mu)$, we define the weighted Laplacian $\Delta_{g,v}$
by $\Delta_{g,v}u = \Delta_g u + \la \nabla_g (\log v^m), \nabla_g u \ra_g$ for a function $u$ on $\mcx$, the weighted scalar curvature $R^m_{g,v}$ by
\begin{equation}\label{eq:wsc}
R^m_{g,v} = R_g - 2mv^{-1}\Delta_g v - m(m-1)v^{-2}|\nabla_g v|_g^2 + m\mu v^{-2},\ \footnotemark
\end{equation}
\footnotetext{It is a scalar curvature associated with the Bakry-\'Emery Ricci tensor.}and
the weighted conformal Laplacian $L^m_{g,v}$ by
\begin{equation}\label{eq:wcl}
L^m_{g,v} = -\Delta_{g,v} + \frac{m+n-1}{4(m+n)} R^m_{g,v}.
\end{equation}
If $m = 0$, then $L^m_{g,v}$ is reduced to the conformal Laplacian $L_g$ on $(\mcx^N,g)$.

We say that $(\mcx,g,v^m dv_g,\mu)$ and $(\mcx,\hg,\hv^m dv_{\hg},\mu)$ are pointwise conformally equivalent if there exists a positive function $u \in C^{\infty}(\mcx)$ such that $\hg = u^{-2}g$ and $\hv = u^{-1}v$.
The weighted conformal Laplacians $L^m_{g,v}$ and $L^m_{\hg,\hv}$ have the conformal covariance property
\begin{equation}\label{eq:wcl2}
L^m_{\hg,\hv} w = u^{\frac{m+n+3}{2}} L^m_{g,v}\(u^{-\frac{m+n-1}{2}} w\) \quad \textup{for all } w \in C^{\infty}(\mcx),
\end{equation}
which justifies their name.

\subsection{Function spaces}
We introduce functional spaces to work with throughout the paper.
\begin{definition}\label{def:fs1}
Let $x=(\bx,x_N) \in \R^N_+$ where $\bx \in \R^n$ and $x_N \in (0,\infty)$. Assume that $p \in (1,\infty)$ and $q \in (0,\infty]$.
\begin{itemize}
\item[-] Let $L^p(\R^N_+;x_N^{\mz})$ and $L^p_w(\R^N_+;x_N^{\mz})$ be the (weighted) $L^p$-space and the (weighted) weak $L^p$-space on $\R^N_+$ with measure $x_N^{\mz}dx$.
\item[-] Let $\dot{H}^{1,2}(\R^N_+;x_N^{\mz})$ be the homogeneous (weighted) Sobolev space on $\R^N_+$ with measure $x_N^{\mz}dx$.
\item[-] Let $\dot{H}^{\ga}(\R^n)$ be the homogeneous fractional Sobolev space on $\R^n$.
\item[-] Let $L^{p,q}(\R^n)$ be the Lorentz space on $\R^n$.
\end{itemize}
\end{definition}

\begin{definition}\label{def:fs2}
Let $(\mcx,g)$ be a Riemannian manifold with or without boundary $M$, $\rho$ the defining function of $M$, and $h = g|_{TM}$ (if $M \ne \emptyset$). Assume that $p \in (1,\infty]$.
\begin{itemize}
\item[-] Let $L^p(\mcx,g)$ be the $L^p$-space on $\mcx$ with measure $dv_g$.
\item[-] Let $L^p(\mcx;\rho^{\mz},g)$ be the (weighted) $L^p$-space on $\mcx$ with measure $\rho^{\mz}dv_g$.
\item[-] Let $H^{1,2}(\mcx;\rho^{\mz},g)$ be the (weighted) Sobolev space on $\mcx$ with measure $\rho^{\mz}dv_g$.
\item[-] Let $H^{1,2}_0(\mcx;\rho^{\mz},g)$ be the space consisting of functions in $H^{1,2}(\mcx;\rho^{\mz},g)$ whose trace on $M$ is zero.
\item[-] Let $H^{\ga}(M,h)$ be the fractional Sobolev space on $M$, that is, the completion of $C^{\infty}_c(M)$ with respect to the norm obtained by pulling back that of the inhomogeneous fractional Sobolev space $H^{\ga}(\R^n)$
    to $(M,h)$ via coordinate charts.
\end{itemize}
\end{definition}

\subsection{Chang-Gonz\'alez extension theorem and related topics}\label{subsec:ext}
We borrow a result from \cite[Theorem 4.1]{CC}, allowing an inhomogeneous term $Q$. Refer also to \cite[Lemma 4.1, Theorem 4.3]{CG} and \cite[Section 3]{MN}.

\begin{prop}\label{prop:ext}
Assume that $n \in \N$, $n > 2\ga$, $\ga \in (0,1)$, $(X^N,g^+)$ is a CCE manifold with conformal infinity $(M^n,[\bh])$,
$\rhog$ is the geodesic defining function of $(M,\bh)$, and $\rho \in \Xi_{g^+,\bh}$; refer to \eqref{eq:Xi}.
We consider the smooth metric measure spaces $(\ox,\bg=\rho^2g^+,\rho^{\mz}dv_{\bg},-2\ga)$ and $(X,g^+,dv_{g^+},-2\ga)$,
and the corresponding weighted conformal Laplacians $L^{\mz}_{\bg,\rho}$ and $L^{\mz}_{g^+,\,1}$, respectively.
Given $f \in C^{\infty}(M)$ and $Q \in C^0(\ox)$, a function $U$ is the solution to
\begin{equation}\label{eq:degeq31}
\begin{cases}
\rho^{\mz} L^{\mz}_{\bg,\rho} U = -\textup{div}_{\bg}\(\rho^{\mz} \nabla_{\bg} U\) + \dfrac{n-2\ga}{4(n+1-2\ga)} \rho^{\mz} R^{\mz}_{\bg,\rho} U = \rho^{\mz}Q &\text{in } X,\\
U = \rho^{s-n}\(F\rhog^{n-s} + G\rhog^s\) &\text{for } F,\, G \in C^2(\ox),\\
U = f &\text{on } M
\end{cases}
\end{equation}
if and only if $u = \rho^{n-s}U$ is the solution to
\begin{equation}\label{eq:degeq32}
\begin{cases}
L^{\mz}_{g^+,\,1}(u) = -\Delta_{g^+}u - s(n-s)u = \rho^{n-s+2}Q &\text{in } X,\\
u = F\rhog^{n-s} + G\rhog^s &\text{for } F,\, G \in C^2(\ox),\\
F|_{\rho=0} = f
\end{cases}
\end{equation}
where $s = \frac{n}{2}+\ga$. Furthermore, if $Q = 0$ on $\ox$, then there holds that
\begin{equation}\label{eq:degeq41}
P^{\ga}_{g^+,\bh}f - \frac{n-2\ga}{2}\, d_\ga\Phi f = \frac{d_{\ga}}{2\ga} \lim_{\rho \to 0} \rho^{\mz}{\frac{\pa U}{\pa \rho}} \quad \text{on } M
\end{equation}
where $d_{\ga} < 0$ is the constant in \eqref{eq:degeq2}.
\end{prop}
\noindent Equation \eqref{eq:degeq3}--\eqref{eq:degeq4} is a special case of \eqref{eq:degeq31} and \eqref{eq:degeq41} obtained by taking $\rho = \rhog$.

\medskip
We next deduce an unique existence result of \eqref{eq:degeq31} with $Q = 0$, after substituting the pointwise condition $U = \rho^{s-n}(F\rhog^{n-s} + G\rhog^s)$ with an energy condition $U \in H^{1,2}(X;\rho^{\mz},\bg)$.
This can be seen as an extension of the uniqueness assertion in \cite[Proposition 3.4]{GZ} to the energy space.
\begin{prop}\label{prop:ext2}
Assume that the hypotheses stated in Proposition \ref{prop:ext} hold and $\frac{n^2}{4}-\ga^2 \notin \sigma_{\pp}(-\Delta_{g^+})$. Then the boundary value problem
\begin{equation}\label{eq:degeq5}
\begin{cases}
-\textup{div}_{\bg}\(\rho^{\mz} \nabla_{\bg} U\) + \dfrac{n-2\ga}{4(n+1-2\ga)} \rho^{\mz} R^{\mz}_{\bg,\rho} U = 0 &\text{in } X,\\
U = f &\text{on } M,\\
U \in H^{1,2}(X;\rho^{\mz},\bg)
\end{cases}
\end{equation}
is uniquely solvable. In particular, if $(\bg,\rho) = (\bgg,\rhog)$, then $U = F + G\rhog^{2\ga}$ on $\ox$ for some $F,\, G \in C^2(\ox)$ such that $F|_{\rho=0} = f$.
\end{prop}
\begin{proof}
According to \cite{GZ}, 
\eqref{eq:degeq32} with $Q=0$ has a unique solution $\tu$ provided $\frac{n^2}{4}-\ga^2 \notin \sigma_{\pp}(-\Delta_{g^+})$.
By Proposition \ref{prop:ext} and \eqref{eq:Xi}, $\wtu := \rho^{s-n}\tu$ is a solution to \eqref{eq:degeq31} and
\[\int_X \rho^{\mz} \(\big|\nabla_{\bg} \wtu\big|_{\bg}^2 + \wtu^2\) dv_{\bg} \le C \int_X \rho^{\mz} \(\rho^{-2\mz} + 1\) dv_{\bg} \le C,\]
which means that $\wtu \in H^{1,2}(X;\rho^{\mz},\bg)$.

Let $U$ be a solution to \eqref{eq:degeq5} and choose any $\ep \in (0,\ga)$. We claim that there exists a constant $C > 0$ independent of $\rho$ (but it may depend on $U$, $\wtu$, and $\ep$), such that
\begin{equation}\label{eq:degeq51}
\big|U-\wtu\big| \le C\rho^{2\ga-\ep} \textup{ in a collar neighborhood of } M.
\end{equation}
If we write $Z = (\rho/\rhog)^{\frac{n-2\ga}{2}}(U-\wtu)$ and $\bgg = \rhog^2g^+$, then
\[\int_X \rhog^{\mz} \(\big|\nabla_{\bgg} Z|_{\bgg}^2 + Z^2\) dv_{\bgg} \le C \int_X \rho^{\mz} \left[\big|\nabla_{\bg} \big(U-\wtu\big)\big|_{\bg}^2 + \(\rho^{-2\mz} + 1\) \big(U-\wtu\big)^2\right] dv_{\bg} < \infty\]
where the latter inequality is a consequence of Lemma \ref{lemma:embed}. This, \eqref{eq:degeq5}, and \eqref{eq:wcl2} yield
\begin{equation}\label{eq:degeq52}
\begin{cases}
-\textup{div}_{\bgg}\(\rhog^{\mz} \nabla_{\bgg} Z\) + \dfrac{n-2\ga}{4(n+1-2\ga)} \rhog^{\mz} R^{\mz}_{\bgg,\rhog} Z= 0 &\textup{in } X,\\
Z = 0 &\textup{on } M,\\
Z \in H^{1,2}(X;\rhog^{\mz},\bgg).
\end{cases}
\end{equation}
By \eqref{eq:Xi} and the compactness of $(\ox,\bgg)$, it suffices to prove that
\begin{equation}\label{eq:degeq53}
|Z| \le C\rhog^{2\ga-\ep} \quad \textup{on } B^n_{\bh}(\sigma,r_0) \times [0,r_0) \subset \ox
\end{equation}
where $B^n_{\bh}(\sigma,r_0)$ is the geodesic ball on $M$ centered at $\sigma \in M$ and of small radius $r_0 > 0$.
Identifying $B^n_{\bh}(\sigma,2r_0)$ with the ball $B^n(0,2r_0)$ in $\R^n$ of radius $2r_0$ centered at $0$ via normal coordinates,
we denote $\mcc^N(0,2r_0) = B^n(0,2r_0) \times (0,2r_0)$. Equation \eqref{eq:degeq52} is then reduced to
\begin{equation}\label{eq:degeq54}
\begin{cases}
-\textup{div}_{\bgg}\(x_N^{\mz} \nabla_{\bgg} Z\) + x_N^{\mz} AZ = 0 &\textup{in } \mcc^N(0,2r_0),\\
Z = 0 &\textup{on } B^n(0,2r_0),\\
Z \in H^{1,2}\(\mcc^N(0,2r_0);x_N^{\mz},\bgg\)
\end{cases}
\end{equation}
where $x_N = \rhog$ and $A := \frac{n-2\ga}{4(n+1-2\ga)} R^{\mz}_{\bgg,\rhog}$.
From the assumption that $(X,g^+)$ is CCE and \eqref{eq:wsc}, we infer
\begin{equation}\label{eq:degeq55}
A 
= \frac{n-2\ga}{4(n+1-2\ga)} \(R_{\bgg} - \frac{2\mz}{\sqrt{|\bgg|}} \frac{\pa_{x_N} \sqrt{|\bgg|}}{x_N}\)
\in C^{\max\{1,n-3\},\beta}\(\overline{\mcc^N(0,2r_0)}\),
\end{equation}
and $A$ is smooth in the tangential directions. In Appendix \ref{subsec:app2}, we will show that all solutions to \eqref{eq:degeq54} satisfy \eqref{eq:degeq53}, thereby proving \eqref{eq:degeq51}.

Let $u-\tu = \rho^{n-s}(U-\wtu)$. By \eqref{eq:degeq51},
\[\int_X (u-\tu)^2 dv_{g^+} = \int_X \rho^{\mz} \left[\rho^{-1} \big(U-\wtu\big)\right]^2 dv_{\bg} \le C\left[1 + \rho^{2(\ga-\ep)}\right] < \infty.\]
Therefore,
\[\left[-\Delta_{g^+}-s(n-s)\right](u-\tu) = 0 \quad \textup{in } X \quad \textup{and} \quad u-\tu \in L^2(X,g^+).\]
Because $s(n-s) 
\notin \sigma_{\pp}(-\Delta_{g^+})$, it holds that $u-\tu = 0$ in $X$. Consequently, $U = \wtu$ on $\ox$.
\end{proof}

Given the geodesic defining function $\rhog$ of $(M,\bh)$ and $\bgg = \rhog^2 g^+$, we also study the unique existence and a priori estimate for a solution to the inhomogeneous problem
\begin{equation}\label{eq:degeq6}
\begin{cases}
-\textup{div}_{\bgg}\(\rhog^{\mz} \nabla_{\bgg} U\) + \dfrac{n-2\ga}{4(n+1-2\ga)} \rhog^{\mz} R^{\mz}_{\bgg,\rhog} U = \rhog^{\mz} Q &\text{in } X,\\
U = 0 &\text{on } M,\\
U \in H^{1,2}(X;\rhog^{\mz},\bgg).
\end{cases}
\end{equation}
On $\overline{\mct_{\delta_0}} = M \times [0,\delta_0] \subset \ox$ (see \eqref{eq:g+ ext}), let $\nabla_{\bx}$ be the tangential gradient.
\begin{prop}\label{prop:ext3}
Assume that the hypotheses stated in Proposition \ref{prop:ext} hold, $\frac{n^2}{4}-\ga^2 \notin \sigma_{\pp}(-\Delta_{g^+})$, and $Q \in L^2(X;\rhog^{\mz},\bgg)$.
Then \eqref{eq:degeq6} admits a unique solution $U$ and
\begin{equation}\label{eq:extreg11}
\|U\|_{H^{1,2}(X;\rhog^{\mz},\bgg)} \le C\|Q\|_{L^2(X;\rhog^{\mz},\bgg)}
\end{equation}
where $C > 0$ depends only on $n$, $\ga$, $(\ox,\bgg)$, and $\rhog$. Furthermore,
\begin{multline}\label{eq:extreg12}
\|U\|_{C^{0,\beta}(\ox)} + \|U\|_{C^{2,\beta}(X \setminus \mct_{\delta_0/4})} + \sum_{\ell=1}^2 \left\|\nabla_{\bx}^{\ell} U\right\|_{C^{0,\beta}(\overline{\mct_{\delta_0}})}
+ \left\|\rhog^{\mz}\pa_{\rhog} U\right\|_{C^{0,\beta}(\overline{\mct_{\delta_0}})} \\
\le C\left[\sum_{\ell=0}^2 \left\|\nabla_{\bx}^{\ell} Q\right\|_{L^{\infty}(\mct_{\delta_0})} + \|Q\|_{C^{0,\beta}(X \setminus \mct_{\delta_0/8})}\right]
\end{multline}
provided the right-hand side is finite. Here, $C > 0$ and $\beta \in (0,1)$ depend only on $n$, $\ga$, $(\ox,\bgg)$, and $\rhog$.
\end{prop}
\begin{proof}
We divide the proof into three steps.

\medskip \noindent \textsc{Step 1: Existence result for smooth $Q$}.
We aim to build a solution to \eqref{eq:degeq6} provided $Q \in C^{\infty}(\ox)$. Considering Proposition \ref{prop:ext}, it is enough to find a solution to \eqref{eq:degeq32} with $f = 0$ on $M$.
We will achieve this goal by adapting the argument in \cite{GZ}; cf. \cite[Section 3]{MN}.

For the moment, we assume that the metric $h_{\rhog}$ in \eqref{eq:g+ ext} is smooth and even in $\rhog$ for all orders, which holds when $n$ is odd or $2$.
Let $\sum_{j=0}^{\infty} q_j\rhog^j$ be the formal Taylor expansion of $Q$ in $\rhog$ where $q_j \in C^{\infty}(M)$. For $j \in \N \cup \{0\}$ and $f_j \in C^{\infty}(M)$, it holds that
\[\left[-\Delta_{g^+}-s(n-s)\right]\(f_j\rhog^{n-s+j}\) = \rhog^{n-s+1} \mcd_{n,\ga} \(f_j\rho_g^j\) \quad \text{in } \mct_{\delta_0}\]
where
\begin{align*}
&\ \mcd_{n,\ga}\(f_j\rhog^j\) \\
&:= \left[-\rhog \pa_{\rhog}^2 + \left\{2\ga-1-\frac{\rhog}{2} \text{Trace}_{h_{\rhog}} \pa_{\rhog} h_{\rhog}\right\} \pa_{\rhog} - \frac{n-2\ga}{4} \text{Trace}_{h_{\rhog}} \pa_{\rhog} h_{\rhog} - \rhog \Delta_{h_{\rhog}}\right]\(f_j\rhog^j\) \\
&= -j(j-2\ga) f_j\rhog^{j-1} + O\(\rhog^j\).
\end{align*}
Beginning with $F_1 = 0$, we set $f_j$ and $F_j$ for $j \ge 2$ by
\[\begin{cases}
\displaystyle f_j = \frac{1}{j(j-2\ga)} \left[\Big(\rhog^{1-j}\mcd_{n,\ga}(F_{j-1}) - \sum_{k=0}^{j-3} q_k\rhog^{k+2-j}\Big)\Big|_{\rhog=0} - q_{j-2}\right],\\
\displaystyle F_j = F_{j-1}+f_j\rhog^j.
\end{cases}\]
Then $\rhog^{n-s} F_{\infty} := \rhog^{n-s} \sum_{j=2}^{\infty} f_j\rhog^j$ is a formal solution to $[-\Delta_{g^+}-s(n-s)]u-\rhog^{n-s+2}Q = O(\rhog^{\infty})$ in $\mct_{\delta_0}$.
By employing Borel's lemma, we can slightly adjust $F_{\infty}$ to construct $\wtf_{\infty} \in C^{\infty}(\ox)$
such that $[-\Delta_{g^+}-s(n-s)](\rhog^{n-s}\wtf_{\infty})-\rhog^{n-s+2}Q \in \dot{C}^{\infty}(\ox)$
where $\dot{C}^{\infty}(\ox)$ is the space of smooth functions on $\ox$ vanishing to infinite order on $M$. Finally, using \cite[Lemma 6.13]{MM}, we observe that
\[u = \rhog^{n-s}\wtf_{\infty} - R(s)\left[\{-\Delta_{g^+}-s(n-s)\}\(\rhog^{n-s}\wtf_{\infty}\)-\rhog^{n-s+2}Q\right] \in \rhog^{n-s}C^{\infty}(\ox) + \rhog^sC^{\infty}(\ox)\]
solves \eqref{eq:degeq32} with $f = 0$ on $M$.
Here, $R(s) = [-\Delta_{g^+}-s(n-s)]^{-1}$ is the resolvent, a bounded operator in $L^2(X,g^+)$ provided $\frac{n^2}{4}-\ga^2 \notin \sigma_{\pp}(-\Delta_{g^+})$. It maps $\dot{C}^{\infty}(\ox)$ to $\rhog^sC^{\infty}(\ox)$.

Generally, when $n \ge 4$ is even, the expansion of $\mcd_{n,\ga} (f_j\rho_g^j)$ may have log-order terms due to the part $\rhog \Delta_{h_{\rhog}} (f_j\rho_g^j)$.
Accordingly, it is necessary to modify $F_{\infty}$ into a polyhomogeneous function:
A smooth function $F_{\infty}$ in $\mct_{\delta_0}$ is said to be polyhomogeneous if there exist a 
sequence $\{k_j\}_{j=0}^{\infty}$ of nonnegative integers and functions $f_{jk} \in C^{\infty}(M)$ such that
\[F_{\infty} = \sum_{j=2}^{\infty} \sum_{k=0}^{k_j} f_{jk} \rhog^j|\log \rhog|^k \quad \text{and} \quad k_j = 0 \quad \text{for } \begin{cases}
j = 2,\ldots,n+1 &\text{if } n \ge 4 \text{ is even},\\
j \in \N &\text{if } n \text{ is odd or } 2;
\end{cases}\]
cf. \cite[Section 5]{CDLS}. Using this new $F_{\infty}$, we can construct a genuine solution to \eqref{eq:degeq32} as in the previous paragraph.

\medskip \noindent \textsc{Step 2: Existence result for $Q \in L^2(X;\rhog^{\mz},\bgg)$}.
We have a Caccioppoli estimate
\begin{equation}\label{eq:extreg21}
\|U\|_{H^{1,2}(X;\rhog^{\mz},\bgg)} \le C\left[\|U\|_{L^2(X;\rhog^{\mz},\bgg)} + \|Q\|_{L^2(X;\rhog^{\mz},\bgg)}\right].
\end{equation}
The proof of Proposition \ref{prop:ext2} implies that a solution to \eqref{eq:degeq6} is unique.
Thus a standard contradiction argument that utilizes \eqref{eq:extreg21}, the uniqueness property,
and a compact embedding $H^{1,2}(X;\rhog^{\mz},\bgg) \hookrightarrow L^2(X;\rhog^{\mz},\bgg)$ yields \eqref{eq:extreg11}.

Fix any $Q_{\infty} \in L^2(X;\rhog^{\mz},\bgg)$.
By applying a smooth partition of unity and convolution, one obtains a sequence $\{Q_i\}_{i=1}^{\infty}$ in $C^{\infty}(\ox)$ such that $Q_i \to Q_{\infty}$ in $L^2(X;\rhog^{\mz},\bgg)$ as $i \to \infty$.
According to the previous step, there is a unique solution $U_i$ to \eqref{eq:degeq6} with $Q = Q_i$.
Also, we see from \eqref{eq:extreg11} that $\{U_i\}_{i=1}^{\infty}$ is bounded in $H^{1,2}(X;\rhog^{\mz},\bgg)$.
Its weak limit $U_{\infty}$ solves \eqref{eq:degeq6} with $Q = Q_{\infty}$.

\medskip \noindent \textsc{Step 3: Existence result for regular $Q$}.
From the classical Schauder theory\footnote{Here and after, the `classical' theory
means that the theory for uniformly elliptic partial differential equations.} and a simple covering argument with Lemma \ref{lemma:ereg}, we deduce that
\begin{equation}\label{eq:extreg22}
\begin{cases}
\begin{medsize}
\displaystyle \|U\|_{C^{2,\beta}(X \setminus \mct_{\delta_0/4})}
\le C\left[\|U\|_{L^2(X \setminus \mct_{\delta_0/8};\rhog^{\mz},\bgg)} + \|Q\|_{C^{0,\beta}(X \setminus \mct_{\delta_0/8})} \right],
\end{medsize} \\
\begin{medsize}
\displaystyle \sum_{\ell=0}^2 \left\|\nabla_{\bx}^{\ell} U\right\|_{C^{0,\beta}(\overline{\mct_{\delta_0/2}})}
+ \left\|\rhog^{\mz}\pa_{\rhog} U\right\|_{C^{0,\beta}(\overline{\mct_{\delta_0/2}})}
\le C\left[\|U\|_{L^2(\mct_{\delta_0};\rhog^{\mz},\bgg)} + \sum_{\ell=0}^2 \left\|\nabla_{\bx}^{\ell} Q\right\|_{L^{\infty}(\mct_{\delta_0})}\right]
\end{medsize}
\end{cases}
\end{equation}
for some $\beta \in (0,1)$. Estimate \eqref{eq:extreg12} is a consequence of \eqref{eq:extreg22} and the uniqueness property of \eqref{eq:degeq6}.
\end{proof}

We close this section by quoting the comment of Case and Chang \cite{CC} on the spectral condition $\lambda_1(-\Delta_{g^+}) > \frac{n^2}{4}-\ga^2$.
This condition is necessary and sufficient for the unique solution $u$ to \eqref{eq:degeq32} with $Q=0$ and $f=1$ to be positive in $X$.
By Proposition \ref{prop:ext}, $U = \rhog^{s-n}u$ satisfies \eqref{eq:degeq31}, and hence $L^{\mz}_{\bgg,\rhog} U = 0$ in $X$.
The map $\rho^* := u^{\frac{2}{n-2\ga}}$ is a defining function of $M$ with expansion
\begin{equation}\label{eq:rho*_exp}
\rho^* = \rhog \left[1 + \frac{1}{d_{\ga}}\, Q^{\ga}_{g^+,\bh} \rhog^{2\ga} + O\(\rhog^{2\min\{1,2\ga\}}\)\right] \quad \textup{for } \rhog > 0 \textup{ small},
\end{equation}
where $Q^{\ga}_{g^+,\bh} := \frac{2}{n-2\ga} P^{\ga}_{g^+,\bh}1$ is the fractional scalar curvature on $(M,\bh)$. Hence \eqref{eq:degeq41} reads
\[P^{\ga}_{g^+,\bh} f - \frac{n-2\ga}{2}\, Q^{\ga}_{g^+,\bh} f = \frac{d_{\ga}}{2\ga} \lim_{\rho^* \to 0} (\rho^*)^{\mz}{\frac{\pa U}{\pa \rho^*}} \quad \textup{on } M\]
Because of its special role, $\rho^*$ is called the adapted defining function of $(M,\bh)$.\footnote{Because the adapted defining function allows one to handle the fractional scalar curvature directly, it became an essential tool in studying the fractional Yamabe problem \cite{GQ, GW, KMW, DSV, MN2}.}
Owing to \eqref{eq:wcl} and \eqref{eq:wcl2}, $R^{\mz}_{(\rho^*)^2g^+,\rho^*} = 0$ in $X$ so that $(\rho^*)^2g^+ \in \mcb_{g^+,\,[\bh]}$.

\section{Weighted Poisson kernel and weighted isoperimetric ratio}\label{sec:weight}
For this section, we recall Convention \ref{conv}.

\subsection{Weighted Poisson kernel}\label{subsec:wpk}
Let $\delta_0 > 0$ be a small number such that \eqref{eq:g+ ext} holds,
$\mct_{\delta_0} = M \times (0,\delta_0)$, and $\pi: \overline{\mct_{\delta_0}} \to M$ be the orthogonal projection onto $M$
so that any $\xi \in \overline{\mct_{\delta_0}}$ can be represented as $\xi = (\pi(\xi),\rhog(\xi))$ where $\rhog(\xi) \in [0,\delta_0]$ is the value of the geodesic defining function of $(M,\bh)$ at $\xi$.

\begin{definition}\label{def:wPk}
We say that $\mck^{\mz}_{\bg,\rho}: \ox \times M \to \R$ \footnote{Although the precise domain of $\mck^{\mz}_{\bg,\rho}$ is $\{(\xi,\sigma) \in \ox \times M: \xi \ne \sigma\}$,
we used a slight abuse of notation for the sake of brevity. We may set $\mck^{\mz}_{\bg,\rho}(\xi,\sigma) = +\infty$ if $\xi = \sigma \in M$.} is the weighted Poisson kernel if it satisfies the following property:
For any given $f \in C^{\infty}(M)$, a solution $U$ to \eqref{eq:degeq5} is written as
\begin{equation}\label{eq:Poibg}
U(\xi) = \(\mck^{\mz}_{\bg,\rho} f\)(\xi) := \int_M \mck^{\mz}_{\bg,\rho}(\xi,\sigma)f(\sigma) (dv_{\bh})_{\sigma} \quad \textup{for } \xi \in \ox.
\end{equation}
Formally, it holds that
\begin{equation}\label{eq:Poi01}
\begin{cases}
-\textup{div}_{\bg}\(\rho^{\mz} \nabla_{\bg} \mck^{\mz}_{\bg,\rho}(\cdot,\sigma)\) + \dfrac{n-2\ga}{4(n+1-2\ga)} \rho^{\mz} R^{\mz}_{\bg,\rho} \mck^{\mz}_{\bg,\rho}(\cdot,\sigma) = 0 &\textup{in } X,\\
\lim\limits_{\rhog(\xi) \to 0} \mck^{\mz}_{\bg,\rho}(\xi,\sigma) = \delta_{\sigma}(\pi(\xi)) &\textup{for } \xi \in \mct_{\delta_0}
\end{cases}
\end{equation}
given any $\sigma \in M$. Here, $\delta_{\sigma}$ signifies the Dirac delta measure at $\sigma$.
\end{definition}
\noindent Mayer and Ndiaye \cite{MN} previously established the existence of the Poisson kernel and its local expansion near $\sigma$ (similar to \eqref{eq:Poi03})
when $\ga \in (0,1) \setminus \{\frac{1}{2}\}$ and $\bgg$ is smooth on $\ox$; the latter condition is implicitly assumed in their paper.
In general, the tensor $h_*$ appearing in Subsection \ref{subsec:fcl} obstructs the smoothness of $\bgg$ on $M$, and $\bgg$ is smooth on $\ox$ only if $h_* = 0$ on $M$.
In the following proposition, we reformulate and refine the results in \cite{MN}, including the cases that $\ga = \frac{1}{2}$ and $\bgg$ is non-smooth.
\begin{prop}\label{prop:Poi}
Assume that $n \in \N$, $n > 2\ga$, $\ga \in (0,1)$, $\mz = 1-2\ga$, $(X,g^+)$ is CCE, and $\frac{n^2}{4}-\ga^2 \notin \sigma_{\pp}(-\Delta_{g^+})$.
Let $d_{\bgg}(\xi_1,\xi_2)$ be the distance between $\xi_1$ and $\xi_2$ on $\ox$ with respect to the metric $\bgg$
and $d_{\bh}(\sigma_1,\sigma_2)$ the distance between $\sigma_1$ and $\sigma_2$ on $M$ with respect to the metric $\bh$. We also define
\begin{equation}\label{eq:Poi02}
\mck_0^{\mz}(\xi,\sigma) = \ka_{n,\ga} \frac{\rhog(\xi)^{2\ga}}{[d_{\bh}(\pi(\xi),\sigma)^2 + \rhog(\xi)^2]^{\frac{n+2\ga}{2}}}
\end{equation}
for $\xi = (\pi(\xi),\rhog(\xi)) \in \overline{\mct_{\delta_0}}$ and $\sigma \in M$. 
Then there exists a weighted Poisson kernel $\mck_{\bgg,\rhog}^{\mz}$ having the following properties:

\medskip \noindent \textup{(a)} We choose small $r_1 \in (0,\frac{1}{2}\min\{\delta_0,\textup{inj}(M,\bh)\})$ where $\textup{inj}(M,\bh)$ is the injectivity radius of $(M,\bh)$.
Fixing $\sigma \in M$, let $x=(\bx,x_N) \in \overline{\mcc^N(0,r_1)} = \overline{B^n(0,r_1)} \times [0,r_1]$ be Fermi coordinates on $\ox$ around $\sigma$.
We set $r = |x|$ and $\theta = (\bth,\theta_N) = (\theta_1,\ldots,\theta_N) = \frac{x}{|x|} \in \S^n_+ := \S^n \cap \R^N_+$. Then
\begin{multline}\label{eq:Poi03}
\begin{medsize}
\mck_{\bgg,\rhog}^{\mz}((\exp_{\sigma}\bx,x_N),\sigma)
\end{medsize} \\
\begin{medsize}
\displaystyle =
\begin{cases}
\displaystyle \ka_{n,\ga} \Bigg[\frac{1}{r^n}\Bigg\{\theta_N^{2\ga} + \sum_{j=2}^{n+4}r^ja_j(\theta)
+ \sum_{j=n}^{n+4}\sum_{k=1}^{k_j} r^j|\log r|^ka_{jk}(\theta)\Bigg\} + \mcr(x)\Bigg]
&\text{if } \ga \in (0,1) \setminus \left\{\frac{1}{2}\right\}, \\
\displaystyle \ka_{n,1/2} \Bigg[\frac{1}{r^n}\Bigg\{\theta_N + \sum_{j=2}^{n+1}r^ja_j(\theta) + \sum_{j=n}^{n+1}\sum_{k=1}^{k_j} r^j|\log r|^ka_{jk}(\theta)\Bigg\} + \mcr(x)\Bigg]
&\text{if } \ga = \frac{1}{2}
\end{cases}
\end{medsize}
\end{multline}
for $x \in \overline{\mcc^N(0,\frac{r_1}{2})}$. Here, $\exp$ is the exponential map on $M$ and $\ka_{n,\ga} > 0$ is the number in \eqref{eq:gaharext2}. In addition,
\begin{itemize}
\item[-] $a_j$ is a function on $\overline{\S^n_+}$ such that $a_j = 0$ on $\pa\S^n_+$.
     The function $a_j$ itself, its tangential derivatives $\nabla_{\bth}^{\ell}a_j$ for $\ell \in \N$,
     and its weighted normal derivative $\theta_N^{\mz}\pa_{\theta_N}a_j$ are H\"older continuous on $\overline{\S^n_+}$.
     Also, $a_j, \nabla_{\bth}^{\ell} a_j \in \theta_N^{2\ga} C^0(\overline{\S^n_+})$ for $\ell \in \N$.
\item[-] $a_{jk}$ is a function on $\overline{\S^n_+}$ that vanishes on $\pa\S^n_+$ and has the same regularity as $a_j$. Also, $\{k_j\}_{j=n}^{\infty}$ is a sequence of nonnegative integers.
    If the metric $h_{\rhog}$ in \eqref{eq:g+ ext} is smooth and even in $\rhog$ for all orders, then
\begin{equation}\label{eq:Poi06}
\begin{cases}
a_{n+4} = a_{jk} = 0 &\text{if } \ga \in (0,1) \setminus \{\frac{1}{2}\},\\
a_{nk} = a_{(n+1)2} = \cdots = a_{(n+1)k_{n+1}} = 0 &\text{if } \ga = \frac{1}{2}.
\end{cases}
\end{equation}
\item[-] $\mcr$ is a function on $\overline{\mcc^N(0,r_1)}$ such that $\mcr = 0$ on $\overline{B^n(0,r_1)} \times \{0\}$.
    The functions $\mcr$, $\nabla_{\bx}\mcr$, $\nabla_{\bx}^2\mcr$, and $x_N^{\mz}\pa_{x_N}\mcr$ are H\"older continuous on $\overline{\mcc^N(0,r_1)}$. In particular, $\mcr \in x_N^{2\ga}C^0(\overline{\mcc^N(0,r_1)})$.
\end{itemize}

\medskip \noindent \textup{(b)} There exists a constant $C > 0$ depending only on $n$, $\ga$, $(\ox,\bgg)$, and $\rhog$ such that
\begin{equation}\label{eq:Poi21}
\left|\mck_{\bgg,\rhog}^{\mz}(\xi,\sigma)\right| \le C\mck_0^{\mz}(\xi,\sigma) \quad \text{for } (\xi,\sigma) \in \ox \times M
\end{equation}
and
\begin{align}
& \left|\(\mck_{\bgg,\rhog}^{\mz} - \mck_0^{\mz}\)(\xi,\sigma)\right| \le C \frac{\rhog(\xi)^{2\ga}}{d_{\bgg}(\xi,\sigma)^{n+2\ga-2}} \label{eq:Poi22}, \\
& \left|\nabla_{\pi(\xi)} \mck_{\bgg,\rhog}^{\mz}(\xi,\sigma)\right| \le C \frac{\rhog(\xi)^{2\ga}}{d_{\bgg}(\xi,\sigma)^{n+2\ga+1}},\quad
\left|\pa_{\rhog(\xi)} \mck_{\bgg,\rhog}^{\mz}(\xi,\sigma)\right| \le C \frac{\rhog(\xi)^{2\ga-1}}{d_{\bgg}(\xi,\sigma)^{n+2\ga}}, \label{eq:Poi23} \\
& \begin{medsize}
\displaystyle \left|\nabla_{\pi(\xi)} \(\mck_{\bgg,\rhog}^{\mz}-\mck_0^{\mz}\)(\xi,\sigma)\right| \le C \frac{\rhog(\xi)^{2\ga}}{d_{\bgg}(\xi,\sigma)^{n+2\ga-1}},\quad
\left|\pa_{\rhog(\xi)} \(\mck_{\bgg,\rhog}^{\mz}-\mck_0^{\mz}\)(\xi,\sigma)\right| \le C \frac{\rhog(\xi)^{2\ga-1}}{d_{\bgg}(\xi,\sigma)^{n+2\ga-2}}
\end{medsize}\label{eq:Poi24}
\end{align}
for $\xi = (\pi(\xi),\rhog(\xi)) \in \overline{\mct_{\delta_0}}$ and $\sigma \in M$. Here, $\nabla_{\pi(\xi)}$ denotes the tangential gradient.

\medskip \noindent \textup{(c)} For each fixed $\xi \in X$, the kernel $\mck_{\bgg,\rhog}^{\mz}(\xi,\sigma)$ is continuous with respect to the $\sigma$-variable.
Furthermore, there exists a constant $C > 0$ depending only on $n$, $\ga$, $(\ox,\bgg)$, and $\rhog$ such that
\begin{equation}\label{eq:Poi26}
\left|\nabla_{\sigma} \mck_{\bgg,\rhog}^{\mz}(\xi,\sigma)\right| \le C \frac{\rhog(\xi)^{2\ga}}{d_{\bgg}(\xi,\sigma)^{n+2\ga+1}},
\quad \left|\nabla_{\sigma} \(\mck_{\bgg,\rhog}^{\mz}-\mck_0^{\mz}\)(\xi,\sigma)\right| \le C \frac{\rhog(\xi)^{2\ga}}{d_{\bgg}(\xi,\sigma)^{n+2\ga-1}}
\end{equation}
for $(\xi,\sigma) \in \ox \times M$.

\medskip \noindent \textup{(d)} If $\lambda_1(-\Delta_{g^+}) > \frac{n^2}{4}-\ga^2$, then
\begin{equation}\label{eq:Poi25}
\mck_{\bgg,\rhog}^{\mz}(\xi,\sigma) \ge 0 \quad \text{for } (\xi,\sigma) \in \ox \times M.
\end{equation}
If $\xi \in X$ and $\sigma \in M$ are elements of the same connected component of $\ox$, then $\mck_{\bgg,\rhog}^{\mz}(\xi,\sigma) > 0$. Otherwise, $\mck_{\bgg,\rhog}^{\mz}(\xi,\sigma) = 0$.
\end{prop}

To prove this proposition, we will employ the next two elementary results.
\begin{lemma}\label{lemma:Poi1}
Under the setting of Proposition \ref{prop:Poi}(a), the following expansions of the metric $\bgg$ hold near $0$:
\[\sqrt{|\bgg|}(x) = 1 - \frac{1}{2} R_{NN}[\bgg]x_N^2 - \sum_{k,l=1}^n\frac{1}{6} R_{kl}[\bh]x_kx_l + O\(|x|^3\),\quad \pa_{x_N}\sqrt{|\bgg|}(x) = O(x_N),\]
and
\[\bgg^{ij}(x) = \delta_{ij} + \frac{1}{3} \sum_{k,l=1}^n R_{ikjl}[\bh]x_kx_l + R_{iNjN}[\bgg] x_N^2 + O\(|x|^3\) \quad \text{for } i, j = 1,\ldots,n.\]
Here, $R_{ikjl}[\bh]$ and $R_{iNjN}[\bgg]$ are components of the Riemannian curvature tensors on $(M,\bh)$ and $(\ox,\bgg)$, respectively,
$R_{ij}[\bh] = \sum_{k=1}^n R_{ikjk}[\bh]$, and $R_{NN}[\bg] = \sum_{i=1}^n R_{iNiN}[\bg]$. Each tensor in the expansions is evaluated at $0$.
\end{lemma}
\begin{proof}
The derivation of the expansions are found in \cite[Lemmas 3.1 and 3.2]{Es}. We also use the fact that the second fundamental form on $M$ vanishes for a CCE manifold $X$.
\end{proof}
\begin{lemma}[Lemma 2.1 in \cite{FF}]\label{lemma:Poi2}
Let $U \in H^{1,2}(B^N_+(0,R))$ where $B^N_+(0,R) = B^N(0,R) \cap \R^N_+$ for $R > 0$. If $U(x) = a(\theta)b(r)$ a.e. for $x \in \R^N_+$ with $r=|x| < R$ and $\theta = \frac{x}{|x|} \in \S^n_+$, then
\[\textup{div}\(x_N^{\mz} \nabla U\) = \frac{1}{r^n} \pa_r\(r^{n+1-2\ga} \pa_r b(r)\) \theta_N^{1-2\ga}a(\theta) + \frac{1}{r^{1+2\ga}} b(r) \textup{div}_{g_{\S^n}}\(\theta_N^{\mz} \nabla_{g_{\S^n}} a(\theta)\)\]
in the weak sense.
\end{lemma}
\begin{proof}[Proof of Proposition \ref{prop:Poi}]
The proof is decomposed into five parts.

\medskip \noindent \textsc{Step 1: Construction of $\mck_{\bgg,\rhog}^{\mz}$ and property} (a). We will present a proof motivated by \cite[Section 5]{HWY2} and \cite[Section 4]{MN}.

Let $x = r\theta \in \overline{\mcc^N(0,r_1)}$ be Fermi coordinates on $X$ around a fixed point $\sigma \in M$. By \eqref{eq:g+ ext},
\begin{multline}\label{eq:Poi41}
\begin{medsize}
\displaystyle \textup{div}_{\bgg}\(\rhog^{\mz} \nabla_{\bgg} U\) = \textup{div}\(x_N^{\mz} \nabla U\) + x_N^{\mz} \left[\sum_{i,j=1}^n \(\bgg^{ij}-\delta^{ij}\)\pa_{x_ix_j}U + \sum_{i,j=1}^n \pa_{x_i} \bgg^{ij} \pa_{x_j}U \right.
\end{medsize} \\
\begin{medsize}
\displaystyle \left. + \frac{1}{\sqrt{|\bgg|}} \sum_{i,j=1}^n \bgg^{ij}\pa_{x_i}\sqrt{|\bgg|}\pa_{x_j}U + \frac{1}{\sqrt{|\bgg|}}\pa_{x_N}\sqrt{|\bgg|}\pa_{x_N}U\right].
\end{medsize}
\end{multline}

\medskip
We first investigate the case when the metric $h_{\rhog}$ in \eqref{eq:g+ ext} is smooth and even in $\rhog$ for all orders. Let $\alpha \in \R$ and $a$ be a function on $\overline{\S^n_+}$ such that
\[(\mcd a)(\theta) := \(a, \nabla_{\bth}a, \nabla_{\bth}^2a, \theta_N\pa_{\theta_N}a, \theta_N\pa_{\theta_N}\nabla_{\bth}a,
\theta_N^2\pa_{\theta_N}^2a\)(\theta) \in \R^{N^2+2N+3} \quad \text{for } \theta \in \overline{\S^n_+}\]
has a finite $L^{\infty}(\S^n_+)$-norm. By employing \eqref{eq:degeq31}, Lemmas \ref{lemma:Poi1} and \ref{lemma:Poi2}, and \eqref{eq:Poi41}, we obtain
\begin{align*}
\rhog^{\mz} L^{\mz}_{\bgg,\rhog} \(r^{\alpha}a(\theta)\) &= -r^{\mz-2} \left[r^{\alpha} \left\{\text{div}_{g_{\S^n}}\(\theta_N^{\mz} \nabla_{g_{\S^n}} a(\theta)\) + \theta_N^{\mz} \alpha(n-2\ga+\alpha)a(\theta)\right\}\right. \\
&\hspace{55pt} \left.+ \theta_N^{\mz} O\(r^{\alpha+2}|\mcd a|\)\right].
\end{align*}

Starting with
\[a_0(\theta) = \theta_N^{2\ga},\quad a_1(\theta) = 0, \quad \text{and} \quad A_0(r,\theta) = A_1(r,\theta) = r^{-n}\theta_N^{2\ga},\]
we define $a_j$ and $A_j$ for an integer $j \ge 2$ by the equation
\begin{equation}\label{eq:aj}
\begin{cases}
\begin{medsize}
\displaystyle -\text{div}_{g_{\S^n}}\(\theta_N^{\mz} \nabla_{g_{\S^n}} a_j\) + \theta_N^{\mz} (n-j)(j-2\ga)a_j = \theta_N^{\mz} \left.\left[\(r^{\mz-2-n+j} \theta_N^{\mz}\)^{-1} \rhog^{\mz} L^{\mz}_{\bgg,\rhog}(A_{j-1})\right]\right|_{r=0}
\end{medsize}
&\begin{medsize}
\displaystyle \text{in } \S^n_+,
\end{medsize} \\
\begin{medsize}
\displaystyle a_j = 0
\end{medsize}
&\begin{medsize}
\displaystyle \text{on } \pa\S^n_+,
\end{medsize} \\
\begin{medsize}
a_j \in H^{1,2}\(\S^n_+;\theta_N^{\mz},g_{\S^n}\)
\end{medsize}
\end{cases}
\end{equation}
and the recurrence relation
\[A_j(x) := A_j(r,\theta) = A_{j-1}(r,\theta) + r^{-n+j}a_j(\theta) \quad \text{for } x = r\theta \in \overline{\mcc^N(0,r_1)}.\]
A direct computation yields
\[\rhog^{\mz} L^{\mz}_{\bgg,\rhog}(A_j) = r^{\mz-1-n+j} \theta_N^{\mz}\mce_j(r,\theta) \quad \text{for } j \in \N \cup \{0\}\]
where $\mce_j$ is a function in $\mcc^N(0,r_1)$ such that $\mce_j = O(\sum_{m=0}^{j-1} |\mcd a_m| + r|\mcd a_j|)$ for $j = 1$ and $O(r)$ for $j = 0$.
Hence the right-hand side of the first equation of \eqref{eq:aj} is well-behaved.

If $\ga \in (0,1) \setminus \{\frac{1}{2}\}$, Proposition \ref{prop:spherical} shows that
$(n-j)(j-2\ga) \notin \textup{Spec}(\textup{div}_{g_{\S^n}}(\theta_N^{\mz} \nabla_{g_{\S^n}} \cdot))$ for all $j \ge 2$. Thus, $a_j$ and $A_j$ are uniquely determined for all $j \in \N \cup \{0\}$.
By Lemma \ref{lemma:ereg2}, $a_j$, $\theta_N^{\mz}\pa_{\theta_N}a_j$, and their tangential derivatives of any order are H\"older continuous on $\overline{\S^n_+}$. 
Since $a_j = 0$ on $\pa\S^n_+$, it also holds that
\begin{equation}\label{eq:Poi47}
\left\|\theta_N^{-2\ga} \nabla_{\bth}^{\ell}a_j\right\|_{C^0(\overline{\S^n_+})}
\le C \left\|\theta_N^{\mz} \pa_{\theta_N}\nabla_{\bth}^{\ell}a_j\right\|_{C^0(\overline{\S^n_+})} \quad \text{for each } \ell \in \N \cup \{0\}.
\end{equation}
Using the above information and scrutinizing each term of $\mce_j$, we deduce that $\nabla_{\bx}^{\ell} \mce_j = O(r^{-\ell})$ for $\ell = 0, 1, 2$, which implies that
\begin{equation}\label{eq:Poi42}
\rhog^{\mz} L^{\mz}_{\bgg,\rhog}(A_{n+3}) = x_N^{\mz} r^2\mce_{n+3}(r,\theta) \quad \text{with } \nabla_{\bx}^{\ell} \(r^2\mce_{n+3}\) = O(1) \quad \text{for } \ell = 0, 1, 2.
\end{equation}

If $\ga = \frac{1}{2}$, namely, $\mz = 0$, then the above iteration scheme works only up to $j = n$, since $(n-(n+1))((n+1)-2\ga) = -n$ is the first eigenvalue of $\Delta_{g_{\S^n}}$. In this case, a log-order correction term leads to
\begin{align*}
L^0_{\bgg,\rhog}(A_{n+1} + |r\log r|a_*(\theta)) &= \frac{1}{r} \mce_n(0,\theta)
+ \frac{\log r}{r}[\Delta_{g_{\S^n}}a_*(\theta)+na_*(\theta)] \\
&\ + \frac{n+1}{r}a_*(\theta) - \frac{1}{r}[\Delta_{g_{\S^n}}a_{n+1}(\theta)+na_{n+1}(\theta)] + O(1).
\end{align*}
Let us write $\mce_n(0,\theta) = \mce_{n0}(\theta) + (n+1)\nu_n\theta_N$ where
\[\int_{\S^n_+} \mce_{n0}(0,\theta)\theta_N (dv_{g_{\S^n}})_{\theta} = 0 \quad \text{and} \quad
\nu_n := \frac{1}{n+1} \frac{\int_{\S^n_+}\mce_n(0,\theta)\theta_N (dv_{g_{\S^n}})_{\theta}}{\int_{\S^n_+} \theta_N^2 (dv_{g_{\S^n}})_{\theta}}.\]
According to the Fredholm alternative, the equation
\[\begin{cases}
- \Delta_{g_{\S^n}}a(\theta)-na(\theta) = -\mce_{n0}(0,\theta) &\text{in } \S^n_+,\\
a = 0 &\text{on } \pa\S^n_+,\\
\int_{\S^n_+} a(\theta)\theta_N (dv_{g_{\S^n}})_{\theta} = 0
\end{cases}\]
has a unique solution in $H^{1,2}(\S^n_+;\theta_N^{\mz},g_{\S^n})$, which we call $a_{n+1}$. We also set
\[a_*(\theta) = - \nu_n\theta_N \quad \text{and} \quad A_*(r,\theta) = A_{n+1}(r,\theta) + |r\log r|a_*(\theta).\]
Then
\begin{equation}\label{eq:Poi43}
L^0_{\bgg,\rhog}(A_*) = \mce_{n+1}(r,\theta) = O(1).
\end{equation}

Now, after choosing cut-off functions $\chi_{01} \in C^{\infty}_c(B^n(0,r_1))$ and $\chi_{02} \in C^{\infty}_c([0,r_1))$ such that $\chi_{01} = 1$ in $B^n(0,\frac{r_1}{2})$ and $\chi_{02} = 1$ on $[0,\frac{r_1}{2}]$,
we set
\begin{equation}\label{eq:chi0}
\chi_0(x) = \chi_{01}(\bx)\chi_{02}(x_N) \quad \text{for } x=(\bx,x_N) \in \mcc^N(0,r_1).
\end{equation}
If $\ga \in (0,1) \setminus \{\frac{1}{2}\}$, we then define $\mcr_{n+3}$ on $\ox$ as a solution to
\begin{equation}\label{eq:Poi44}
\begin{cases}
-\textup{div}_{\bgg}\(\rhog^{\mz} \nabla_{\bgg} \mcr_{n+3}\) + \dfrac{n-2\ga}{4(n+1-2\ga)} \rhog^{\mz} R^{\mz}_{\bgg,\rhog} \mcr_{n+3} = \rhog^{\mz} \mcq_{n+3} &\text{in } X,\\
\mcr_{n+3} = 0 &\text{on } M,\\
\mcr_{n+3} \in H^{1,2}(X;\rhog^{\mz},\bgg)
\end{cases}
\end{equation}
where
\begin{multline*}
\mcq_{n+3}(\xi) := -\left[\chi_0r^2\mce_{n+3} + A_{n+3} \Delta_{\bgg}\chi_0 + 2\la \nabla_{\bgg}A_{n+3}, \nabla_{\bgg}\chi_0 \ra_{\bgg} + \frac{\mz}{x_N} A_{n+3}\chi_{01}(\pa_{x_N}\chi_{02})\right](x) \\
\text{for } \xi = (\pi(\xi),\rhog(\xi)) \in \overline{B^n_{\bh}(\sigma,r_1)} \times [0,r_1] \text{ and } x = \(\exp_{\sigma}^{-1}(\pi(\xi)),\rhog(\xi)\) \in \overline{\mcc^N(0,r_1)}.
\end{multline*}
Note that $r^2 = |x|^2 = d_{\bh}(\xi,\sigma)^2+\rhog(\xi)^2$. By Proposition \ref{prop:ext3}, $\mcr_{n+3}$ is uniquely determined and meets the regularity condition specified in the statement.
If $\ga = \frac{1}{2}$, we define a function $\mcr_*$ on $\ox$ as a solution to \eqref{eq:Poi44} where $\mcq_{n+3}$ is replaced with
\[\mcq_*(\xi) := -\left[\chi_0\mce_{n+1} + A_* \Delta_{\bgg}\chi_0 + 2\la \nabla_{\bgg}A_*, \nabla_{\bgg}\chi_0 \ra_{\bgg} + \frac{\mz}{x_N} A_*\chi_{01}(\pa_{x_N}\chi_{02})\right](x).\]
The classical elliptic regularity implies that $\mcr_*$ satisfies the required regularity property.

For $(\xi,\sigma) \in \ox \times M$, let
\begin{equation}\label{eq:Poi45}
\mck_{\bgg,\rhog}^{\mz}(\xi,\sigma) = \begin{cases}
\displaystyle \ka_{n,\ga} \left[\chi_0(x)(A_{n+3})_{\sigma}(x) + (\mcr_{n+3})_{\sigma}(\xi)\right] &\text{if } \ga \in (0,1) \setminus \left\{\frac{1}{2}\right\},\\
\displaystyle \ka_{n,1/2} \left[\chi_0(x)(A_*)_{\sigma}(x) + (\mcr_*)_{\sigma}(\xi)\right] &\text{if } \ga = \frac{1}{2}
\end{cases}
\end{equation}
where $x = (\exp_{\sigma}^{-1}(\pi(\xi)),\rhog(\xi))$, which is well-defined since the support of $\chi_0$ is in $\overline{\mcc^N(0,r_1)}$.
The subscripts $\sigma$ are used to emphasize the dependence of $A_{n+3}$ and the other functions on $\sigma$.
Then, from \eqref{eq:Poi42}--\eqref{eq:Poi45}, we deduce that $\mck_{\bgg,\rhog}^{\mz}$ is a solution to the first equation of \eqref{eq:Poi01}.
Moreover, \eqref{eq:Poi45} boils down to \eqref{eq:Poi03} in $\overline{\mcc^N(0,\frac{r_1}{2})}$ where \eqref{eq:Poi06} and $a_{(n+1)1} = a_*$ hold.

\medskip
Let us turn to the case when the metric $h_{\rhog}$ has log-order terms in its expansion. In this case, it holds that \[\rhog^{\mz} L^{\mz}_{\bgg,\rhog}(A_{n-1}) = O\(x_N^{\mz}r^{-2}|\log r|\),\]
so the next-order term to $A_{n-1}$ must contain logarithms even if $\ga \in (0,1) \setminus \{\frac{1}{2}\}$. Using the identity
\begin{align*}
&\ \textup{div}\(x_N^{\mz} \nabla \(r^{\alpha}|\log r|^k a(\theta)\)\) \\
&= r^{\alpha+\mz-2} \left[\left\{\text{div}_{g_{\S^n}}\(\theta_N^{\mz} \nabla_{g_{\S^n}} a(\theta)\) + \theta_N^{\mz} \alpha(n-2\ga+\alpha)a(\theta)\right\} |\log r|^k \right. \\
&\hspace{55pt} \left. - k(n-2\ga+2\alpha)|\log r|^{k-1}\theta_N^{\mz}a(\theta) + k(k-1)|\log r|^{k-2} \theta_N^{\mz}a(\theta)\right]
\quad \text{for } \alpha,\, k \in \R
\end{align*}
and arguing as in the previous paragraphs, we observe that
\[\rhog^{\mz} L^{\mz}_{\bgg,\rhog}(A_n + |\log r|a_{n1}(\theta)) = O\(x_N^{\mz} r^{-1}|\log r|^{k_0}\) \quad \text{for some } k_0 \in \N.\]
By progressively extracting higher-order terms of $\mck_{\bgg,\rhog}^{\mz}(\cdot,\sigma)$, we can define
\begin{align}
&\ \begin{medsize}
\mck_{\bgg,\rhog}^{\mz}(\xi,\sigma)
\end{medsize} \label{eq:Poi04} \\
&= \begin{cases}
\begin{medsize}
\displaystyle \ka_{n,\ga} \Bigg[\chi_0(x)\Bigg\{(A_{n+4})_{\sigma}(x) + \sum_{j=n}^{n+4}\sum_{k=1}^{k_j} r^{-n+j}|\log r|^k(a_{jk})_{\sigma}(\theta)\Bigg\} + (\mcr_{n+4})_{\sigma}(\xi)\Bigg]
\end{medsize}
&\begin{medsize}
\text{if } \ga \in (0,1) \setminus \left\{\frac{1}{2}\right\},
\end{medsize} \\
\begin{medsize}
\displaystyle \ka_{n,1/2} \Bigg[\chi_0(x)\Bigg\{(A_{n+1})_{\sigma}(x) + \sum_{j=n}^{n+1}\sum_{k=1}^{k_j} r^{-n+j}|\log r|^k(a_{jk})_{\sigma}(\theta)\Bigg\} + (\mcr_{n+1})_{\sigma}(\xi)\Bigg]
\end{medsize}
&\begin{medsize}
\text{if } \ga = \frac{1}{2}
\end{medsize}
\end{cases} \nonumber
\end{align} 
for $(\xi,\sigma) \in \ox \times M$, where $x = (\exp_{\sigma}^{-1}(\pi(\xi)),\rhog(\xi))$.
Then it solves the first equation of \eqref{eq:Poi01}, and boils down to \eqref{eq:Poi03} in $\overline{\mcc^N(0,\frac{r_1}{2})}$.
It is worth mentioning that the number $k_j \in \N$ relies only on the metric $h_{\rhog}$ and is independent of a particular choice of $\sigma \in M$.

\medskip \noindent \textsc{Step 2: Property} (b). Assume that $\ga \in (0,1) \setminus \{\frac{1}{2}\}$. By Lemma \ref{lemma:ereg2} and the construction of $a_j$ and $a_{jk}$,
there exist numbers $C > 0$ and $\beta \in (0,1)$ depending only on $n$ and $\ga$ such that
\begin{multline}\label{eq:Poi46}
\sup_{\sigma \in M} \Bigg[\sum_{j=2}^{n+4} \left\|\big(\wmcd a_j\big)_{\sigma}\right\|_{C^{0,\beta}(\overline{\S^n_+})}
+ \sum_{j=n}^{n+4}\sum_{k=1}^{k_j} \left\|\big(\wmcd a_{jk}\big)_{\sigma}\right\|_{C^{0,\beta}(\overline{\S^n_+})}\Bigg] \le C \\
\quad \text{where } \wmcd a := (a, \nabla_{\bth}a, \theta_N^{\mz}\pa_{\theta_N}a).
\end{multline}
An analogous estimate holds for $\ga = \frac{1}{2}$.

We separately examine three distinct cases $\xi = (\pi(\xi),\rhog(\xi)) \in B^n_{\bh}(\sigma,\frac{r_1}{2}) \times [0,\frac{r_1}{2})$,
$\xi \in (M \setminus B^n_{\bh}(\sigma,\frac{r_1}{2})) \times [0,\frac{r_1}{2})$, and $\rho_g(\xi) \ge \frac{r_1}{2}$.
For each of these cases, we can verify \eqref{eq:Poi21}--\eqref{eq:Poi24} by utilizing \eqref{eq:Poi03}, \eqref{eq:Poi46},
and the uniform $C^1(B^n_{\bh}(\sigma,r_1))$-boundedness of the map $\exp_{\sigma}^{-1}$ with respect to $\sigma \in M$.
We leave the details to the reader.

\medskip \noindent \textsc{Step 3: Verification of \eqref{eq:Poibg}}. For any $f \in C^1(M)$, we will check that $\mck_{\bgg,\rhog}^{\mz}f$ solves \eqref{eq:degeq5} with $(\bg,\rho) = (\bgg,\rhog)$.

Thanks to the first equation of \eqref{eq:Poi01} and the classical elliptic regularity, the map $\xi \in X \mapsto \mck_{\bgg,\rhog}^{\mz}(\xi,\sigma)$ is smooth for each fixed $\sigma \in M$.
From this fact, we observe that $\mck_{\bgg,\rhog}^{\mz}f$ satisfies the first equation in \eqref{eq:degeq5}.

Suppose that $\xi \in X$ approaches to $\xi_0 \in M$. It holds that
\begin{multline}\label{eq:Poi31}
\left|\int_M \mck_{\bgg,\rhog}^{\mz}(\xi,\sigma)f(\sigma) (dv_{\bh})_{\sigma} - f(\xi_0)\right| \\
\le \left|\int_M \mck_0^{\mz}(\xi,\sigma)f(\sigma) (dv_{\bh})_{\sigma} - f(\xi_0)\right| + \int_M \left|\mck_{\bgg,\rhog}^{\mz}(\xi,\sigma)-\mck_0^{\mz}(\xi,\sigma)\right| |f(\sigma)| (dv_{\bh})_{\sigma}.
\end{multline}
The first term in the right-hand side of \eqref{eq:Poi31} is bounded by
\begin{multline*}
\left|\ka_{n,\ga} \int_{B^n(0,r_1)} \frac{x_N^{2\ga}}{(|\bx-\bw|^2+x_N^2)^{\frac{n+2\ga}{2}}} f(\bw) d\bw - f(0)\right| + o(1) \\
\le \ka_{n,\ga} \|f\|_{C^1(M)} \int_{B^n(0,r_1)} \frac{x_N^{2\ga}|\bw|}{(|\bx-\bw|^2+x_N^2)^{\frac{n+2\ga}{2}}} d\bw + o(1) = o(1).
\end{multline*}
Here, $o(1) \to 0$ as $\xi \to \xi_0$, and $\xi$ is identified with $x=(\bx,x_N) \in \overline{\mcc^N(0,r_1)}$ through Fermi coordinates on $\ox$ around $\xi_0$. 
Furthermore, the second term in the right-hand side of \eqref{eq:Poi31} converges to 0 as $\xi \to \xi_0$, because of \eqref{eq:Poi22} and
\[\int_M \frac{\rhog(\xi)^{2\ga}}{d_{\bgg}(\xi,\sigma)^{n+2\ga-2}} |f(\sigma)| (dv_{\bh})_{\sigma}
\le C\|f\|_{C^0(M)} \rhog(\xi)^{2\ga}.\]
Thus the second equation in \eqref{eq:degeq5} is fulfilled.

Next, for each $\xi = (\pi(\xi),\rhog(\xi)) \in M \times (0,r_1) \subset X$, we write
\begin{align}
(\mck_{\bgg,\rhog}^{\mz}f)(\xi) &= \int_M \mck_0^{\mz}(\xi,\sigma) [f(\sigma)-f(\pi(\xi))] (dv_{\bh})_{\sigma}
+ \int_M \left[\mck^{\mz}_{\bg,\rho}(\xi,\sigma) - \mck_0^{\mz}(\xi,\sigma)\right] f(\sigma) (dv_{\bh})_{\sigma} \nonumber \\
&\ + f(\pi(\xi)) \int_M \mck_0^{\mz}(\xi,\sigma) (dv_{\bh})_{\sigma} =: (\mcj_{11} + \mcj_{12} + \mcj_{13})(\xi). \label{eq:Poi32}
\end{align}
Then, by applying \eqref{eq:Poi02} and \eqref{eq:Poi24}, we compute
\begin{equation}\label{eq:Poi33}
\begin{aligned}
|\nabla \mcj_{11}(\xi)| &\le C\|f\|_{C^1(M)} \left[\int_M \left|\nabla_{\xi} \mck_0^{\mz}(\xi,\sigma)\right| d_{\bh}(\pi(\xi),\sigma) (dv_{\bh})_{\sigma} + 1\right] \\
&\le C\|f\|_{C^1(M)} \left[\int_M \frac{\rhog(\xi)^{2\ga-1}}{d_{\bgg}(\xi,\sigma)^{n+2\ga-1}} (dv_{\bh})_{\sigma} + 1\right] \\
&\le C\|f\|_{C^1(M)} \begin{cases}
\rhog(\xi)^{-\mz} &\text{for } \ga \in (0,\frac{1}{2}),\\
|\log \rhog(\xi)| &\text{for } \ga = \frac{1}{2},\\
1 &\text{for } \ga \in (\frac{1}{2},1)
\end{cases}
\end{aligned}
\end{equation}
and
\begin{equation}\label{eq:Poi34}
|\nabla \mcj_{12}(\xi)| \le C\|f\|_{C^0(M)} \int_M \frac{\rhog(\xi)^{2\ga-1}}{d_{\bgg}(\xi,\sigma)^{n+2\ga-2}} (dv_{\bh})_{\sigma} \le C\|f\|_{C^0(M)} \rhog(\xi)^{-\mz}.
\end{equation}
In order to estimate $|\nabla \mcj_{13}(\xi)|$, we associate $\xi \in B^n_{\bh}(\pi(\xi),r_1) \times (0,r_1)$ with $(0,x_N) \in \mcc^N(0,r_1)$
and $\sigma \in M$ with $\bw \in B^n(0,r_1)$ through Fermi coordinates $\ox$ around $(\pi(\xi),0)$. Since $\mck_{n,\ga}1 = 1$, we have
\begin{align*}
\left|\nabla_{\xi} \int_M \mck_0^{\mz}(\xi,\sigma) (dv_{\bh})_{\sigma}\right| &\le \int_{M \setminus B^n_{\bh}(\pi(\xi),r_1)} \left|(\nabla_{\xi} \mck_0^{\mz})(\xi,\sigma)\right| (dv_{\bh})_{\sigma} \\
&\ + \int_{\R^n \setminus B^n(0,r_1)} \left|(\nabla_x \mck_{n,\ga})((0,x_N),\bw)\right| d\bw \\
&\ + C\int_{B^n(0,r_1)} \left|(\nabla_x \mck_{n,\ga})((0,x_N),\bw)\right| |\bw|^2 d\bw \le C\rhog(\xi)^{-\mz}
\end{align*}
where we used \eqref{eq:Poi23} as well. It follows that
\begin{equation}\label{eq:Poi35}
|\nabla \mcj_{13}(\xi)| \le C\|f\|_{C^1(M)} + C\|f\|_{C^0(M)}\rhog(\xi)^{-\mz}.
\end{equation}
Putting \eqref{eq:Poi32}--\eqref{eq:Poi35} together, we see that
$\nabla(\mck_{\bgg,\rhog}^{\mz}f) \in L^2(X;\rhog^{\mz},\bgg)$.
Showing $\mck_{\bgg,\rhog}^{\mz}f \in L^2(X;\rhog^{\mz},\bgg)$ is simpler, so we omit it. It follows that $\mck_{\bgg,\rhog}^{\mz}f \in H^{1,2}(X;\rhog^{\mz},\bgg)$.

By virtue of the uniqueness assertion in Proposition \ref{prop:ext2}, a solution $U$ to \eqref{eq:degeq5} must satisfy $U = \mck_{\bgg,\rhog}^{\mz}f$, so $\mck_{\bgg,\rhog}^{\mz}$ is certainly a weighted Poisson kernel.

\medskip \noindent \textsc{Step 4: Property} (c). Let us prove \eqref{eq:Poi26}.
To this end, we examine each term consisting of $\mck_{\bgg,\rhog}^{\mz}$; refer to \eqref{eq:Poi45}--\eqref{eq:Poi04}. As a starting point, we consider
\begin{equation}\label{eq:Poi36}
\begin{aligned}
\chi_0(x)|x|^{-n}a_0\(\frac{x}{|x|}\) &= \chi_0(x) \frac{x_N^{2\ga}}{|x|^{n+2\ga}} \\
&= \chi_{01}\(\exp_{\sigma}^{-1}(\pi(\xi))\) \chi_{02}(\rhog(\xi)) \frac{\rhog(\xi)^{2\ga}}{[d_{\bh}(\pi(\xi),\sigma)^2+\rhog(\xi)^2]^{\frac{n+2\ga}{2}}}
\end{aligned}
\end{equation}
where $x = (\exp_{\sigma}^{-1}(\pi(\xi)),\rhog(\xi))$. The right-hand side of \eqref{eq:Poi36} is nonzero only if $d_{\bh}(\pi(\xi),\sigma) \in [0,r_1)$ and $\rhog(\xi) \in [0,r_1)$.
Hence, by adjusting the value of $r_1$ as needed, we ensure that its derivative with respect to $\sigma$ is well-defined and bounded by $C\rhog(\xi)^{2\ga}d_{\bgg}(\xi,\sigma)^{-(n+2\ga+1)}$.

We next treat
\begin{equation}\label{eq:Poi37}
\chi_0(x)|x|^{-n+2}(a_2)_{\sigma}\(\frac{x}{|x|}\) =
\frac{\chi_{01}\(\exp_{\sigma}^{-1}(\pi(\xi))\) \chi_{02}(\rhog(\xi))}{[d_{\bh}(\pi(\xi),\sigma)^2+\rhog(\xi)^2]^{\frac{n-2}{2}}}
(a_2)_{\sigma}\(\frac{\(\exp_{\sigma}^{-1}(\pi(\xi)),\rhog(\xi)\)}{[d_{\bh}(\pi(\xi),\sigma)^2+\rhog(\xi)^2]^{\frac{1}{2}}}\).
\end{equation}
Contrary to $a_0$, the function $(a_2)_{\sigma}$ depends on $\sigma$, so we have to verify the well-definedness of $\nabla_{\sigma}(a_2)_{\sigma}$ and determine its bound.
For $\sigma_1, \sigma_2 \in M$, we can derive an equation of $(a_2)_{\sigma_1}-(a_2)_{\sigma_2}$ from \eqref{eq:aj}.
An application of Lemma \ref{lemma:ereg2} to that equation yields numbers $C > 0$ and $\beta \in (0,1)$ depending only on $n$, $\ga$, $(\ox,\bgg)$, and $\rhog$ such that
\[\sum_{j=2}^{n+4} \left\|\big(\wmcd a_j\big)_{\sigma_1} - \big(\wmcd a_j\big)_{\sigma_2}\right\|_{C^{0,\beta}(\overline{\S^n_+})}
+ \sum_{j=n}^{n+4}\sum_{k=1}^{k_j} \left\|\big(\wmcd a_{jk}\big)_{\sigma_1} - \big(\wmcd a_{jk}\big)_{\sigma_2}\right\|_{C^{0,\beta}(\overline{\S^n_+})} \le Cd_{\bh}(\sigma_1,\sigma_2)\]
for any $\sigma_1, \sigma_2 \in M$; see \eqref{eq:Poi46} for the definition of $\wmcd$. Taking $\sigma_2 \to \sigma_1$ gives
\begin{equation}\label{eq:Poi38}
\sup_{\sigma \in M} \Bigg[\sum_{j=2}^{n+4} \left\|\wmcd \nabla_{\sigma} (a_j)_{\sigma}\right\|_{C^{0,\beta}(\overline{\S^n_+})}
+ \sum_{j=n}^{n+4}\sum_{k=1}^{k_j} \left\|\wmcd \nabla_{\sigma} (a_j)_{\sigma}\right\|_{C^{0,\beta}(\overline{\S^n_+})}\Bigg] \le C.
\end{equation}
By using \eqref{eq:Poi38} and considering \eqref{eq:Poi47} with the substitution of $\nabla_{\sigma}a_j$ for $a_j$, we see that the derivative of the right-hand side of \eqref{eq:Poi37} with respect to $\sigma$ is bounded by $C\rhog(\xi)^{2\ga}d_{\bgg}(\xi,\sigma)^{-(n+2\ga-1)}$.
Through similar reasoning, we can treat the terms of the form $\chi_0(x)|x|^{-n+j}(a_j)_{\sigma}(x/|x|)$ and $\chi_0(x)|x|^{-n+j}|\log|x||^k(a_{jk})_{\sigma}(x/|x|)$.

Lastly, employing \eqref{eq:Poi44} and \eqref{eq:extreg12}, we obtain $|\nabla_{\sigma}(\mcr_{n+3})_{\sigma}| \le C\rhog(\xi)^{2\ga}$; analogous estimates for $|\nabla_{\sigma}(\mcr_*)_{\sigma}|$, etc. are also available.
From the preceding discussion, we establish \eqref{eq:Poi26}.

In light of the previous analysis, it is clear that $\mck_{\bgg,\rhog}^{\mz}(\xi,\sigma)$ is continuous with respect to the $\sigma$-variable whenever $\rhog(\xi) > 0$.

\medskip \noindent \textsc{Step 5: Property} (d). Assume that $\lambda_1(-\Delta_{g^+}) > \frac{n^2}{4}-\ga^2$. We know that there exists the adapted defining function $\rho^*$ of $(M,\bh)$.

Let $f$ be a positive smooth function on $M$. Propositions \ref{prop:ext} and \ref{prop:ext2} tell us that
if $u$ is a solution to \eqref{eq:degeq32} with $Q = 0$, then $U^* := (\rho^*)^{s-n}u$ solves \eqref{eq:degeq5} with $(\bg,\rho) = ((\rho^*)^2g^+,\rho^*)$ (so that $R^{\mz}_{\bg,\rho} = 0$ in $X$)
and $\mck_{\bgg,\rhog}^{\mz}f = (\rho^*/\rhog)^{n-s} U^*$ on $\ox$. Also, by \eqref{eq:rho*_exp}, $U^* = f > 0$ on $M$.
Testing \eqref{eq:degeq5} with $(U^*)_- := \max\{-U^*,0\}$ and applying the classical strong maximum principle, we deduce $U^* > 0$ on $\ox$. It follows that 
$\mck_{\bgg,\rhog}^{\mz}f > 0$ on $\ox$.

Given any $\sigma_0 \in M$, let $\{f_j\}_{j \in \N}$ be an approximation to the identity on $M$ such that $f_j$ is positive and smooth on $M$ for all $j \in \N$
and $f_j \rightharpoonup \delta_{\sigma_0}$ in $\mcm(M)$ as $j \to \infty$ where $\mcm(M)$ is the space of bounded nonnegative measures on $M$.
Invoking Step 4, we easily verify that $(\mck_{\bgg,\rhog}^{\mz}f_j)(\xi) \to \mck_{\bgg,\rhog}^{\mz}(\xi,\sigma_0)$ for all $\xi \in X$.
From the inequality $\mck_{\bgg,\rhog}^{\mz}f_j > 0$ in $\ox$, we conclude that \eqref{eq:Poi25} is true.

For a fixed $\sigma \in M$, classical Harnack's inequality implies that the zero set of $\mck_{\bgg,\rhog}^{\mz}(\cdot,\sigma)$ is both open and closed in $X$.
On the other hand, we know from \eqref{eq:Poi03} that $\mck_{\bgg,\rhog}^{\mz}(\xi,\sigma) > 0$ if $\xi \in X$ is sufficiently close to $\sigma$.
Consequently, if $\xi$ and $\sigma$ are points in the same connected component of $\ox$, then $\mck_{\bgg,\rhog}^{\mz}(\xi,\sigma) > 0$.
Otherwise, we infer from \eqref{eq:Poi01} and Proposition \ref{prop:ext2} that $\mck_{\bgg,\rhog}^{\mz}(\xi,\sigma) = 0$ for all $(\xi,\sigma) \in X \times M$.
\end{proof}
\begin{remark}\label{remark:Poi}
We have two remarks for Proposition \ref{prop:Poi}.

\medskip \noindent
(1) By exploiting the conformal covariance property \eqref{eq:wcl2}, one can define a weighted Poisson kernel $\mck^{\mz}_{\bg,\rho}$
for an arbitrary pair $(\bg,\rho)$ of a defining function $\rho$ of $M$ and the compactified metric $\bg = \rho^2g^+$ by
\[\mck^{\mz}_{\bg,\rho}(\xi,\sigma) = \left[\frac{\rhog(\xi)}{\rho(\xi)}\right]^{\frac{n-2\ga}{2}}
\mck_{\bgg,\rhog}^{\mz}(\xi,\sigma) \quad \textup{for } (\xi,\sigma) \in \ox \times M.\]

\medskip \noindent
(2) Assume that $\ga \in (0,1) \setminus \left\{\frac{1}{2}\right\}$. While defining the kernel $\mck_{\bgg,\rhog}^{\mz}$ using \eqref{eq:Poi04}, it is possible to further decompose $(\mcr_{n+4})_{\sigma}$ into
\begin{equation}\label{eq:Poi05}
(\mcr_{n+4})_{\sigma}(\xi)
= \frac{\chi_0(x)}{r^n} \Bigg[\sum_{j=n+5}^{\ell} r^j(a_j)_{\sigma}(\theta)
+ \sum_{j=n+5}^{\ell}\sum_{k=1}^{k_j} \(r^j|\log r|^k(a_{jk}))_{\sigma}(\theta)\)_{\sigma}\Bigg] + (\mcr_{\ell})_{\sigma}(\xi)
\end{equation}
where $\ell$ is any integer satisfying $\ell \ge n+5$.
This decomposition is achieved by employing the uniqueness assertion for \eqref{eq:degeq5} in Proposition \ref{prop:ext2}.
Identity \eqref{eq:Poi05} is useful when one aims to control higher-order tangential derivatives of the kernel with respect to the both variables $\xi$ and $\sigma$,
since Lemma \ref{lemma:ereg}(a) guarantees that $(\mcr_{\ell})_{\sigma}$ has better regularity than $(\mcr_{n+4})_{\sigma}$.
For example, the $m$-th tangential derivatives of $(\mcr_{\ell})_{\sigma}$ with respect to $\xi$ exist and are H\"older continuous for all $m = 0,\ldots,\ell-(n+2)$.
The same idea can be applied when $\ga = \frac{1}{2}$.
\end{remark}

\subsection{Weighted isoperimetric ratio}
We are concerned with the relationship among the variational problem \eqref{eq:thetanw}, the Chang-Gonz\'alez extension \eqref{eq:degeq31} and \eqref{eq:degeq41}, and inequalities \eqref{eq:main11} and \eqref{eq:main12}.

\medskip
Suppose that $\th = f^{\frac{4}{n-2\ga}} \bh$ on $M$ for a positive function $f \in C^{\infty}(M)$, $\trh \in \Xi_{g^+,\th}$,
$\rhog$ is the geodesic defining function of $(M,\bh)$, $\bgg = \rhog^2 g^+$, and $\tg = \trh^2 g^+ = (\trh/\rhog)^2 \bgg \in \mcb_{g^+,\,[\bh]}$; refer to \eqref{eq:mcb}.
By \eqref{eq:wcl} and \eqref{eq:wcl2}, the function $U = (\trh/\rhog)^{\frac{n-2\ga}{2}}$ on $\ox$ is a positive solution to
\begin{equation}\label{eq:wir3}
\begin{cases}
L^{\mz}_{\bgg,\rhog} U = U^{\frac{n-2\ga+4}{n-2\ga}} \cdot \dfrac{n-2\ga}{4(n+1-2\ga)} R^{\mz}_{\tg,\trh} = 0 &\textup{in } X,\\
U = f &\textup{on } M.
\end{cases}
\end{equation}

We claim that $U \in H^{1,2}(X;\rhog^{\mz},\bgg)$. As mentioned in \cite[Lemma 2.3]{GQ}, if we write $\trh_\geo = e^w\rhog$ on $\ox$, then $w$ solves a first-order partial differential equation
\begin{equation}\label{eq:wir4}
\begin{cases}
\displaystyle \frac{\pa w}{\pa \rhog} + \frac{\rhog}{2} \left[\(\frac{\pa w}{\pa \rhog}\)^2 + |\nabla w|^2_{h_{\rhog}}\right] = 0 &\text{near } M,\\
e^{2w} = f^{\frac{4}{n-2\ga}} &\text{on } M
\end{cases}
\end{equation}
where the metric $h_{\rhog}$ on $M$ is defined by \eqref{eq:g+ ext}. Because this equation is non-characteristic, it has a unique solution near $M$, which must be $w$.
Also, because $\bgg$ belongs to the class $C^{3,\beta}$ for any $\beta \in (0,1)$, we have that $w \in C^3(\ox)$ (after extending $w$ suitably) and thus $\trh_\geo/\rhog \in C^3(\ox)$. 
Meanwhile, \eqref{eq:Xi} implies that
\[U = \(\frac{\trh}{\trh_\geo}\)^{\frac{n-2\ga}{2}} \(\frac{\trh_\geo}{\rhog}\)^{\frac{n-2\ga}{2}}
= \left[1 + \Phi \trh_\geo^{2\ga} + o'\(\trh_\geo^{2\ga}\)\right]^{\frac{n-2\ga}{2}} \(\frac{\trh_\geo}{\rhog}\)^{\frac{n-2\ga}{2}} \quad \text{near } M.\]
The assertion now follows.

Now, by virtue of Proposition \ref{prop:ext2}, we infer that $U = \mck_{\bgg,\rhog}^{\mz}f$ and
\begin{equation}\label{eq:wir0}
I_{n,\ga}(\ox,\tg,\trh) = \frac{\int_X \trh^{\mz} dv_{\tg}}{(\int_M dv_{\th})^{\frac{n-2\ga+2}{n}}}
= \left[\frac{\left\|\mck_{\bgg,\rhog}^{\mz}f\right\|_{L^{\frac{2(n-2\ga+2)}{n-2\ga}}(X;\rhog^{\mz},\bgg)}} {\|f\|_{L^{\frac{2n}{n-2\ga}}(M,\bh)}}\right]^{\frac{2(n-2\ga+2)}{n-2\ga}}.
\end{equation}
From this, we deduce the following result.
\begin{lemma}\label{lemma:wir}
Assume that $n \in \N$, $n > 2\ga$, $\ga \in (0,1)$, and $\lambda_1(-\Delta_{g^+}) > \frac{n^2}{4}-\ga^2$. Let
\begin{equation}\label{eq:mcs}
\begin{aligned}
\mcs_1 &= \sup\left\{\frac{\left\|\mck_{\bgg,\rhog}^{\mz}f\right\|_{L^{\frac{2(n-2\ga+2)}{n-2\ga}}(X;\rhog^{\mz},\bgg)}} {\|f\|_{L^{\frac{2n}{n-2\ga}}(M,\bh)}}: f \in C^{\infty}(M),\, f > 0 \textup{ on } M\right\},\\
\mcs_2 &= \sup\left\{\frac{\left\|\mck_{\bgg,\rhog}^{\mz}f\right\|_{L^{\frac{2(n-2\ga+2)}{n-2\ga}}(X;\rhog^{\mz},\bgg)}} {\|f\|_{L^{\frac{2n}{n-2\ga}}(M,\bh)}}: f \in L^{\frac{2n}{n-2\ga}}(M,\bh) \setminus \{0\}\right\},\\
\mcs_3 &= \sup\left\{\left\|\mck_{\bgg,\rhog}^{\mz}f\right\|_{L^{\frac{2(n-2\ga+2)}{n-2\ga}}(X;\rhog^{\mz},\bgg)}: \|f\|_{L^{\frac{2n}{n-2\ga}}(M,\bh)} = 1\right\}.
\end{aligned}
\end{equation}
Then it holds that
\begin{equation}\label{eq:wir1}
\left[\Theta_{n,\ga}(X,g^+,[\bh])\right]^{\frac{n-2\ga}{2(n-2\ga+2)}} = \mcs_1 = \mcs_2 = \mcs_3 < \infty.
\end{equation}
In particular,
\begin{equation}\label{eq:wir2}
\Theta_n\(\R^N_+, g_\ph^+, \left[4(1+|\bx|^4)^{-2}g_{\R^N}\right]\) = \Theta_{n,\ga}\(\B^N,g_\pb^+,[g_{\S^n}]\) = C_{n,\ga}^{\frac{2(n-2\ga+2)}{n-2\ga}}
\end{equation}
where $\bx \in \R^N_+$ and $C_{n,\ga} > 0$ is the constant in \eqref{eq:main11} and \eqref{eq:main12}.
\end{lemma}
\begin{proof}
In Lemma \ref{lemma:Poiest2}, we will show that $\mcs_3 < \infty$. The equality $\mcs_2$ = $\mcs_3$ is a consequence of a simple scaling argument.
In addition, $\mcs_1 = \mcs_2$ follows from \eqref{eq:Poi25} and the fact that a nonzero nonnegative function in $L^{\frac{2n}{n-2\ga}}(M)$ can be approximated with arbitrary precision by positive smooth functions on $M$.

We next verify the first equality in \eqref{eq:wir1}. By the paragraphs before the statement of Lemma \ref{lemma:wir}, it holds that $[\Theta_{n,\ga}(X,g^+,[\bh])]^{\frac{n-2\ga}{2(n-2\ga+2)}} \le \mcs_1$.
Hence it suffices to check the reverse inequality.
Given a positive function $f \in C^{\infty}(M)$, we set $U = \mck_{\bgg,\rhog}^{\mz}f$ on $\ox$, $\tg = U^{\frac{4}{n-2\ga}} \bgg$ on $\ox$, and $\th = \tg|_{TM} = f^{\frac{4}{n-2\ga}} \bh \in [\bh]$.
By Proposition \ref{prop:Poi}(d) and \eqref{eq:wir3}, 
$\trh = U^{\frac{2}{n-2\ga}}\rhog$ is a defining function of $M$ and $R^{\mz}_{\tg,\trh} = 0$ in $X$.
Because $U \in H^{1,2}(X;\rhog^{\mz},\bgg)$ owing to Step 3 of the proof of Proposition \ref{prop:Poi} (see also Definition \ref{def:wPk}), Proposition \ref{prop:ext2} implies that $U = F + G\rhog^{2\ga}$ on $\ox$ with $F,\, G \in C^2(\ox)$. We also have that
\[F = f + O'(\rhog^2) \text{ and } G = g^{(0)} + O'(\rhog^2) \quad \text{for some } g^{(0)} \in C^{\infty}(M).\] 
Here, $O'(\rhog^2)$ denotes a function on $\ox$ such that both the function and its tangential derivatives of any order are bounded by $C\rhog^2$, and its normal derivative is bounded by $C\rhog$ for some $C > 0$.
Equation \eqref{eq:wir4} gives that $\frac{\pa w}{\pa \rhog} = 0$ on $M$, so
\[\frac{\trh_\geo}{\rhog} = e^w = f^{\frac{2}{n-2\ga}} + O'\(\rhog^2\) \quad \text{near } M.\] 
Thus
\[\begin{aligned}
\trh &= \rhog \(F + G\rhog^{2\ga}\)^{\frac{2}{n-2\ga}} = \rhog f^{\frac{2}{n-2\ga}} \left[1 + \frac{2}{n-2\ga} \frac{g^{(0)}}{f} \rhog^{2\ga} + o'\(\rhog^{2\ga}\)\right] \\
&= \trh_\geo \left[1 + \frac{2}{n-2\ga} g^{(0)}f^{-\frac{n+2\ga}{n-2\ga}} \trh_\geo^{2\ga} + o'\(\trh_\geo^{2\ga}\)\right]
\end{aligned} \quad \text{near } M,\]
so $\trh \in \Xi_{g^+,\th}$. Consequently, $\tg \in \mcb_{g^+,\,[\bh]}$ and \eqref{eq:wir0} yields that $[\Theta_{n,\ga}(X,g^+,[\bh])]^{\frac{n-2\ga}{2(n-2\ga+2)}} \ge \mcs_1$.

\medskip
Finally, \eqref{eq:wir2} follows from \eqref{eq:main11}, \eqref{eq:main12}, \eqref{eq:wir1}, and the relation $\phi_\m^*g_\pb^+ = g_\ph^+$.
\end{proof}

\section{Results on the upper half-space}\label{sec:main11}
\subsection{Basic inequalities}
We derive inequality \eqref{eq:main11} under a broader setting.
\begin{lemma}\label{lemma:Poiest1}
Assume that $n \in \N$ and $\ga \in (-\infty,1)$. There exists a constant $C > 0$ depending only on $n$, $\ga$, and $p$ such that
\begin{equation}\label{eq:Poiest11}
\|\mck_{n,\ga}f\|_{L^{\frac{(n-2\ga+2)p}{n}}(\R^N_+;x_N^{\mz})} \le C\|f\|_{L^p(\R^n)}
\end{equation}
for all $p \in (1,\infty]$ and $f \in L^p(\R^n)$, and
\begin{equation}\label{eq:Poiest12}
\|\mck_{n,\ga}f\|_{L^{\frac{(n-2\ga+2)p}{n}}(\R^N_+;x_N^{\mz})} \le C\|f\|_{L^{p,\frac{(n-2\ga+2)p}{n}}(\R^n)}
\end{equation}
for all $p \in (1,\infty)$ and $f \in L^p(\R^n)$.
\end{lemma}
\begin{proof}
Following the proof of \cite[Proposition 2.1]{HWY}, we deduce that there exists constant $C > 0$ depending only on $n$ and $\ga$ such that
\begin{equation}\label{eq:west}
\|\mck_{n,\ga}f\|_{L^{\frac{n-2\ga+2}{n}}_w(\R^N_+;x_N^{\mz})}\le C\|f\|_{L^1(\R^n)} \quad \textup{for all } f \in L^1(\R^n).
\end{equation}
On the other hand, using \eqref{eq:gaharext} and $\mck_{n,\ga}1 = 1$, it is plain to see that
\begin{equation}\label{eq:Linfty}
\|\mck_{n,\ga}f\|_{L^{\infty}(\R^N_+;x_N^{\mz})} \le \|f\|_{L^{\infty}(\R^n)} \quad \textup{for all } f \in L^{\infty}(\R^n).
\end{equation}
Estimates \eqref{eq:Poiest11} and \eqref{eq:Poiest12} can be derived from \eqref{eq:west}, \eqref{eq:Linfty}, and the standard and off-diagonal Marcinkiewicz interpolation theorem. 
\end{proof}

\begin{lemma}\label{lemma:nonvanish}
Let $n \in \N$, $\ga \in (-\infty,1)$, and $p \in (1,\infty)$. If $f \in L^p(\R^n)$ is a nonnegative radial function, then
\[\|\mck_{n,\ga}f\|_{L^{\frac{(n-2\ga+2)p}{n}}(\R^N_+;x_N^{\mz})} \le C \left[\sup_{r>0}\(r^{\frac{n}{p}}f(r)\)\right]^{\frac{2(1-\ga)}{n-2\ga+2}} \|f\|_{L^p(\R^n)}^{\frac{n}{n-2\ga+2}}\]
where $C > 0$ depends only on $n$, $\ga$, and $p$. Here, we wrote $f(|\bx|) = f(x)$.
\end{lemma}
\begin{proof}
By the radial symmetry of $f$,
\[f(\bx) \le |\bx|^{-n/p} \sup_{r>0}\(r^{\frac{n}{p}}f(r)\) \quad \textup{for } \bx \in \R^n \quad \textup{so that} \quad \|f\|_{L^{p,\infty}(\R^n)} \le C\sup_{r>0}\(r^{\frac{n}{p}}f(r)\).\]
Then, from \eqref{eq:Poiest12} and the interpolation inequality, we see that
\begin{align*}
\|\mck_{n,\ga}f\|_{L^{\frac{(n-2\ga+2)p}{n}}(\R^N_+;x_N^{\mz})} &\le C\|f\|_{L^{p,\frac{(n-2\ga+2)p}{n}}(\R^n)} \le C\|f\|_{L^p(\R^n)}^{\frac{n}{n-2\ga+2}}\|f\|_{L^{p,\infty}(\R^n)}^{\frac{2(1-\ga)}{n-2\ga+2}}\\
&\le C\left[\sup_{r>0}\(r^{\frac{n}{p}}f(r)\)\right]^{\frac{2(1-\ga)}{n-2\ga+2}} \|f\|_{L^p(\R^n)}^{\frac{n}{n-2\ga+2}}. \qedhere
\end{align*}
\end{proof}

\subsection{Proof of Theorem \ref{thm:main11}}
We consider the variational problem
\begin{equation}\label{eq:varprob}
s_{n,\ga,p} := \sup\left\{\|\mck_{n,\ga}f\|_{L^{\frac{(n-2\ga+2)p}{n}}(\R^N_+;x_N^{\mz})}^{\frac{(n-2\ga+2)p}{n}}: \|f\|_{L^p(\R^n)}=1\right\}.
\end{equation}
Thanks to Lemma \ref{lemma:Poiest1}, it holds that $s_{n,\ga,p} < \infty$ if $n \in \N$, $\ga \in (-\infty,1)$, and $p \in (1,\infty)$.
In the next theorem, we establish the existence and classification results for maximizers of \eqref{eq:varprob}, by utilizing the symmetrization argument presented in \cite[Section 4]{HWY}. 
Note that Theorem \ref{thm:main11} is a specific case of this theorem combined with \eqref{eq:Poiest11}.
\begin{theorem}\label{thm:main11g}
Let $n \in \N$, $\ga \in (-\frac{n}{2p},1)$, and $p \in (1,\infty)$. Then there exists $f \in L^p(\R^n)$ with $\|f\|_{L^p(\R^n)}=1$ such that
\[\|\mck_{n,\ga}f\|_{L^{\frac{(n-2\ga+2)p}{n}}(\R^N_+;x_N^{\mz})}^{\frac{(n-2\ga+2)p}{n}} = s_{n,\ga,p}.\]
Additionally, up to a multiplication by a nonzero constant, every maximizer $f$ of \eqref{eq:varprob} is nonnegative, radially symmetric with respect to some point, strictly decreasing in the radial direction, and it satisfies
\begin{equation}\label{eq:main11g2}
f(\bw)^{p-1} = \int_0^{\infty} \int_{\R^n} \frac{x_N}{(|\bx-\bw|^2+x_N^2)^{\frac{n+2\ga}{2}}} (\mck_{n,\ga}f)(\bx,x_N)^{\frac{(n-2\ga+2)p}{n}-1} d\bx dx_N \quad \text{for } \bw \in \R^n.
\end{equation}
In particular, if $n > 2\ga$, $p=\frac{2n}{n-2\ga}$, and $\ga \in (0,1)$, then all maximizers $f$ of \eqref{eq:varprob} assume the form
\[f(\bx) = \pm c_0\(\frac{\lambda}{\lambda^2 + |\bx-\bx_0|^2}\)^{\frac{n-2\ga}{2}} \quad \text{for all } \bx \in \R^n\]
where $c_0 \in \R$ is determined by $n$ and $\ga$, and $\lambda > 0$ and $\bx_0 \in \R^n$ are arbitrary.
\end{theorem}
\begin{remark}
For a given $n \in \N$, $\ga \in (0,1)$, and $p \in (1,\infty)$, the argument in Subsection \ref{subsec:reg} will demonstrate that every maximizer $f \in L^p(\R^n)$ belongs to $C^{\infty}(\R^n)$.
\end{remark}

\begin{proof}[Proof of Theorem \ref{thm:main11g}]
For $\bx \in \R^n$ and $x_N > 0$, we set $\mck_{n,\ga,x_N}(\bx) = \mck_{n,\ga}((\bx,x_N),0)$; see \eqref{eq:gaharext2}.
Let also $f^*$ be the symmetric decreasing rearrangement of $f$. Recall that if $f \in L^p(\R^n)$, then $\|f\|_{L^p(\R^n)}=\|f^*\|_{L^p(\R^n)}$. By Riesz's rearrangement inequality and the duality argument,
\begin{equation}\label{eq:symm}
\begin{aligned}
\|\mck_{n,\ga}f\|_{L^{\frac{(n-2\ga+2)p}{n}}(\R^N_+;x_N^{\mz})}^{\frac{(n-2\ga+2)p}{n}} &= \int_0^{\infty} \|\mck_{n,\ga,x_N}* f\|_{L^{\frac{(n-2\ga+2)p}{n}}(\R^n)}^{\frac{(n-2\ga+2)p}{n}} x_N^{\mz} dx_N \\
&\le \int_0^{\infty} \|\mck_{n,\ga,x_N}* f^*\|_{L^{\frac{(n-2\ga+2)p}{n}}(\R^n)}^{\frac{(n-2\ga+2)p}{n}} x_N^{\mz} dx_N \\
&= \|\mck_{n,\ga}f^*\|_{L^{\frac{(n-2\ga+2)p}{n}}(\R^N_+;x_N^{\mz})}^{\frac{(n-2\ga+2)p}{n}}.
\end{aligned}
\end{equation}
Thus there exists a maximizing sequence $\{f_i\}_{i=1}^{\infty} \subset L^p(\R^n)$ for problem \eqref{eq:varprob}
such that $f_i$ is nonnegative radial nonincreasing for each $i \in \N$. 

We claim that $f_i \to f$ a.e. for some nonnegative radial nonincreasing function $f \in L^p(\R^n) \setminus \{0\}$.
Indeed, by Lemma \ref{lemma:nonvanish} and the fact that $\{f_i\}_{i=1}^{\infty}$ is a maximizing sequence for \eqref{eq:varprob},
there exists a constant $C_0 > 0$ independent of $i \in \N$ such that $\sup_{r>0}(r^{n/p}f_i(r)) \ge C_0$ for all $i$.
Let us choose a number $r_i > 0$ such that $r_i^{n/p}f_i(r_i)\ge \frac{C_0}{2}$.
If $g(\bx) := r^{n/p}f(r\bx)$ for $\bx \in \R^n$, then $(\mck_{n,\ga}g)(x) = r^{n/p}(\mck_{n,\ga}f)(rx)$ for $x \in \R^N_+$ so that
\[\|g\|_{L^p(\R^n)}=\|f\|_{L^p(\R^n)} \quad \textup{and} \quad \|\mck_{n,\ga}g\|_{L^{\frac{(n-2\ga+2)p}{n}}(\R^N_+;x_N^{\mz})} = \|\mck_{n,\ga}f\|_{L^{\frac{(n-2\ga+2)p}{n}}(\R^N_+;x_N^{\mz})}.\]
Accordingly, by replacing $f_i(\cdot)$ by $r_i^{n/p} f_i(r_i \cdot)$ (we still denote it by $f_i$),
we obtain a sequence $\{f_i\}_{i=1}^{\infty}$ of nonnegative radially nonincreasing functions such that $f_i(1)\ge \frac{C_0}{2}$,
\[\|f_i\|_{L^p(\R^n)}=1, \quad \textup{and} \quad \|\mck_{n,\ga}f_i\|_{L^{\frac{(n-2\ga+2)p}{n}}(\R^N_+;x_N^{\mz})}^{\frac{(n-2\ga+2)p}{n}} \to s_{n,\ga,p} \quad \textup{as } i\to \infty.\]
Since $\|f_i\|_{L^p(\R^n)}=1$ and $f_i$ is nonnegative radial nonincreasing, we have $f_i(|\bx|) \le |\B^n|^{-1/p}|\bx|^{-n/p}$ for $\bx \in \R^n$.
Helly's selection principle then implies that $f_i \to f$ a.e., up to a subsequence, for some nonnegative radial nonincreasing function $f \in L^p(\R^n) \setminus \{0\}$.

Applying the Br\'ezis-Lieb lemma, we discover
\begin{equation}\label{f conv}
\lim_{i\to \infty}\|f_i-f\|_{L^p(\R^n)}^p = 1-\|f\|_{L^p(\R^n)}^p.
\end{equation}
For each $x \in \R^N_+$ and $\ga > -\frac{n}{2p}$, the function $\bw \mapsto \mck_{n,\ga}(x,\bw)|\bw|^{-n/p}$
is integrable on $\R^n$. Hence the dominated convergence theorem yields that $(\mck_{n,\ga}f_i)(x) \to (\mck_{n,\ga}f)(x)$ a.e. as $i \to \infty$. One more application of the Br\'ezis-Lieb lemma gives
\begin{align*}
s_{n,\ga,p} &= \|\mck_{n,\ga}f_i\|_{L^{\frac{(n-2\ga+2)p}{n}}(\R^N_+;x_N^{\mz})}^{\frac{(n-2\ga+2)p}{n}} + o(1)\\
&=\|\mck_{n,\ga}f\|_{L^{\frac{(n-2\ga+2)p}{n}}(\R^N_+;x_N^{\mz})}^{\frac{(n-2\ga+2)p}{n}} + \|\mck_{n,\ga}(f_i-f)\|_{L^{\frac{(n-2\ga+2)p}{n}}(\R^N_+;x_N^{\mz})}^{\frac{(n-2\ga+2)p}{n}} + o(1)\\
&\le s_{n,\ga,p} \(\|f\|_{L^p(\R^n)}^{\frac{(n-2\ga+2)p}{n}} + \|f_i-f\|_{L^p(\R^n)}^{\frac{(n-2\ga+2)p}{n}}\) + o(1) \quad \textup{as } i \to \infty.
\end{align*}
We get from \eqref{f conv} that
\[1 \le \|f\|_{L^p(\R^n)}^{\frac{(n-2\ga+2)p}{n}} + \(1-\|f\|_{L^p(\R^n)}^p\)^{\frac{n-2\ga+2}{n}}.\]
In view of $p>1$, $\frac{n-2\ga+2}{n}>1$ for $\ga<1$, and $0 < \|f\|_{L^p(\R^n)} \le 1$, we obtain that $\|f\|_{L^p(\R^n)}=1$. By \eqref{f conv} again, $f_i \to f$ in $L^p(\R^n)$ as $i \to \infty$.
Lemma \ref{lemma:Poiest1} then implies that $\mck_{n,\ga}f_i \to \mck_{n,\ga}f$ in $L^{\frac{(n-2\ga+2)p}{n}}(\R^N_+;x_N^{\mz})$, so $f$ is a maximizer of problem \eqref{eq:varprob}.

\medskip
We next prove that any maximizer of \eqref{eq:varprob} is nonnegative, radially symmetric with respect to some point, and strictly decreasing along the radial direction, up to a scaling by a nonzero constant.
If $f \in L^p(\R^n)$ is a maximizer of \eqref{eq:varprob}, then so is $|f|$, whence
\[\|\mck_{n,\ga}f\|_{L^{\frac{(n-2\ga+2)p}{n}}(\R^N_+;x_N^{\mz})} = \|\mck_{n,\ga}|f|\|_{L^{\frac{(n-2\ga+2)p}{n}}(\R^N_+;x_N^{\mz})}.\]
Since $|(\mck_{n,\ga}f)(x)|\le (\mck_{n,\ga}|f|)(x)$ for $x \in \R^N_+$, we see that
\[|(\mck_{n,\ga}f)(x)|= (\mck_{n,\ga}|f|)(x) \quad \textup{for a.e. } x \in \R^N_+.\]
This means that either $f \ge 0$ or $f \le 0$ on $\R^n$. Suppose that $f \ge 0$ on $\R^n$. Then it satisfies the Euler-Lagrange equation
\begin{equation}\label{eq:EL}
f(\bw)^{p-1} = \int_0^{\infty} \int_{\R^n} \frac{x_N^{2\ga}}{(|\bx-\bw|^2+x_N^2)^{\frac{n+2\ga}{2}}}(\mck_{n,\ga}f)(\bx,x_N)^{\frac{(n-2\ga+2)p}{n}-1} d\bx\, x_N^{\mz} dx_N
\end{equation}
for $\bw \in \R^n$, after scaling by a positive constant. Meanwhile, thanks to \eqref{eq:symm} and the fact that $f \in L^p(\R^n)$ is a maximizer of \eqref{eq:varprob},
\[\|\mck_{n,\ga,x_N}* f\|_{L^{\frac{(n-2\ga+2)p}{n}}(\R^n)} =\|\mck_{n,\ga,x_N}* f^*\|_{L^{\frac{(n-2\ga+2)p}{n}}(\R^n)} \quad \textup{for all } x_N>0.\]
This implies that $f(\bx)=f^*(\bx-\bx_0)$ for some $\bx_0\in \R^n$; refer to e.g. \cite[Theorem 3.9]{LL}. 
Without loss of generality, we may assume that $f$ is radial nonincreasing and satisfies \eqref{eq:EL}.
We observe that $\mck_{n,\ga,x_N}(\bx)$ is strictly decreasing in $|\bx|$ for each $x_N>0$, and $(\mck_{n,\ga}f)(\bx,x_N) = (\mck_{n,\ga}f)(|\bx|,x_N)$ for $(\bx,x_N) \in \R^N_+$.
By virtue of \eqref{eq:EL} and $p > 1$, we deduce that $f$ is strictly decreasing along the radial direction.

\medskip
Finally, the classification result of maximizers for $n > 2\ga$, $p=\frac{2n}{n-2\ga}$, and $\ga \in (0,1)$ follows from \cite[Proposition 1.3]{HWY} and the fact that the Kelvin transforms
\[f^{\sharp}(\bx) = \frac{1}{|\bx|^{n-2\ga}}f\(\frac{\bx}{|\bx|^2}\) \quad \textup{and} \quad
U^{\sharp}(x) = \frac{1}{|x|^{n-2\ga}}U\(\frac{x}{|x|^2}\) = \mck_{n,\ga}f^{\sharp}(x).\]
of a radial maximizer $f$ of \eqref{eq:varprob} and $U = \mck_{n,\ga}f$ satisfy
\[\big\|f^{\sharp}\big\|_{L^{\frac{2n}{n-2\ga}}(\R^n)} = \|f\|_{L^{\frac{2n}{n-2\ga}}(\R^n)} \quad \textup{and} \quad
\big\|U^{\sharp}\big\|_{L^{\frac{2(n-2\ga+2)}{n-2\ga}}(\R^N_+;x_N^{\mz})} = \|U\|_{L^{\frac{2(n-2\ga+2)}{n-2\ga}}(\R^N_+;x_N^{\mz})},\]
respectively. This concludes the proof of Theorem \ref{thm:main11g}.
\end{proof}

\section{Results on the unit ball}
\subsection{Proof of Corollary \ref{cor:main12}}
By using Theorem \ref{thm:main11} and conformal geometric techniques, we prove Corollary \ref{cor:main12}. We first show (b) and then (a).

\medskip \noindent (b) Let $\rho_\pb$, $g_\pb^+$, and $\bg_\pb = \rho_\pb^2 g_\pb^+$ be the function and the metrics in \eqref{eq:gpb}--\eqref{eq:rhobgpb}.
Then $\rho_\pb$ is the defining function of $(\S^n,\bh_\pb)$, and
\begin{equation}\label{eq:rhopbgeo}
|d(\log\rho_\pb)|_{g_\pb^+}^2 = 1, \quad \textup{which is equivalent to } |d\rho_\pb|^2_{\bg_\pb} = 1.
\end{equation}
This means that $\rho_\pb$ is geodesic.

\begin{lemma}\label{lemma:smms}
Let $\mz = 1-2\ga \in (-1,1)$, and $\phi_\m: \R^N_+ \to \B^N$ be the M\"obius transformation defined in \eqref{eq:Mobius}.
The smooth metric measure spaces $(\R^N_+, |dx|^2, x_N^{\mz} dx, -2\ga)$ and $(\R^N_+ = \phi_\m^*(\B^N), \phi_\m^*\,\bg_\pb, \phi_\m^*(\rho_\pb^{\mz}dv_{\bg_\pb}), -2\ga)$ are pointwise conformally equivalent.
\end{lemma}
\begin{proof}
We see that $\phi_\m$ and $\rho_\pb$ extend smoothly to $\overline{\R^N_+}$ and $\overline{\B^N}$, respectively, and $\rho_\pb(\phi_\m(\bx,0)) = 0$ for all $\bx \in \R^n$. Thus the function
\begin{equation}\label{eq:u}
u(x) = \frac{x_N}{\rho_\pb(\phi_\m(x))} \quad \textup{for } x \in \R^N_+
\end{equation}
is smooth on $\overline{\R^N_+}$. Moreover,
\[(\phi_\m^*\bg_\pb)(x) = \rho_\pb^2(\phi_\m(x)) \(\phi_\m^*g_\pb^+\)(x) = \rho_\pb^2(\phi_\m(x)) g_\ph^+(x) 
= u^{-2}(x) |dx|^2\]
and
\[\phi_\m^*\(\rho_\pb^{\mz}dv_{\bg_\pb}\)(x) = \rho_\pb^{\mz}(\phi_\m(x)) \phi_\m^*(dv_{\bg_\pb})(x) = \(u^{-1}(x) x_N\)^{\mz} dv_{\phi_\m^*\bg_\pb}(x)\]
for all $x \in \overline{\R^N_+}$. Therefore the assertion is true.
\end{proof}

Given $\tf \in C^{\infty}(\S^n)$, we set
\[f(\bx) = \frac{\tf(\phi_\m(\bx,0))}{(|\bx|^2+1)^{\frac{n-2\ga}{2}}} \quad \textup{for } \bx \in \R^n, \qquad
U = \mck_{n,\ga}f, \quad V = u^{\frac{n-2\ga}{2}} U \quad \textup{in } \R^N_+.\]
We also write $x = \phi_\m^{-1}(y) \in \R^N_+$ and $(\bw,0) = \phi_\m^{-1}(\zeta)\in \pa\R^N_+$ for $\zeta \in \S^n$. By \eqref{eq:gaharext}--\eqref{eq:gaharext2} and \eqref{eq:Mobius2},
\begin{equation}\label{eq:gaharext3}
\begin{aligned}
U\(\phi_\m^{-1}(y)\) &= \ka_{n,\ga} \int_{\S^n} \frac{x_N^{2\ga}}{\(|\bx-\bw|^2+x_N^2\)^{\frac{n+2\ga}{2}}}
\frac{\tf(\zeta)}{(|\bw|^2+1)^{\frac{n-2\ga}{2}}} \(\frac{|\bw|^2+1}{2}\)^n (dv_{g_{\S^n}})_{\zeta} \\
&= \ka_{n,\ga} \int_{\S^n} \frac{(1-|y|^2)^{2\ga} |y+e_N|^{n-2\ga}}{|y-\zeta|^{n+2\ga}} \tf(\zeta) (dv_{\bh_\pb})_{\zeta} \quad \textup{for } y \in \B^N.
\end{aligned}
\end{equation}
From \eqref{eq:u}, \eqref{eq:Mobius}, \eqref{eq:rhobgpb}, \eqref{eq:gaharext3}, and \eqref{eq:wmcp}, we observe that
\begin{equation}\label{eq:smms0}
\wtv(y) := V\(\phi_\m^{-1}(y)\) 
= \(\frac{1+|y|}{|y+e_N|}\)^{n-2\ga} U\(\phi_\m^{-1}(y)\) = \(\wmck_{n,\ga}\tf\)(y) \quad \textup{for } y \in \B^N
\end{equation}
and
\begin{equation}\label{eq:smms1}
\wtv(y) = \(\frac{2}{1+y_N}\)^{\frac{n-2\ga}{2}} f\(\phi_\m^{-1}(y)\) = \tf(y) \quad \textup{for } y = (\bar{y},y_N) \in \S^n.
\end{equation}
Because of Lemma \ref{lemma:smms}, \eqref{eq:wcl2}, and \eqref{eq:degeq31}, we have
\[L^{\mz}_{\phi_\m^*\bg_\pb,\,\phi_\m^*\rho_\pb}(V) = u^{\frac{n-2\ga+4}{2}} L^{\mz}_{|dx|^2,x_N}(U) = -u^{\frac{n-2\ga+4}{2}}x_N^{2\ga-1} \textup{div}\(x_N^{\mz} \nabla U\) = 0 \quad \textup{in } \R^N_+.\]
This and \eqref{eq:smms0} lead to
\begin{equation}\label{eq:smms2}
-\textup{div}_{\bg_\pb}\(\rho_\pb^{\mz} \nabla_{\bg_\pb} \wtv\) + \frac{n-2\ga}{4(n+1-2\ga)} \rho_\pb^{\mz} R^{\mz}_{\bg_\pb,\,\rho_\pb} \wtv = 0 \quad \textup{in } \B^N.
\end{equation}
Appealing to the conformal covariance property of the conformal Laplacian (i.e. \eqref{eq:wcl2} with $m = 0$) and \eqref{eq:rhobgpb}, we also compute
\[R_{\bg_\pb} = -\frac{N-1}{N-2}\, (1+|y|)^{N+2}\, \Delta \left[\frac{1}{(1+|y|)^{N-2}}\right] = (N-1)^2\frac{(1+|y|)^2}{|y|}\]
and
\begin{align*}
-\Delta_{\bg_\pb} \rho_\pb &= L_{\bg_\pb} \rho_\pb - \frac{N-2}{4(N-1)} R_{\bg_\pb} \rho_\pb \\
&= -\frac{1}{4}(1+|y|)^{N+2}\, \Delta\left[\frac{1-|y|}{(1+|y|)^{N-1}}\right] - \frac{(N-2)(N-1)}{4} \frac{1-|y|^2}{|y|} = \frac{N-1}{2} \frac{1-|y|^2}{|y|}.
\end{align*}
From these equalities, \eqref{eq:wsc}, and \eqref{eq:rhopbgeo}, it follows that
\begin{equation}\label{eq:smms3}
\begin{aligned}
R^{\mz}_{\bg_\pb,\,\rho_\pb}(y) &= R_{\bg_\pb} - 2\mz \rho_\pb^{-1}\Delta_{\bg_\pb} \rho_\pb
+ 2\ga(1-2\ga) \rho_\pb^{-2}|\nabla_{\bg_\pb} \rho_\pb|_{\bg_\pb}^2 - 2\ga(1-2\ga) \rho_\pb^{-2} \\
&= n^2\frac{(1+|y|)^2}{|y|} + n\mz \frac{(1+|y|)^2}{|y|} = n(n+1-2\ga) \frac{(1+|y|)^2}{|y|} \quad \textup{for } y \in \B^N.
\end{aligned}
\end{equation}
The next result establishes the uniform continuity of both $\wtv$ and $\rho_\pb^{\mz}\frac{\pa \wtv}{\pa \rho_\pb}$ on $\overline{\B^N}$. Its proof will be postponed to Subsection \ref{subsec:smms5}.
\begin{lemma}\label{lemma:smms5}
Let $\ga \in (0,1)$. Then the functions $\wtv$ and $\rho_\pb^{\mz}\frac{\pa \wtv}{\pa \rho_\pb}$ are uniformly continuous on $\overline{\B^N}$. Furthermore, if $d_{\ga} < 0$ is the constant in \eqref{eq:degeq2}, then
\begin{equation}\label{eq:smms50}
\begin{cases}
\displaystyle \wtv(y) \to \tf(y_0),\\
\displaystyle \(\frac{d_{\ga}}{2\ga} \rho_\pb^{\mz}\frac{\pa \wtv}{\pa \rho_\pb}\)(y) \to \(P^{\ga}_{g_\pb^+,\bh_\pb} \tf\)(y_0)
\end{cases}
\text{as } y \in \B^N \to y_0 \in \S^n.
\end{equation}
\end{lemma}
\noindent Consequently, by testing \eqref{eq:smms2}--\eqref{eq:smms3} with $\wtv$ and then applying \eqref{eq:degeq41}, \eqref{eq:smms1}, and the previous lemma, we find $c > 0$ depending only on $n$ and $\ga$ such that
\begin{equation}\label{eq:smms4}
\begin{aligned}
&\ c\big\|\wtv\big\|_{H^{1,2}(\B^N;\rho_\pb^{\mz},\bg_\pb)} \\
&\le -\frac{d_{\ga}}{2\ga} \int_{\B^N} \rho_\pb^{\mz} \left[\big|\nabla_{\bg_\pb} \wtv(y)\big|_{\bg_\pb}^2 + \dfrac{n(n-2\ga)}{4} \dfrac{(1+|y|)^2}{|y|} \big(\wtv(y)\big)^2\right] (dv_{\bg_\pb})_y\\
&= \int_{\S^n} \tf P^{\ga}_{g_\pb^+,\bh_\pb} \tf dv_{\bh_\pb} < \infty.
\end{aligned}
\end{equation}
By combining \eqref{eq:smms1}--\eqref{eq:smms3} with \eqref{eq:smms4}, we can obtain equation (1.18).
The uniqueness of $\wtv$ can be easily deduce from \eqref{eq:smms4},
and the H\"older regularity of $\wtv$ and $\rho_\pb^{\mz}\frac{\pa \wtv}{\pa \rho_\pb}$ on $\overline{\B^N}$ follows from elliptic regularity, which is demonstrated in \cite[Appendix A]{KMW2}; cf. Lemma \ref{lemma:ereg} below.

\medskip \noindent (a) By \eqref{eq:rhobgpb}, \eqref{eq:u}, \eqref{eq:smms0}, and \eqref{eq:Mobius},
\begin{align*}
&\ \left\|\wmck_{n,\ga} \tf\right\|_{L^{\frac{2(n-2\ga+2)}{n-2\ga}}(\B^N;\rho_\pb^{\mz},\bg_\pb)} \\
&= \left[\int_{\R^N_+} x_N^{\mz}|U(x)|^{\frac{2(n-2\ga+2)}{n-2\ga}} \(\dfrac{2}{|\phi_\m(x)|+1}\)^{2N} \(\dfrac{u(x)}{|x+e_N|^2}\)^N dx\right]^{\frac{n-2\ga}{2(n-2\ga+2)}} \\
&= \|U\|_{L^{\frac{2(n-2\ga+2)}{n-2\ga}}(\R^N_+;x_N^{\mz})}
\end{align*}
and
\begin{align*}
\big\|\tf\big\|_{L^{\frac{2n}{n-2\ga}}(\S^n,\bh_\pb)} &= \left[\frac{1}{2^n} \int_{\S^n} \big|\tf(y)\big|^{\frac{2n}{n-2\ga}} (dv_{g_{\S^n}})_y\right]^{\frac{n-2\ga}{2n}} \\
&= \left[\int_{\R^n} \big|\tf(\phi_\m(\bx,0))\big|^{\frac{2n}{n-2\ga}} \frac{d\bx}{(1+|\bx|^2)^n} \right]^{\frac{n-2\ga}{2n}} = \|f\|_{L^{\frac{2n}{n-2\ga}}(\R^n)},
\end{align*}
which implies \eqref{eq:main12}. The equality assertion for \eqref{eq:main12} is the consequence of the relation
\[f(\bx) = \(\frac{\lambda}{\lambda^2 + |\bx-\bx_0|^2}\)^{\frac{n-2\ga}{2}} \quad \Leftrightarrow \quad \tf(y) = \left[\frac{2\lambda(1+\zeta_{0N})}{\lambda^2(1+\zeta_{0N})(1+y_N)+2(1-\zeta_0 \cdot y)}\right]^{\frac{n-2\ga}{2}}\]
for $\bx \in \R^n$ and $y = \phi_\m(\bx,0) \in \S^n$ where $\zeta_0 = (\bar{\zeta}_0,\zeta_{0N}) = \phi_\m(\bx_0,0) \in \S^n$.

\medskip
The proof of Corollary \ref{cor:main12} is now concluded under the validity of Lemma \ref{lemma:smms5}.

\subsection{Proof of Lemma \ref{lemma:smms5}}\label{subsec:smms5}
Throughout this subsection, we write
\[\wtk_{n,\ga}(y,\zeta) = \frac{\ka_{n,\ga}}{2^n} \frac{(1+|y|)^{n-2\ga}(1-|y|^2)^{2\ga}}{|y-\zeta|^{n+2\ga}} \quad \textup{for } y \in \B^N \textup{ and } \zeta \in \S^n.\]
Then $\wtv(y) = \int_{\S^n} \wtk_{n,\ga}(y,\zeta)\tf(\zeta) (dv_{g_{\S^n}})_{\zeta}$ for $y \in \B^N$, which immediately yields that $\wtv$ and $\rho_\pb^{\mz}\frac{\pa \wtv}{\pa \rho_\pb}$ are continuous in $\B^N$.
Hence, to show their uniform continuity on $\overline{\B^N}$, we only have to deduce \eqref{eq:smms50}.

\medskip
We initiate the proof by evaluating some integrals needed later. Let $r = |y| \in [0,1)$. By using the spherical coordinate system and taking $s = \frac{1+r}{1-r} \tan \frac{\vph}{2}$, we compute
\begin{equation}\label{eq:smms51}
\begin{aligned}
&\ \(1-|y|^2\)^{2\ga} \int_{\S^n} \frac{(dv_{g_{\S^n}})_{\zeta}}{|y-\zeta|^{n+2\ga}} \\
&= \(1-r^2\)^{2\ga} \frac{2\pi^{\frac{n}{2}}}{\Gamma\(\frac{n}{2}\)} \int_0^{\pi} \frac{\sin^{n-1}\vph}{(1+r^2-2r\cos\vph)^{\frac{n+2\ga}{2}}}\, d\vph \\
&= \frac{1}{(1+r)^{n-2\ga}} \frac{2^{n+1}\pi^{\frac{n}{2}}}{\Gamma\(\frac{n}{2}\)}
\int_0^{\infty} \frac{s^{n-1}}{(1+s^2)^{\frac{n+2\ga}{2}}} \frac{ds}{[1+\big(\frac{1-r}{1+r}\big)^2s^2]^{\frac{n-2\ga}{2}}} \\ 
&= \frac{2^n\pi^{\frac{n}{2}}}{(1+r)^{n-2\ga}} \left[\frac{\Gamma(\ga)}{\Gamma\(\frac{n+2\ga}{2}\)} + (1-r)^{2\ga} \frac{\Gamma(-\ga)}{2^{2\ga}\Gamma\(\frac{n-2\ga}{2}\)} + O\((1-r)^{\min\{2,1+2\ga\}}\) \right]
\end{aligned}
\end{equation}
and
\begin{equation}\label{eq:smms52}
\begin{aligned}
&\ \(1-|y|^2\)^{2\ga} \int_{\S^n} \frac{|y|^2-y \cdot \zeta}{|y-\zeta|^{n+2\ga+2}} (dv_{g_{\S^n}})_{\zeta} \\
&= \frac{r}{(1+r)^{n+1-2\ga}(1-r)} \frac{2^{n+1}\pi^{\frac{n}{2}}}{\Gamma\(\frac{n}{2}\)} \int_0^{\infty} \frac{s^{n-1}((1-r)s^2-(1+r))}{(1+s^2)^{\frac{n+2\ga+2}{2}}} \frac{ds}{[1+\big(\frac{1-r}{1+r}\big)^2s^2]^{\frac{n-2\ga}{2}}} \\
&= \frac{2^{n+1}\pi^{\frac{n}{2}}r}{(1+r)^{n+1-2\ga}} \left[- \frac{(1-r)^{-1} 2\ga\Gamma(\ga)}{(n+2\ga)\Gamma\(\frac{n+2\ga}{2}\)}
+ \frac{\Gamma(\ga)}{2\Gamma\(\frac{n+2\ga}{2}\)} + \frac{(1-r)^{2\ga}\Gamma(-\ga)}{2^{1+2\ga}\Gamma\(\frac{n-2\ga}{2}\)} + O(1-r)\right].
\end{aligned}
\end{equation}
We obtained the Puiseux series on the last lines of \eqref{eq:smms51} and \eqref{eq:smms52} by utilizing Mathematica software. These equalities hold for $r \in [0,1)$ near $1$.

\medskip
We are now ready to deduce the first claim of \eqref{eq:smms50}. By \eqref{eq:wmcp} and \eqref{eq:smms51},
\begin{equation}\label{eq:smms53}
\(\wmck_{n,\ga} 1\)(y) = \int_{\S^n} \wtk_{n,\ga}(y,\zeta) (dv_{g_{\S^n}})_{\zeta} \to \frac{\ka_{n,\ga} \pi^{\frac{n}{2}} \Gamma(\ga)}{\Gamma\(\frac{n+2\ga}{2}\)} = 1 \quad \textup{as } r \to 1^-.
\end{equation}
Fix any $y_0 \in \S^n$. Then we have
\begin{align*}
\left|\wtv(y) - \tf(y_0)\right| \le \int_{\S^n} \wtk_{n,\ga}(y,\zeta) \left|\tf(\zeta)- \tf(y_0)\right| (dv_{g_{\S^n}})_{\zeta} + \left|\(\wmck_{n,\ga} 1\)(y) - 1\right| \left|\tf(y_0)\right|.
\end{align*}
By exploiting the continuity of $\tf$ and \eqref{eq:smms53}, we see that the right-hand side tends to 0 as $y \to y_0$.

\medskip
We next derive the second claim of \eqref{eq:smms50}. By applying \eqref{eq:rhobgpb} and \eqref{eq:wmcp}, we calculate
\begin{align}
\(\rho_\pb^{\mz}\frac{\pa \wtv}{\pa \rho_\pb}\)(y)
&= \(\frac{1-r}{1+r}\)^{\mz} \frac{dr}{d\rho_\pb} \frac{\pa \wtv}{\pa r}(y) = -\frac{(1+r)^{1+2\ga}(1-r)^{1-2\ga}}{2} \left[\frac{y}{r} \cdot \nabla \wtv(y)\right] \label{eq:smms54} \\
&= \frac{(1+r)^{1+2\ga}(1-r)^{1-2\ga}}{2} \left[\left\{\frac{4\ga r}{1-r^2} - \frac{n-2\ga}{1+r}\right\} \int_{\S^n} \wtk_{n,\ga}(y,\zeta) \tf(\zeta) (dv_{g_{\S^n}})_{\zeta} \right. \nonumber \\
&\hspace{125pt} \left. + \frac{n+2\ga}{r} \int_{\S^n} \wtk_{n,\ga}(y,\zeta) \frac{(r^2-y \cdot \zeta)\tf(\zeta)}{|y-\zeta|^2} (dv_{g_{\S^n}})_{\zeta}\right] \nonumber
\end{align}
for $y \in \B^N$.

Assume for the moment that $\tf = 1$. From \eqref{eq:smms54} and \eqref{eq:smms51}--\eqref{eq:smms52}, we see
\begin{equation}\label{eq:smms55}
\begin{aligned}
&\ \(\frac{d_{\ga}}{2\ga} \rho_\pb^{\mz}\frac{\pa \wtv}{\pa \rho_\pb}\)(y) \\
&= \frac{(1+r)^{2\ga}(1-r)^{1-2\ga} \Gamma(\ga)}{4^{1-\ga}\ga \Gamma(-\ga)} \left[\frac{4\ga r - (n-2\ga)(1-r)}{1-r}\left\{1 + (1-r)^{2\ga} \frac{\Gamma(-\ga)\Gamma\(\frac{n+2\ga}{2}\)}{2^{2\ga}\Gamma(\ga) \Gamma\(\frac{n-2\ga}{2}\)}\right\} \right. \\
&\hspace{125pt} \left. + (n+2\ga) \left\{-(1-r)^{-1} \frac{4\ga}{n+2\ga} + 1 
\right\} + O\((1-r)^{\min\{1,2\ga\}}\)\right] \\
&\to 2^{2\ga} \frac{\Gamma\(\frac{n+2\ga}{2}\)}{\Gamma\(\frac{n-2\ga}{2}\)}
= 2^{2\ga} P^{\ga}_{g_\pb^+,g_{\S^n}}1 = P^{\ga}_{g_\pb^+,\bh_\pb}1 \quad \textup{as } r \to 1^-. \end{aligned}
\end{equation}
The first equality on the last line of \eqref{eq:smms55} is a classical result whose proof is found in \cite{PS} 
or \cite[Subsection 6.4]{Go}. 
The second equality is a consequence of \eqref{eq:fcl}.

Next, we deal with general $\tf \in C^{\infty}(\S^n)$. We will regard $\tf$ as a smooth function in $\R^n \setminus \{0\}$
by setting $\tf(y) = \tf(y/|y|)$ for $y \in \R^n \setminus \{0\}$. According to \cite{PS} (see also \cite[Proposition 6.4.3]{Go}),
\begin{equation}\label{eq:smms56}
\(P^{\ga}_{g_\pb^+,\bh_\pb} \tf\)(y_0) = \(P^{\ga}_{g_\pb^+,\bh_\pb}1\) \tf(y_0) + a_{n,\ga} \int_{\S^n} \frac{\tf(y_0)-\tf(\zeta)}{|y_0-\zeta|^{n+2\ga}} (dv_{\bh_\pb})_{\zeta} \quad \textup{for } y_0 \in \S^n
\end{equation}
where $a_{n,\ga} := 2^{n+2\ga} \cdot 2^{2\ga}\ga \Gamma(\frac{n+2\ga}{2})/(\pi^{\frac{n}{2}}\Gamma(1-\ga))$ and $\int_{\S^n}$ is understood as $\lim_{\ep \to 0} \int_{|\zeta-y_0| > \ep}$.
Hence if
\begin{equation}\label{eq:smms571}
\int_{\S^n} \frac{\tf(y)-\tf(\zeta)}{|y-\zeta|^{n+2\ga}} (dv_{\bh_\pb})_{\zeta} \to \int_{\S^n} \frac{\tf(y_0)-\tf(\zeta)}{|y_0-\zeta|^{n+2\ga}} (dv_{\bh_\pb})_{\zeta} \quad \textup{as } y \to y_0
\end{equation}
and there exists a universal constant $C > 0$ such that
\begin{equation}\label{eq:smms572}
\left|\int_{\S^n} \frac{(r^2-y \cdot \zeta)(\tf(y)-\tf(\zeta))}{|y-\zeta|^{n+2\ga+2}} (dv_{\bh_\pb})_{\zeta}\right| \le C \quad \textup{for all } y \in \B^N \textup{ close to } y_0,
\end{equation}
then \eqref{eq:smms54}--\eqref{eq:smms572} will give
\begin{align*}
\(\frac{d_{\ga}}{2\ga} \rho_\pb^{\mz}\frac{\pa \wtv}{\pa \rho_\pb}\)(y) &= \(P^{\ga}_{g_\pb^+,\bh_\pb}1 + o(1)\) \tf(y) + \(a_{n,\ga}+o(1)\) \int_{\S^n} \frac{\tf(y)-\tf(\zeta)}{|y-\zeta|^{n+2\ga}} (dv_{\bh_\pb})_{\zeta} \\
&\ + \left[\frac{(n+2\ga)a_{n,\ga}}{2\ga} + o(1)\right] (1-r)
\int_{\S^n} \frac{(r^2-y \cdot \zeta)(\tf(y)-\tf(\zeta))}{|y-\zeta|^{n+2\ga+2}} (dv_{\bh_\pb})_{\zeta} \\
&\to \(P^{\ga}_{g_\pb^+,\bh_\pb} \tf\)(y_0) \quad \textup{as } y \to y_0,
\end{align*}
as claimed.

Let us verify \eqref{eq:smms571}. The following argument is in the spirit of the proof of \cite[Lemma 2.1]{Sa}.
We represent the basis of normalized spherical harmonics in $L^2(\S^n,g_{\S^n})$ by $\{Y_{m,\ell}\}$,
where $\ell \in \N \cup \{0\}$ denotes the degree of a spherical harmonic and $m = 1, \ldots, d(\ell)$ labels the degeneracy.
Recall that $Y_{m,\ell}$ is the restriction of a harmonic homogeneous polynomial of degree $\ell$, which we will keep denoting by $Y_{m,\ell}$.
It is well-known that
\begin{equation}\label{eq:smms581}
d(\ell) \simeq \ell^{n-1} \quad \textup{and} \quad |\nabla^{\alpha} Y_{m,\ell}(x)| \le C \ell^{\frac{n-1}{2}+|\alpha|} |x|^{\ell-|\alpha|} \quad \textup{for any } x \in \R^N \textup{ and multi-index } \alpha.
\end{equation}
Writing $y = r\ty \in \B^N$ with $r = |y|$ and $\ty = y/|y| \in \S^n$, and utilizing the partial wave decomposition of $\tf$ and the Funk-Hecke formula, we obtain
\begin{equation}\label{eq:smms582}
\int_{|\zeta-y| > \ep} \frac{\tf(y)-\tf(\zeta)}{|y-\zeta|^{n+2\ga}} (dv_{g_{\S^n}})_{\zeta}
= \left|\S^{n-1}\right| \sum_{\ell=1}^{\infty}\sum_{m=1}^{d(\ell)} c_{m,\ell} \mci_{\ell}(r) Y_{m,\ell}(\ty)
\end{equation}
provided $0 < \ep < 1-r$. Here,
\[c_{m,\ell} := \int_{\S^n} (\tf Y_{m,\ell}) dv_{g_{\S^n}}, \quad \mci_{\ell}(r) := \int_{-1}^1 \frac{(1-s^2)^{\frac{n-2}{2}}}{(1+r^2-2rs)^{\frac{n+2\ga}{2}}} [1-P_{\ell}(s)] ds > 0,\]
and $P_{\ell}$ is the Legendre polynomial of degree $\ell$. 
Because $0 \le 1-P_{\ell}(s) \le C\ell^2(1-s)$ for all $s \in [-1,1]$ and $\ga \in (0,1)$, we have that $\mci_{\ell}(r) < \infty$ for $r$ near $1$.
Also, thanks to the smoothness of $\tf$, if we select any $i \in \N$, then there exists $C_i > 0$ depending only on $n$, $\tf$,
and $i$ such that $|c_{m,\ell}| \le C_i\ell^{\frac{n-1}{2}-2i}$. 
Consequently, by applying \eqref{eq:smms581}, the mean value theorem, and the dominated convergence theorem, and then choosing sufficiently large $i \in \N$, we find
\begin{align*}
&\ \sum_{\ell=1}^{\infty}\sum_{m=1}^{d(\ell)} |c_{m,\ell}| \left|\mci_{\ell}(r) Y_{m,\ell}(\ty) - \mci_{\ell}(1) Y_{m,\ell}(y_0)\right| \\
&\le CC_i \sum_{\ell=1}^{\infty} \ell^{n-1} \ell^{\frac{n-1}{2}-2i} (|\mci_{\ell}(r)-\mci_{\ell}(1)| \ell^{\frac{n-1}{2}} + \ell^2 \cdot \ell^{\frac{n-1}{2}+1}|\ty-y_0|) \\
&\le CC_i \sum_{\ell=1}^{\infty} \ell^{2n+1-2i} \left[\int_{-1}^1 \left|\frac{1}{(1+r^2-2rs)^{\frac{n+2\ga}{2}}} - \frac{1}{(2-2s)^{\frac{n+2\ga}{2}}}\right| (1-s)^{\frac{n}{2}} ds + |\ty-y_0|\right] \to 0
\end{align*}
as $y \to y_0$ (so that $r \to 1$ and $\ty \to y_0$). In view of \eqref{eq:smms582}, this implies \eqref{eq:smms571}. The proof of \eqref{eq:smms572} is similar, so we skip it.

\medskip
This finishes the proof of Lemma \ref{lemma:smms5}.

\section{The variational problem involving the weighted isoperimetric ratio}\label{sec:main2}
\subsection{The existence of nonnegative maximizers of variational problem $\mcs_3$}
We again recall Convention \ref{conv}.

In Lemma \ref{lemma:wir}, we established the `equivalence' between the variational problem \eqref{eq:thetanw} and the problems denoted as $\mcs_1$, $\mcs_2$, and $\mcs_3$ in \eqref{eq:mcs},
thereby reducing a variational problem involving the weighted isoperimetric ratio into ones involving the weighted Poisson kernels.
In this subsection, we prove the finiteness of $\mcs_3$, the validity of \eqref{eq:thetaest}, and the existence of nonnegative maximizers of $\mcs_3$ under the condition \eqref{eq:thetaest2}.
To achieve them, we employ the concentration-compactness argument in \cite[Section 3]{HWY} and \cite[Section 6]{HWY2}.
As previously stated, we work in a broader setting than Lemma \ref{lemma:wir} or Theorem \ref{thm:main21}.

\begin{lemma}\label{lemma:Poiest2}
Assume that $n \in \N$, $n > 2\ga$, $\ga \in (0,1)$, and $p \in (1,\infty]$.
For a function $f$ on $M$, let $\mck_{\bgg,\rhog}^{\mz}f$ be the function on $\ox$ given in \eqref{eq:Poibg}.
Then there exists a constant $C > 0$ depending only on $n$, $\ga$, $p$, $(\ox,\bgg)$, and $\rhog$ such that
\[\|\mck_{\bgg,\rhog}^{\mz}f\|_{L^{\frac{(n-2\ga+2)p}{n}}(X;\rhog^{\mz},\bgg)} \le C\|f\|_{L^p(M,\bh)} \quad \text{for all } f \in L^p(M,\bh).\]
In particular, if $p \in (1,\infty)$, then we have
\begin{equation}\label{eq:thetanwg}
\begin{aligned}
\mcs_{3;p} &= \mcs_{3;n,\ga,p}\(X,\bgg,\rhog\) \\
&:= \sup\left\{\left\|\mck_{\bgg,\rhog}^{\mz}f\right\|_{L^{\frac{(n-2\ga+2)p}{n}}(X;\rhog^{\mz},\bgg)}^{\frac{(n-2\ga+2)p}{n}}: \|f\|_{L^p(M,\bh)} = 1\right\} < \infty.
\end{aligned}
\end{equation}
\end{lemma}
\begin{proof}
Having \eqref{eq:Poi21} in hand, one can follow the argument in the proof of Lemma \ref{lemma:Poiest1}.
\end{proof}

\begin{lemma}\label{lemma:tech1}
Let $n \in \N$, $\ga \in (0,1)$, $p \in (1,\infty)$, and $\delta_0 > 0$ be a small number such that \eqref{eq:g+ ext} holds.
Pick any $t \in (1,\frac{n-2\ga+2}{n-p})$ if $1 < p < n$ and $t \in (1,\infty)$ if $p \ge n$.
Then there exists a constant $C > 0$ depending only on $n$, $\ga$, $p$, $(\ox,\bgg)$, and $\rhog$ such that
\[\left\|\int_M \frac{\rhog^{2\ga} f(\sigma)}{d_{\bgg}(\cdot,\sigma)^{n+2\ga-1}}(dv_{\bh})_\sigma \right\|_{L^{tp}(\mct_{\delta_0};\rhog^{\mz},\bgg)}
\le C\|f\|_{L^p(M,\bh)} \quad \text{for all } f \in L^p(M,\bh).\]
where $\mct_{\delta_0} = M \times (0,\delta_0)$.
\end{lemma}
\begin{proof}
Owing to the assumption on $t$, there is a constant $b$ satisfying
\[tp-\frac{nt(p-1)}{n+2\ga-1} < b < \frac{n-2\ga+2+2\ga tp}{n+2\ga-1}.\]
We write $a=tp>p$, $q=\frac{ap}{a-p}>1$, and $p'=\frac{p}{p-1}>1$. By H\"older's inequality with $\frac{1}{a}+\frac{1}{q}+\frac{1}{p'}=1$ and the fact that $\frac{(n+2\ga-1)(a-b)p'}{a} < n$, we get
\begin{equation}\label{eq:tech11}
\int_M \frac{|f(\sigma)|}{d_{\bgg}(\xi,\sigma)^{n+2\ga-1}}(dv_{\bh})_\sigma
\le C \left[\int_M \frac{|f(\sigma)|^p}{d_{\bgg}(\xi,\sigma)^{(n+2\ga-1)b}} (dv_{\bh})_\sigma\right]^{\frac{1}{tp}} \left[\int_M |f(\sigma)|^p(dv_{\bh})_\sigma\right]^{\frac{a-p}{ap}}
\end{equation}
where $C > 0$ is independent of $\xi \in \overline{\mct_{\delta_0}}$. Since $(n+2\ga-1)b-(2\ga tp+\mz)<n+1$, it follows from \eqref{eq:tech11} that
\begin{align*}
&\ \int_{\mct_{\delta_0}} \left|\int_M\frac{f(\sigma)}{d_{\bgg}(\xi,\sigma)^{n+2\ga-1}}(dv_{\bh})_\sigma\right|^{tp} \rhog(\xi)^{2\ga t p+\mz} (dv_{\bgg})_{\xi}\\
&\le C\left[\int_M |f(\sigma)|^p(dv_{\bh})_\sigma\right]^{\frac{(a-p)t}{a}} \int_M |f(\sigma)|^p \int_{\mct_{\delta_0}}
\frac{\rhog(\xi)^{2\ga tp + \mz}}{d_{\bgg}(\xi,\sigma)^{(n+2\ga-1)b}}(dv_{\bgg})_{\xi} (dv_{\bh})_\sigma\\
&\le C\left[\int_M |f(\sigma)|^p(dv_{\bh})_\sigma\right]^t. \qedhere
\end{align*}
\end{proof}

\begin{lemma}\label{lemma:tech2}
Under the assumptions of Lemma \ref{lemma:tech1}, we set
\[(\mck_0^{\mz}f)(\xi) = \int_M \mck_0^{\mz}(\xi,\sigma)f(\sigma) (dv_{\bh})_{\sigma} \quad \text{for } \xi \in \overline{\mct_{\delta_0}}\]
where $\mck_0^{\mz}$ is the function defined in \eqref{eq:Poi02}.
Let also $\eta \in C^{0,1}(M)$ (i.e., $\eta$ is Lipschitz continuous on $M$) and $\pi: \overline{\mct_{\delta_0}} \to M$ the orthogonal projection onto $M$.
Then there exists a constant $C > 0$ depending only on $n$, $\ga$, $p$, $(\ox,\bgg)$, and $\rhog$ such that
\begin{equation}\label{eq:tech21}
\|\mck_{\bgg,\rhog}^{\mz}f - \mck_0^{\mz}f\|_{L^{tp}(\mct_{\delta_0};\rhog^{\mz},\bgg)} \le C\|f\|_{L^p(M,\bh)}
\end{equation}
and
\begin{equation}\label{eq:tech22}
\left\|(\eta \circ \pi) \cdot \mck_{\bgg,\rhog}^{\mz}f - \mck_{\bgg,\rhog}^{\mz}(\eta f)\right\|_{L^{tp}(\mct_{\delta_0};\rhog^{\mz},\bgg)} \le C\|\nabla \eta\|_{L^{\infty}(M,\bh)}\|f\|_{L^p(M,\bh)}.
\end{equation}
for all $f \in L^p(M,\bh)$.
\end{lemma}
\begin{proof}
By \eqref{eq:Poi22}, we have
\begin{align*}
\left|\(\mck_{\bgg,\rhog}^{\mz}f\)(\xi) - \(\mck_0^{\mz}f\)(\xi)\right|
&\le \int_M \left|\mck_{\bgg,\rhog}^{\mz}(\xi,\sigma) - \mck_0^{\mz}(\xi,\sigma)\right| |f(\sigma)| (dv_{\bh})_\sigma \\
&\le C\int_M\frac{\rhog(\xi)^{2\ga}|f(\sigma)|} {d_{\bgg}(\xi,\sigma)^{n+2\ga-2}}(dv_{\bh})_\sigma,
\end{align*}
and
\begin{align*}
\left|(\eta \circ \pi)(\xi) (\mck_{\bgg,\rhog}^{\mz}f)(\xi)-\mck_{\bgg,\rhog}^{\mz}(\eta f)(\xi)\right|
&\le \int_M |\eta(\pi(\xi))-\eta(\sigma)| \left|\mck_{\bgg,\rhog}^{\mz}(\xi,\sigma)\right| \left|f(\sigma)\right| (dv_{\bh})_\sigma\\
&\le C\|\nabla \eta\|_{L^{\infty}(M,\bh)}\int_M\frac{\rhog(\xi)^{2\ga}|f(\sigma)|}{d_{\bgg}(\xi,\sigma)^{n+2\ga-1} }(dv_{\bh})_\sigma
\end{align*}
for $\xi \in \overline{\mct_{\delta_0}}$. Therefore, inequalities \eqref{eq:tech21} and \eqref{eq:tech22} follow from Lemma \ref{lemma:tech1}.
\end{proof}

\begin{lemma}\label{lemma:tech3}
Let $n \in \N$, $n > 2\ga$, $\ga \in (0,1)$, and $p \in (1,\infty)$. For any $\ep>0$, there exists small $\delta \in (0,\frac{1}{2}\min\{\delta_0,\textup{inj}(M,\bh)\})$ such that
\[\left\|\mck_{\bgg,\rhog}^{\mz}f\right\|_{L^{\frac{(n-2\ga+2)p}{n}}(\mct_{\delta};\rhog^{\mz},\bgg)}
\le \left[s_{n,\ga,p}^{\frac{n}{(n-2\ga+2)p}} + \ep\right] \|f\|_{L^p(M,\bh)} \quad \text{for all } f \in L^p(M,\bh)\]
where $s_{n,\ga,p} > 0$ is the quantity introduced in \eqref{eq:varprob}.
\end{lemma}
\begin{proof}
We set $\mcc^N_{\bgg}(\sigma,2\delta) = B^n_{\bh}(\sigma,2\delta) \times [0,2\delta) \subset \ox$ for $\sigma \in M$ and $\delta > 0$ small enough.
There exist $m \in \N$ and $\{\sigma_i\}_{i=1}^m \subset M$ such that $\overline{\mct_{\delta}} \subset \cup_{i=1}^m \mcc_{i,2\delta}$ where $\mcc_{i,2\delta} := \mcc^N_{\bgg}(\sigma_i,2\delta)$.
Let $\{\eta_i\}_{i=1}^m \subset C^{\infty}(M)$ be a partition of unity for $\overline{\mct_{\delta}}$ subordinate to $\{\mcc_{i,2\delta}\}_{i=1}^m$
such that $\eta_i^{1/p} \in C^{\infty}(M)$ for each $i = 1,\ldots,m$. Then it holds that
\begin{equation}\label{eq:tech31}
\left\|\mck_{\bgg,\rhog}^{\mz}f\right\|_{L^{\frac{(n-2\ga+2)p}{n}}(\mct_{\delta};\rhog^{\mz},\bgg)}^p
\le \sum_{i=1}^m \left\|(\eta_i \circ \pi)^{\frac{1}{p}} \mck_{\bgg,\rhog}^{\mz}f \right\|_{L^{\frac{(n-2\ga+2)p}{n}}(\mcc_{i,2\delta};\rhog^{\mz},\bgg)}^p.
\end{equation}
By virtue of Lemma \ref{lemma:tech2}, we can select $t > \frac{n-2\ga+2}{n}$ such that \eqref{eq:tech21} and \eqref{eq:tech22} hold.
Employing \eqref{eq:tech22} and H\"older's inequality with $\frac{n-2\ga+2}{nt}+\frac{nt-(n-2\ga+2)}{nt}=1$, we find
\begin{align}
&\ \left\|(\eta_i \circ \pi)^{\frac{1}{p}} \mck_{\bgg,\rhog}^{\mz}f \right\|_{L^{\frac{(n-2\ga+2)p}{n}}(\mcc_{i,2\delta};\rhog^{\mz},\bgg)} \label{eq:tech32} \\
&\le \left\|\mck_{\bgg,\rhog}^{\mz}(\eta_i^{\frac{1}{p}} f)\right\|_{L^{\frac{(n-2\ga+2)p}{n}}(\mcc_{i,2\delta};\rhog^{\mz},\bgg)} + \left\|(\eta_i \circ \pi)^{\frac{1}{p}} \cdot \mck_{\bgg,\rhog}^{\mz}f
- \mck_{\bgg,\rhog}^{\mz} (\eta_i^{\frac{1}{p}}f)\right\|_{L^{\frac{(n-2\ga+2)p}{n}}(\mcc_{i,2\delta};\rhog^{\mz},\bgg)} \nonumber \\
&\le \left\|\mck_{\bgg,\rhog}^{\mz}(\eta_i^{\frac{1}{p}} f)\right\|_{L^{\frac{(n-2\ga+2)p}{n}}(\mcc_{i,2\delta};\rhog^{\mz},\bgg)}
+ C\delta^{\frac{2(1-\ga)}{p}(\frac{n}{n-2\ga+2}-\frac{1}{t})} \Big\|\nabla (\eta_i^{\frac{1}{p}})\Big\|_{L^{\infty}(M,\bh)} \|f\|_{L^p(M,\bh)}. \nonumber
\end{align}
Besides, we see from \eqref{eq:tech21} and \eqref{eq:varprob} that
\begin{equation}\label{eq:tech33}
\begin{aligned}
&\ \left\|\mck_{\bgg,\rhog}^{\mz}(\eta_i^{\frac{1}{p}} f)\right\|_{L^{\frac{(n-2\ga+2)p}{n}}(\mcc_{i,2\delta};\rhog^{\mz},\bgg)} \\
&\le \left\|\mck_0^{\mz} (\eta_i^{\frac{1}{p}} f)\right\|_{L^{\frac{(n-2\ga+2)p}{n}}(\mcc_{i,2\delta};\rhog^{\mz},\bgg)}
+ \left\|\mck_{\bgg,\rhog}^{\mz}(\eta_i^{\frac{1}{p}}f) - \mck_0^{\mz}(\eta_i^{\frac{1}{p}} f)\right\|_{L^{\frac{(n-2\ga+2)p}{n}}(\mcc_{i,2\delta};\rhog^{\mz},\bgg)} \\
&\le s_{n,\ga,p}^{\frac{n}{(n-2\ga+2)p}} (1+o(1))\|\eta_i^{\frac{1}{p}}f\|_{L^p(M,\bh)} + C\delta^{\frac{2(1-\ga)}{p}(\frac{n}{n-2\ga+2}-\frac{1}{t})}\|f\|_{L^p(M,\bh)}
\end{aligned}
\end{equation}
where $o(1) \to 0$ as $\delta \to 0$. Combining \eqref{eq:tech31}--\eqref{eq:tech33}, we obtain
\begin{align*}
&\begin{medsize}
\displaystyle \ \left\|\mck_{\bgg,\rhog}^{\mz}f\right\|_{L^{\frac{(n-2\ga+2)p}{n}}(\mct_{\delta};\rhog^{\mz},\bgg)}^p
\end{medsize} \\
&\begin{medsize}
\displaystyle \le \sum_{i=1}^m \left[s_{n,\ga,p}^{\frac{n}{(n-2\ga+2)p}} (1+o(1))\|\eta_i^{\frac{1}{p}}f\|_{L^p(M,\bh)}
+ C\delta^{\frac{2(1-\ga)}{p}(\frac{n}{n-2\ga+2}-\frac{1}{t})} \left\{\Big\|\nabla (\eta_i^{\frac{1}{p}})\Big\|_{L^{\infty}(M,\bh)} + 1\right\} \|f\|_{L^p(M,\bh)}\right]^p
\end{medsize} \\
&\begin{medsize}
\displaystyle \le s_{n,\ga,p}^{\frac{n}{n-2\ga+2}} (1+o(1))\sum_{i=1}^m\|\eta_i^{\frac{1}{p}} f\|_{L^p(M,\bh)}^p+o(1)\|f\|_{L^p(M,\bh)}^p = s_{n,\ga,p}^{\frac{n}{n-2\ga+2}} (1+o(1))\|f\|_{L^p(M,\bh)}^p
\end{medsize}
\end{align*}
as desired.
\end{proof}

\begin{prop}\label{prop:ccp}\textup{[A concentration-compactness principle]} 
Let $n \in \N$, $n > 2\ga$, $\ga \in (0,1)$, and $p \in (1,\infty)$. Assume that
\begin{equation}\label{eq:ccp0}
\begin{cases}
f_i \rightharpoonup f \text{ in } L^p(M), \quad |f_i|^p dv_{\bh} \rightharpoonup \mu \text{ in } \mcm(M),\\
|\mck_{\bgg,\rhog}^{\mz}f_i|^{\frac{(n-2\ga+2)p}{n}} \rhog^{\mz}dv_{\bgg} \rightharpoonup \nu \text{ in } \mcm(\ox)
\end{cases}
\text{as } i \to \infty,
\end{equation}
where $\mcm(M)$ and $\mcm(\ox)$ are the spaces of bounded nonnegative measures on $M$ and $\ox$, respectively.
Then there exist an at most countable set $J$ (used as an index set), a family of $\{\sigma_j\}_{j \in J}$ of distinct points on $M$,
and sequences $\{\nu^{(j)}\}_{j \in J},\, \{\mu^{(j)}\}_{j \in J}$ of numbers in $(0,\infty)$ such that
\begin{enumerate}
\item[(i)] $\nu = |\mck_{\bgg,\rhog}^{\mz}f|^{\frac{(n-2\ga+2)p}{n}} \rhog^{\mz}dv_{\bgg} + \sum\limits_{j \in J} \nu^{(j)} \delta_{\sigma_j}$ so that $\nu^{(j)} = \nu(\sigma_j)$ for each $j \in J$;
\item[(ii)] $\mu \ge |f|^p dv_{\bh} + \sum\limits_{j \in J} \mu^{(j)} \delta_{\sigma_j}$ where we write $\mu^{(j)} = \mu(\sigma_j)$ for each $j \in J$;
\item[(iii)] $(\nu^{(j)})^{\frac{n}{(n-2\ga+2)p}} \le s_{n,\ga,p}^{\frac{n}{(n-2\ga+2)p}} (\mu^{(j)})^{\frac{1}{p}}$ for each $j \in J$.
\end{enumerate}
\end{prop}
\begin{proof}
By \eqref{eq:Poi21} and classical elliptic estimates, we may assume that $\mck_{\bgg,\rhog}^{\mz}f_i \to \mck_{\bgg,\rhog}^{\mz}f$ in $C^{\infty}_{\text{loc}}(X)$ as $i \to \infty$. Thus
\begin{equation}\label{eq:ccp1}
\nu|_X=|\mck_{\bgg,\rhog}^{\mz}f|^{\frac{(n-2\ga+2)p}{n}} \rhog^{\mz}dv_{\bgg}.
\end{equation}

Fix a number $t > \frac{n-2\ga+2}{n}$ for which \eqref{eq:tech21} and \eqref{eq:tech22} hold, and let $\vph \in C^{0,1}(M)$.
By Lemma \ref{lemma:tech3} and \eqref{eq:tech22}, for any $\ep>0$, there exists $\delta \in (0,\delta_0)$ such that
\begin{align*}
&\ \left\|(\vph \circ \pi)\mck_{\bgg,\rhog}^{\mz}f_i\right\|_{L^{\frac{(n-2\ga+2)p}{n}}(\mct_{\delta};\rhog^{\mz},\bgg)} \\
&\le \left\|\mck_{\bgg,\rhog}^{\mz}(\vph f_i)\right\|_{L^{\frac{(n-2\ga+2)p}{n}}(\mct_{\delta};\rhog^{\mz},\bgg)}
+ \left\|\mck_{\bgg,\rhog}^{\mz}(\vph f_i) - (\vph \circ \pi) \mck_{\bgg,\rhog}^{\mz}f_i\right\|_{L^{\frac{(n-2\ga+2)p}{n}}(\mct_{\delta};\rhog^{\mz},\bgg)} \\
&\le \left[s_{n,\ga,p}^{\frac{n}{(n-2\ga+2)p}} + \ep\right] \|\vph f_i\|_{L^p(M,\bh)} + C \delta^{\frac{2(1-\ga)}{p}(\frac{n}{n-2\ga+2}-\frac{1}{t})} \|\nabla \vph\|_{L^{\infty}(M,\bh)} \|f_i\|_{L^p(M,\bh)}.
\end{align*}
We recall that $\sup_{i \in \N} \|f_i\|_{L^p(M,\bh)} < \infty$ by the uniform boundedness principle. Letting $i \to \infty$, $\delta \to 0$ and using \eqref{eq:ccp1}, and then taking $\ep \to 0$, we discover
\[\(\int_M |\vph|^{\frac{(n-2\ga+2)p}{n}}d\nu\)^{\frac{n}{(n-2\ga+2)p}} \le s_{n,\ga,p}^{\frac{n}{(n-2\ga+2)p}} \(\int_M |\vph|^p d\mu\)^{\frac{1}{p}}.\]
This implies
\begin{equation}\label{eq:ccp2}
\nu(E)^{\frac{n}{(n-2\ga+2)p}} \le s_{n,\ga,p}^{\frac{n}{(n-2\ga+2)p}} \mu(E)^{\frac{1}{p}} \quad \textup{for any Borel set } E \subset M.
\end{equation}
Hence the Radon-Nikodym theorem yields
\[\nu(E) = \int_E g\, d\mu \quad \textup{for any Borel set } E \subset M,\]
where the function $g$ on $M$ satisfies
\begin{equation}\label{eq:ccp3}
g(\sigma) = \lim_{r \to 0} \frac{\nu(B(\sigma,r))}{\mu(B(\sigma,r))} \quad \textup{for } \mu\textup{--a.e. } \sigma \in M.
\end{equation}
Because $\mu(M) 
< \infty$, the set $\{\sigma \in M: \mu(\{\sigma\})>0\}$ is at most countable. Let us denote it by $\{\sigma_j\}_{j \in J}$. If $\sigma \in M \setminus \{\sigma_j\}_{j \in J}$, then \eqref{eq:ccp2} and \eqref{eq:ccp3} imply
\[g(\sigma)\le s_{n,\ga,p} \lim_{r \to 0}\mu(B(\sigma,r))^{\frac{2(1-\ga)}{n}} = s_{n,\ga,p} \mu(\{\sigma\})^{\frac{2(1-\ga)}{n}} = 0.\]
Thus (i) holds. Moreover, by \eqref{eq:ccp2} and weak lower semi-continuity of the $L^p(M,\bh)$-norm,
\begin{equation}\label{eq:ccp4}
\mu \ge s_{n,\ga,p}^{-\frac{n}{n-2\ga+2}}\sum_{j \in J}(\nu^{(j)})^{\frac{n}{n-2\ga+2}} \delta_{\sigma_j} \quad \textup{ and } \quad \mu \ge |f|^p dv_{\bh}.
\end{equation}
Since the measures on the right-hand sides in \eqref{eq:ccp4} are mutually singular, (ii) and (iii) must be true.
\end{proof}

We now prove the main result of this section. Thanks to \eqref{eq:wir1} and \eqref{eq:varprob}, we have that $s_{n,\ga,\frac{2n}{n-2\ga}} = \Theta_{n,\ga}(\B^N,g_\pb^+,[g_{\S^n}])$.
\begin{theorem}\label{thm:main21g}
Let $n \in \N$, $n > 2\ga$, $\ga \in (0,1)$, and $p \in (1,\infty)$. If $\lambda_1(-\Delta_{g^+}) > \frac{n^2}{4}-\ga^2$, then
\begin{equation}\label{eq:main21g1}
\mcs_{3;p} = \mcs_{3;n,\ga,p}\(X,\bgg,\rhog\) \ge s_{n,\ga,p}
\end{equation}
and there is $f \in L^p(M,\bh)$ with $\|f\|_{L^p(M,\bh)}=1$ such that
\[\mcs_{3;n,\ga,p}\(X,\bgg,\rhog\) = \left\|\mck_{\bgg,\rhog}^{\mz}f\right\|_{L^{\frac{(n-2\ga+2)p}{n}}(X;\rhog^{\mz},\bgg)}^{\frac{(n-2\ga+2)p}{n}}\]
provided $\mcs_{3;n,\ga,p}(X,\bgg,\rhog) > s_{n,\ga,p}$.
Furthermore, up to a multiplication by a nonzero constant, every maximizer $f$ of the variational problem $\mcs_{3;n,\ga,p}(X,\bgg,\rhog)$ is nonnegative and satisfies
\begin{equation}\label{eq:main21g2}
f(\sigma)^{p-1} = \int_{\ox} \rhog(\xi)^{\mz} \mck_{\bgg,\rhog}^{\mz}(\xi,\sigma) (\mck_{\bgg,\rhog}^{\mz}f)(\xi)^{\frac{(n-2\ga+2)p}{n}-1}(dv_{\bgg})_{\xi} \quad \text{for } \sigma \in M.
\end{equation}
\end{theorem}
\begin{proof}
Given $\sigma \in M$ and $\delta > 0$ small enough, we redefine $\mcc^N_{\bgg}(\sigma,2\delta) = B^n_{\bh}(\sigma,2\delta) \times (0,2\delta) \subset X$.
We identify the sets $C_{\sigma,2\delta}$ and $\mcc^N(0,2\delta) = B^n(0,2\delta) \times (0,2\delta) \subset \R^N_+$ through Fermi coordinates on $\ox$ around $\sigma$.
For a fixed $f \in C^{\infty}_c(\R^n) \setminus \{0\}$, we take $\ep > 0$ so small that the support of $f_{\ep} := \ep^{-\frac{n}{p}}f(\ep^{-1}\cdot)$ is contained in $B^n(0,2\delta)$.
By setting $f_{\ep} = 0$ on $M \setminus B^n_{\bh}(\sigma,2\delta)$, we can also regard it as a function on $M$. Then, \eqref{eq:tech21} gives
\begin{multline*}
\left\|\mck_{\bgg,\rhog}^{\mz} f_{\ep}\right\|_{L^{\frac{(n-2\ga+2)p}{n}}(\mcc^N_{\bgg}(\sigma,2\delta); \rhog^{\mz},\bgg)} \\
\ge \left\|\mck_0^{\mz} f_{\ep}\right\|_{L^{\frac{(n-2\ga+2)p}{n}}(\mcc^N_{\bgg}(\sigma,2\delta); \rhog^{\mz},\bgg)}
- C \delta^{\frac{2(1-\ga)}{p}(\frac{n}{n-2\ga+2}-\frac{1}{t})} \|f_{\ep}\|_{L^p(M,\bh)},
\end{multline*}
\[\|\mck_{n,\ga}f_{\ep}\|_{L^{\frac{(n-2\ga+2)p}{n}}(\mcc^N(0,2\delta);x_N^{\mz})}\le (1+o(1))\|\mck_0^{\mz}f_{\ep}\|_{L^{\frac{(n-2\ga+2)p}{n}}(\mcc^N_{\bgg}(\sigma,2\delta);\rhog^{\mz},\bgg)},\]
and
\[\|f_{\ep}\|_{L^p(M,\bh)} 
\le (1+o(1))\|f_{\ep}\|_{L^p(B^n(0,2\delta))}\]
where $o(1) \to 0$ as $\delta \to 0$ and $t > \frac{n-2\ga+2}{n}$ is a number such that \eqref{eq:tech21} holds. Hence
\begin{align*}
\mcs_{3;p}^{\frac{n}{(n-2\ga+2)p}} &\ge \frac{\left\|\mck_{\bgg,\rhog}^{\mz} f_{\ep}\right\|_{L^{\frac{(n-2\ga+2)p}{n}}(\mcc^N_{\bgg}(\sigma,2\delta); \rhog^{\mz},\bgg)}}{\|f_{\ep}\|_{L^p(M,\bh)}} \\
&\ge (1+o(1)) \frac{\left\|\mck_{n,\ga} f_{\ep}\right\|_{L^{\frac{(n-2\ga+2)p}{n}}(\mcc^N(0,2\delta);x_N^{\mz})}}{\|f_{\ep}\|_{L^p(B^n(0,2\delta))}} - C\delta^{\frac{2(1-\ga)}{p}(\frac{n}{n-2\ga+2}-\frac{1}{t})}.
\end{align*}
Taking $\ep \to 0$ and then $\delta \to 0$, 
we obtain \eqref{eq:main21g1}.

Next, we suppose that $\mcs_{3;p} > s_{n,\ga,p}$. Let $\{f_i\}_{i=1}^{\infty} \subset L^p(M)$ be a maximizing sequence for $\mcs_{3;p}$ with $\|f_i\|_{L^p(M,\bh)}=1$.
Because of Lemma \ref{lemma:Poiest2}, we may assume that \eqref{eq:ccp0} holds as $i\to \infty$. 
If $\mu^{(j)} > 0$ for some $j \in J$, then it follows from Proposition \ref{prop:ccp} and \eqref{eq:main21g1} that
\[\mcs_{3;p} = \nu(\ox) = \left\|\mck_{\bgg,\rhog}^{\mz}f\right\|_{L^{\frac{(n-2\ga+2)p}{n}}(X;\rhog^{\mz},\bgg)}^{\frac{(n-2\ga+2)p}{n}} + \sum_{j \in J}\nu^{(j)}
< \mcs_{3;p} \Bigg[\|f\|_{L^p(M,\bh)}^{\frac{(n-2\ga+2)p}{n}} + \sum_{j \in J}(\mu^{(j)})^{\frac{n-2\ga+2}{n}}\Bigg],\]
which is absurd since
\[1 < \Bigg[\|f\|_{L^p(M,\bh)}^{\frac{(n-2\ga+2)p}{n}} + \sum_{j \in J}(\mu^{(j)})^{\frac{n-2\ga+2}{n}}\Bigg]^{\frac{n}{n-2\ga+2}} \le \|f\|_{L^p(M,\bh)}^p + \sum_{j \in J}\mu^{(j)} \le \mu(M) = 1.\]
Therefore, $\mu^{(j)} = \nu^{(j)} =0$ for all $j \in J$ and $\|f\|_{L^p(M,\bh)}=1$, which means that $f$ is a maximizer of variational problem $\mcs_{3;p}$.

Finally, the proof of Theorem \ref{thm:main11g} with \eqref{eq:Poi25} show that any maximizer of $\mcs_{3;p}$ is nonnegative and satisfies the Euler-Lagrange equation \eqref{eq:main21g2}, up to a scaling by a nonzero constant.
\end{proof}

\subsection{Regularity of the nonnegative maximizers of $\mcs_3$}\label{subsec:reg}
The next proposition tells us that every nonnegative maximizer of $\mcs_3$ in \eqref{eq:mcs} is smooth and positive on $M$, thereby serving as a maximizer of $\mcs_1$.
\begin{prop}\label{prop:reg}
Let $n \in \N$, $n > 2\ga$, $\ga \in (0,1)$, $p \in (1,\infty)$, and $\lambda_1(-\Delta_{g^+}) > \frac{n^2}{4}-\ga^2$.
If $f \in L^p(M,\bh)$ is a nonzero nonnegative function that satisfies \eqref{eq:main21g2}, then $f \in C^{\infty}(M)$ and $f > 0$ on $M$.
\end{prop}
\noindent We will prove Proposition \ref{prop:reg} by developing a regularity theory for integral equations based on a thorough potential analysis.
We will carry out computations in a general context, which can be easily adapted to other relevant problems.
We split the proof into three parts: Lemmas \ref{lemma:reg1}, \ref{lemma:reg2}, and \ref{lemma:reg3}.
\begin{lemma}\label{lemma:reg1}
Under the setting of Proposition \ref{prop:reg}, we have that $f \in L^{\infty}(M,\bh)$.
\end{lemma}
\noindent To prove this lemma, we rely on the following local regularity results.
\begin{lemma}\label{lemma:reg11}
Let $n=N-1 \in \N$, $n > 2\ga$, $\ga \in (0,1)$, and $\lambda_1(-\Delta_{g^+}) > \frac{n^2}{4}-\ga^2$.
We set $a, b \in (1,\infty)$, and $r \in [1,\infty)$.
Given $\sigma_0 \in M$ and $R \in (0,1)$ small enough, let $B^N_{\bgg,+}(\sigma_0,R)$ be the geodesic half-ball on $\ox$ centered at $\sigma_0$ and of radius $R$.
We write $B_R^+ = B^N_{\bgg,+}(\sigma_0,R)$ and $B_R = B^n_{\bh}(\sigma_0,R)$ for brevity.

\medskip \noindent \textup{(a)} Assume that $\frac{n-2\ga+2}{n}<s<t<\infty$,
\[\frac{n-2\ga+2}{ra}+\frac{n}{b} = \frac{2(1-\ga)}{r}, \quad \text{and} \quad \frac{2(1-\ga)}{n-2\ga+2} < \frac{r}{t}+\frac{1}{a}<\frac{r}{s}+\frac{1}{a} < 1.\]
Suppose also that $u_1,v_1 \in L^s(B^+_R;\rhog^{\mz})$, $U \in L^a(B^+_R;\rhog^{\mz})$, and $F \in L^b(B_R)$ are nonnegative functions such that $v_1 \in L^t(B^+_{R/2};\rhog^{\mz})$ and
\[\|U\|_{L^a(B^+_R;\rhog^{\mz})}^{\frac{1}{r}} \|F\|_{L^b(B_R)} \le \ep\]
where $\ep > 0$ is a small number determined by $n$, $\ga$, $a$, $b$, $r$, $s$, $t$, $(\ox,\bgg)$, and $\rhog$. If
\[u_1(\xi) \le \int_{B_R} \mck_{\bgg,\rhog}^{\mz}(\xi,\sigma) F(\sigma) \left[\int_{B^+_R} \rhog(\vth)^{\mz}
\mck_{\bgg,\rhog}^{\mz}(\vth,\sigma) U(\vth)u_1(\vth)^r (dv_{\bgg})_{\vth}\right]^{\frac{1}{r}} (dv_{\bh})_{\sigma} + v_1(\xi)\]
for $\xi \in B^+_R$, then $u_1 \in L^t(B^+_{R/4};\rhog^{\mz})$ and
\[\|u_1\|_{L^t(B^+_{R/4};\rhog^{\mz})}\le C\left[R^{(n-2\ga+2)(\frac{1}{t}-\frac{1}{s})} \|u_1\|_{L^s(B^+_R;\rhog^{\mz})} + \|v_1\|_{L^t(B^+_{R/2};\rhog^{\mz})}\right]\]
where $C > 0$ depends only on $n$, $\ga$, $a$, $b$, $r$, $s$, $t$, $(\ox,\bgg)$, and $\rhog$.

\medskip \noindent \textup{(b)} Assume that $1<s<t<\infty$,
\[\frac{n}{ra}+\frac{n-2\ga+2}{b}= {2(1-\ga)}, \quad \text{and} \quad 0 < \frac{r}{s}+\frac{1}{a} < 1.\]
Suppose also that $u_2,v_2 \in L^s(B_R)$, $F \in L^a(B_R)$, and $U \in L^b(B^+_R;\rhog^{\mz})$ are nonnegative functions such that $v_2 \in L^t(B_{R/2})$ and
\[\|F\|_{L^a(B_R)}^{\frac{1}{r}} \|U\|_{L^b(B^+_R;\rhog^{\mz})} \le \ep\]
where $\ep > 0$ is a small number determined by $n$, $\ga$, $a$, $b$, $r$, $s$, $t$, $(\ox,\bgg)$, and $\rhog$. If
\[u_2(\sigma) \le \int_{B^+_R} \rhog(\xi)^{\mz} \mck_{\bgg,\rhog}^{\mz}(\xi,\sigma) U(\xi)
\left[\int_{B_R} \mck_{\bgg,\rhog}^{\mz}(\xi,\tau) F(\tau)u_2(\tau)^r (dv_{\bh})_\tau\right]^{\frac{1}{r}} (dv_{\bgg})_{\xi} + v_2(\sigma)\]
for $\sigma \in B_R$, then $u_2 \in L^t(B_{R/4})$ and
\[\|u_2\|_{L^t(B_{R/4})}\le C\left[R^{n(\frac{1}{t}-\frac{1}{s})} \|u_2\|_{L^s(B_R)}+\|v_2\|_{L^t(B_{R/2})}\right]\]
where $C > 0$ depends only on $n$, $\ga$, $a$, $b$, $r$, $s$, $t$, $(\ox,\bgg)$, and $\rhog$.

\medskip \noindent \textup{(c)} Assume that $t \in (1,\infty)$ and $U \in L^{\frac{(n-2\ga+2)t}{n+2t(1-\ga)}}(B^+_R;\rhog^{\mz})$ is a nonnegative function. Then
\[\left\|\int_{B^+_R} \rhog(\xi)^{\mz} \mck_{\bgg,\rhog}^{\mz}(\xi,\cdot) U(\xi) (dv_{\bgg})_{\xi}\right\|_{L^t(B_R)} \le C\|U\|_{L^{\frac{(n-2\ga+2)t}{n+2t(1-\ga)}}(B^+_R;\rhog^{\mz})}\]
where $C > 0$ depends only on $n$, $\ga$, $t$, $(\ox,\bgg)$, and $\rhog$.
\end{lemma}
\begin{proof}
Using \eqref{eq:Poi21}, \eqref{eq:Poi25} and Lemma \ref{lemma:Poiest2}, we can modify the proof of \cite[Propositions 2.3, 5.2, and 5.3]{HWY} to deduce the above assertions. We omit the details.
\end{proof}
\begin{proof}[Proof of Lemma \ref{lemma:reg1}]
We will adopt the argument in \cite[Section 5]{HWY} and \cite[Section 4]{HWY2}. 
Below, we present a brief sketch of the proof.

\medskip
Let $p_0 = \frac{1}{p-1} \in (0,\infty)$, $f_0 = f^{p-1} \in L^{p_0+1}(M)$, and $U_0 = \mck_{\bgg,\rhog}^{\mz}f = \mck_{\bgg,\rhog}^{\mz}f_0^{p_0}$ on $\ox$. Then
\begin{equation}\label{eq:reg10}
f_0(\sigma) = \int_{\ox} \rhog(\xi)^{\mz} \mck_{\bgg,\rhog}^{\mz}(\xi,\sigma) U_0(\xi)^{\frac{n-2\ga+2 +2p_0(1-\ga)}{np_0}} (dv_{\bgg})_{\xi} \quad \text{for } \sigma \in M.
\end{equation}
Fixing any $\sigma_0 \in M$ and $R \in (0,1)$ small, we define
\begin{align*}
f_R(\sigma) &= \int_{\ox \setminus B^+_R} \rhog(\xi)^{\mz} \mck_{\bgg,\rhog}^{\mz}(\xi,\sigma) U_0(\xi)^{\frac{n-2\ga+2 +2p_0(1-\ga)}{np_0}}(dv_{\bgg})_{\xi}, \\
U_R(\xi) &= \int_{M \setminus B_R} \mck_{\bgg,\rhog}^{\mz}(\xi,\sigma) f_0^{p_0}(\sigma) (dv_{\bh})_{\sigma}.
\end{align*}
By Lemma \ref{lemma:Poiest2} and the dominated convergence theorem, $U_R \in C^0(\overline{B^+_{R/2}})$ and $f_R \in C^0(\overline{B_{R/2}})$.

\noindent \medskip \textsc{Case 1: $0<p_0<\frac{n-2\ga+2}{n}$}. Owing to the condition on $p_0$, there exists $r \ge 1$ such that $r \in (\frac{1}{p_0},\frac{n-2\ga+2+2p_0(1-\ga)}{np_0})$. It holds that
\[\begin{medsize}
\displaystyle U_0(\xi) \le \int_{B_R} \mck_{\bgg,\rhog}^{\mz}(\xi,\sigma) f_0(\sigma)^{p_0-\frac{1}{r}} \left[\int_{B^+_R} \rhog(\vth)^{\mz}
\mck_{\bgg,\rhog}^{\mz}(\vth,\sigma) U_0(\vth)^{\frac{n-2\ga+2+2p_0(1-\ga)}{np_0}} (dv_{\bgg})_{\vth}\right]^{\frac{1}{r}} (dv_{\bh})_{\sigma} + V_R(\xi)
\end{medsize}\]
where
\[V_R(\xi) := U_R(\xi) +\int_{B_R} \mck_{\bgg,\rhog}^{\mz}(\xi,\sigma) f_0(\sigma)^{p_0-\frac{1}{r}} f_R (\sigma)^{\frac{1}{r}} (dv_{\bh})_{\sigma}.\]
By appealing Lemma \ref{lemma:Poiest2}, we obtain
\[V_R \in L^{\frac{n-2\ga+2}{n}\frac{p_0+1}{p_0}}(B^+_R; \rhog^{\mz},\bgg) \cap L^{\frac{n-2\ga+2}{n}\frac{(p_0+1)r}{p_0r-1}}(B^+_{R/2}; \rhog^{\mz},\bgg).\]
Let us set
\[F(\sigma)=f_0(\sigma)^{p_0-\frac{1}{r}}, \quad U(\vth)=U_0(\vth)^{\frac{n-2\ga+2+2p_0(1-\ga)}{np_0}-r}, \quad u_1(\vth)=U_0(\vth), \quad v_1(\xi)=V_R(\xi),\]
\[a=\frac{(n-2\ga+2)(p_0+1)}{n-2\ga+2+2p_0(1-\ga)-np_0r}, \quad b=\frac{(p_0+1)r}{p_0r-1}, \quad s=\frac{n-2\ga+2}{n}\frac{p_0+1}{p_0},\]
and
\begin{equation}\label{eq:reg11}
\frac{(n-2\ga+2)(n-2\ga+2+2p_0(1-\ga))}{2np_0(1-\ga)} <t<\frac{n-2\ga+2}{n}\frac{(p_0+1)r}{p_0r-1}.
\end{equation}
Then the conditions necessary to apply Lemma \ref{lemma:reg11}(a) are met provided $R > 0$ is sufficiently small, and so $U_0 \in L^t(B^+_{R/4};\rhog^{\mz})$.
From this, the fact that
\[\mck_{\bgg,\rhog}^{\mz}(\cdot,\sigma) \in L^c(B^+_{R/4};\rhog^{\mz}) \quad \text{for all } \sigma \in B_{R/8} \text{ and } c \in \left[1,\tfrac{n-2\ga+2}{n}\),\]
\eqref{eq:reg10}, H\"older's inequality, and the first inequality in \eqref{eq:reg11}, we see
\[f_0 (\sigma) \le C\|U_0\|_{L^t(B^+_{R/4};\rhog^{\mz},\bgg)}^{\frac{n-2\ga+2+2p_0(1-\ga) }{np_0}} +\|f_{R/4}\|_{L^{\infty}(B_{R/8},\bh)} \quad \text{for } \sigma \in B_{R/8}.\]
Since $\sigma_0\in M$ is arbitrary, $f_0 \in L^{\infty}(M,\bh)$, and hence $U_0 \in L^{\infty}(X,\bgg)$.

\medskip \noindent \textsc{Case 2: $\frac{n-2\ga+2}{n}\le p_0<\infty$}. In contrast to Case 1, we employ Lemma \ref{lemma:reg11}(b)--(c) at this time. For Lemma \ref{lemma:reg11}(b), we choose
\[U(\xi) = U_0(\xi)^{\frac{n-2\ga+2+2p_0(1-\ga)}{np_0}-\frac{1}{r}}, \quad F(\tau) = f_0^{p_0-r}(\tau), \quad u_2(\tau) = f_0(\tau),\]
\[v_2(\sigma) = f_R(\sigma) + \int_{B_R^+} \rhog(\xi)^{\mz} \mck_{\bgg,\rhog}^{\mz}(\xi,\sigma) U_0(\xi)^{\frac{n-2\ga+2+2p_0(1-\ga)}{np_0}-\frac{1}{r}} U_R(\xi)^{\frac{1}{r}}(dv_{\bgg})_{\xi},\]
\[a=\frac{p_0+1}{p_0-r}, \quad b=\frac{(n-2\ga+2)(p_0+1)}{np_0}\(\frac{n-2\ga+2 +2p_0(1-\ga)}{np_0}-\frac{1}{r}\)^{-1}, \quad s=p_0+1,\]
and $p_0+1<t<\infty$. We leave the details for interested readers.
\end{proof}

\begin{lemma}\label{lemma:reg2}
Under the setting of Proposition \ref{prop:reg}, we have that $f > 0$ on $M$ and
\begin{equation}\label{eq:reg20}
f \in C^{0,\beta_0}(M) \quad \text{where } \beta_0 \in
\begin{cases}
(0,1] &\text{for } \ga \in (0,\frac{1}{2}), \\
(0,1) &\text{for } \ga = \frac{1}{2}, \\
(0,2(1-\ga)] &\text{for } \ga \in (\frac{1}{2},1).
\end{cases}
\end{equation}
\end{lemma}
\begin{proof}
Let $r_1 > 0$ be the number in Proposition \ref{prop:Poi}, $\sigma_0 \in M$, and $\mcc^N_{\bgg}(\sigma_0,\frac{r_1}{2}) = B^n_{\bh}(\sigma_0,\frac{r_1}{2}) \times (0,\frac{r_1}{2})$.
We write $U_1 = U_0^{\frac{n-2\ga+2+2p_0(1-\ga)}{np_0}}$, which belongs to $L^{\infty}(X,\bgg)$ thanks to Lemma \ref{lemma:reg1}. For $\sigma \in B^n_{\bh}(\sigma_0,\frac{r_1}{4})$, we consider
\begin{align*}
&\ f_0(\sigma) - f_0(\sigma_0) \\
&= \int_{\mcc^N(0,\frac{r_1}{2})} x_N^{\mz} \left[\mck_0^{\mz}((\exp_{\sigma_0}\bx,x_N),\exp_{\sigma_0}\bw) - \mck_{n,\ga}(x,0)\right] U_1(x) \sqrt{|\bgg(x)|} dx \\
& \ + \int_{\mcc^N_{\bgg}(\sigma_0,\frac{r_1}{2})} \rhog(\xi)^{\mz} \left[\(\mck_{\bgg,\rhog}^{\mz}-\mck_0^{\mz}\)(\xi,\sigma)
- \(\mck_{\bgg,\rhog}^{\mz}-\mck_0^{\mz}\)(\xi,\sigma_0)\right] U_1(\xi) (dv_{\bgg})_{\xi} \\
& \ + \int_{X \setminus \mcc^N_{\bgg}(\sigma_0,\frac{r_1}{2})} \rhog(\xi)^{\mz} \left[\mck_{\bgg,\rhog}^{\mz}(\xi,\sigma)-\mck_{\bgg,\rhog}^{\mz}(\xi,\sigma_0)\right] U_1(\xi) (dv_{\bgg})_{\xi} \\
&=: \mcj_{21} + \mcj_{22} + \mcj_{23}.
\end{align*}
Here, $\mck_{n,\ga}$ is the function in \eqref{eq:gaharext2}. We also associated $\sigma$ with $\bw \in B^n(0,\frac{r_1}{4})$ and $\xi \in \mcc^N_{\bgg}(\sigma_0,\frac{r_1}{2})$ with $x=(\bx,x_N) \in \mcc^N(0,\frac{r_1}{2})$ via Fermi coordinates on $\ox$ around $\sigma_0$.
We will estimate each of the terms $\mcj_{21}$, $\mcj_{22}$, and $\mcj_{23}$.

Let $H_1(\bx,\bw) = d_{\bh}(\exp_{\sigma_0}\bx,\exp_{\sigma_0}\bw)^2$ so that $H_1(\bx,0) = |\bx|^2$. There is a constant $C > 0$ depending only on $n$, $\ga$, and $(\ox,\bgg)$ such that
\begin{multline}\label{eq:reg21}
\begin{medsize}
\displaystyle |\mcj_{21}| \le C\|U_1\|_{L^{\infty}(X)} \left[\int_0^{\frac{r_1}{2}} \int_{B^n(0,\frac{3}{2}|\bw|)} \frac{x_N}{(H_1(\bx,\bw)+x_N^2)^{\frac{n+2\ga}{2}}} d\bx dx_N
+ \int_0^{\frac{r_1}{2}} \int_{B^n(0,\frac{3}{2}|\bw|)} \frac{x_N}{|x|^{n+2\ga}} d\bx dx_N \right.
\end{medsize} \\
\begin{medsize}
\displaystyle + \left. \int_0^{\frac{r_1}{2}} \int_{B^n(0,\frac{r_1}{2}) \setminus B^n(0,\frac{3}{2}|\bw|)} x_N
\left\{\frac{1}{(H_1(\bx,\bw)+x_N^2)^{\frac{n+2\ga}{2}}} - \frac{1}{(H_1(\bx,0)+x_N^2)^{\frac{n+2\ga}{2}}}\right\} d\bx dx_N \right].
\end{medsize}
\end{multline}
The estimation of the three integrals in \eqref{eq:reg21} can be achieved as follows: In view of \eqref{eq:dist1}, we obtain $H_1(\bx,\bw) \ge \frac{1}{2}|\bx-\bw|^2$ by lowering the value of $r_1$ if necessary. Also, $B^n(0,\frac{3}{2}|\bw|) \subset B^n(\bw,\frac{5}{2}|\bw|)$.
Hence the first integral in \eqref{eq:reg21} is bounded by a constant multiple of
\[\int_0^{\frac{r_1}{2}} \int_{B^n(0,\frac{5}{2}|\bw|)} \frac{x_N}{|x|^{n+2\ga}} d\bx dx_N \le C \begin{cases}
|\bw| &\text{if } n=1 \text{ and } \ga \in (0,\frac{1}{2}),\\
|\bw| |\log|\bw|| &\text{if } n=1 \text{ and } \ga=\frac{1}{2},\\
|\bw|^{2(1-\ga)} &\text{otherwise}.
\end{cases}\] 
The second integral in \eqref{eq:reg21} can be treated in the same fashion. Furthermore, 
\eqref{eq:dist2} implies that the third integral in \eqref{eq:reg21} is bounded by a constant multiple of
\[|\bw| \int_0^{\frac{r_1}{2}} \int_{B^n(0,\frac{r_1}{2}) \setminus B^n(0,\frac{3}{2}|\bw|)} \frac{x_N}{|x|^{n+2\ga+1}} d\bx dx_N
\le C \begin{cases}
|\bw| &\text{if } \ga \in (0,\frac{1}{2}),\\
|\bw| |\log|\bw|| &\text{if } \ga=\frac{1}{2},\\
|\bw|^{2(1-\ga)} &\text{if } \ga \in (\frac{1}{2},1).
\end{cases}\] 
In conclusion, $|\mcj_{21}| \le C|\bw|^{\beta_0}$ for $|\bw| < \frac{r_1}{4}$.

Invoking the mean value theorem, \eqref{eq:Poi26}, and \eqref{eq:dist1}, we discover $|\mcj_{22}| + |\mcj_{23}| \le Cd_{\bh}(\sigma,\sigma_0)$. Therefore, \eqref{eq:reg20} is true.
In addition, the continuity of $f$ on $M$, Proposition \ref{prop:Poi}(d), and \eqref{eq:main21g2} imply that $U_0 = \mck_{\bgg,\rhog}^{\mz}f > 0$ on $\ox$ and $f_0, f > 0$ on $M$.
\end{proof}

\begin{lemma}\label{lemma:reg3}
Under the setting of Proposition \ref{prop:reg}, we have that $f \in C^{\infty}(M)$.
\end{lemma}
\begin{proof}
We will prove the lemma by establishing ``tangential Schauder estimates" for $U_0$ and $f$.
Throughout the proof, we write $C^{\beta}(M)$ to denote the H\"older space $C^{\lfloor\beta\rfloor,\beta-\lfloor\beta\rfloor}(M)$ where $\lfloor\beta\rfloor$ is the greatest integer less than $\beta \in (0,\infty) \setminus \N$.

\medskip \noindent \textsc{Step 1.} Assume that $f \in C^{\beta_1}(M)$ for some $\beta_1 \in (0,\infty) \setminus \N$.

Given $\xi_0 = (\pi(\xi_0),\rhog(\xi_0)) \in \overline{\mct_{r_1/2}} = M \times [0,\frac{r_1}{2}]$ and $\xi \in \overline{B^n_{\bh}(\pi(\xi_0),\frac{r_1}{4})} \times [0,\frac{r_1}{2}]$ such that $\rhog(\xi) = \rhog(\xi_0)$, we shall deduce the tangential H\"older estimate
\begin{equation}\label{eq:reg34}
\left|\(\nabla_{\pi(\xi)}^jU_0\)(\xi) - \(\nabla_{\pi(\xi)}^jU_0\)(\xi_0)\right| \le C
\begin{cases}
d_{\bh}(\pi(\xi),\pi(\xi_0)) &\text{if } j = 0,\ldots,\lfloor\beta_1\rfloor-1, \\ 
d_{\bh}(\pi(\xi),\pi(\xi_0))^{\beta_1-\lfloor\beta_1\rfloor} &\text{if } j = \lfloor\beta_1\rfloor
\end{cases}
\end{equation}
where $C > 0$ depends only on $n$, $\ga$, $(\ox,\bgg)$, $\rhog$, and $j$.

We write $\xi = (\exp_{\pi(\xi_0)}\bx,x_N)$, $\sigma = \exp_{\pi(\xi_0)}\bw \in B^n_{\bh}(\pi(\xi_0),r_1)$,
$K^{\mz}_{\bgg,\rhog}(x,\bw) = \mck_{\bgg,\rhog}^{\mz}(\xi,\sigma)$, and $K_0^{\mz}(x,\bw) = \mck_0^{\mz}(\xi,\sigma)$. We also set
\[H_1(\bx,\bw) = d_{\bh}(\exp_{\pi(\xi_0)}\bx, \exp_{\pi(\xi_0)}\bw)^2 \quad \text{and} \quad H_2(\bx,\bw) = H_1(\bx,\bx+\bw),\]
and select a cut-off function $\chi_1 \in C^{\infty}_c(B^n(0,r_1))$ such that $\chi_1 = 1$ in $B^n(0,\frac{r_1}{2})$. Then
\begin{align*}
& \ \(\nabla_{\pi(\xi)}^jU_0\)(\xi) - \(\nabla_{\pi(\xi)}^jU_0\)(\xi_0) \\
&= \left[\nabla_{\bx}^j \int_{\R^n} \frac{\ka_{n,\ga} x_N^{2\ga}}{(H_1(\bx,\bw)+x_N^2)^{\frac{n+2\ga}{2}}} f_1(\bw) d\bw
- \left. \nabla_{\bx}^j \int_{\R^n} \frac{\ka_{n,\ga} x_N^{2\ga}}{(H_1(\bx,\bw)+x_N^2)^{\frac{n+2\ga}{2}}} f_1(\bw) d\bw \right|_{\bx=0}\right] \\
&\ + \left[\nabla_{\bx}^j \int_{\R^n} \(K^{\mz}_{\bgg,\rhog}-K_0^{\mz}\)(x,\bw) f_1(\bw) d\bw
- \left. \nabla_{\bx}^j \int_{\R^n} \(K^{\mz}_{\bgg,\rhog}-K_0^{\mz}\)(x,\bw) f_1(\bw) d\bw \right|_{\bx=0}\right] \\
&\ + \int_{B^n(0,r_1)} \left[\(\nabla_{\bx}^jK^{\mz}_{\bgg,\rhog}\)(x,\bw) - \(\nabla_{\bx}^jK^{\mz}_{\bgg,\rhog}\)((0,x_N),\bw)\right] \Big((1-\chi_1)f\sqrt{|\bh|}\Big)(\bw) d\bw \\
&\ + \int_{M \setminus B^n_{\bh}(\pi(\xi_0),r_1)} \left[\(\nabla_{\pi(\xi)}^j \mck_{\bgg,\rhog}^{\mz}\)(\xi,\sigma)
- \(\nabla_{\pi(\xi)}^j \mck_{\bgg,\rhog}^{\mz}\)(\xi_0,\sigma)\right] f(\sigma) (dv_{\bh})_{\sigma} \\
&=: \mcj_{31}^j + \mcj_{32}^j + \mcj_{33}^j + \mcj_{34}^j
\end{align*}
where $f_1 := \chi_1f\sqrt{|\bh|}$. We will estimate each of the terms $\mcj_{31}^j,\ldots,\mcj_{34}^j$.

Suppose that $\lfloor\beta_1\rfloor = 0$. Just as we treated $\mcj_{22}$ and $\mcj_{23}$ in the proof of Lemma \ref{lemma:reg2},
we verify that $|\mcj_{32}^0|+|\mcj_{33}^0|+|\mcj_{34}^0| \le Cd_{\bh}(\pi(\xi),\pi(\xi_0))$. By the mean value theorem, \eqref{eq:dist1}, and \eqref{eq:dist3}, we have
\begin{align}
\left|\mcj_{31}^0\right| &\le C\int_{\R^n} x_N^{2\ga} \left|\frac{1}{(H_2(\bx,\bw)+x_N^2)^{\frac{n+2\ga}{2}}}
- \frac{1}{(H_2(0,\bw)+x_N^2)^{\frac{n+2\ga}{2}}}\right| \left|f_1(\bw+\bx)\right| d\bw \nonumber \\
&\ + C\int_{\R^n} \frac{x_N^{2\ga}}{\(|\bw|^2+x_N^2\)^{\frac{n+2\ga}{2}}} \left|f_1(\bw+\bx) - f_1(\bw)\right| d\bw \label{eq:reg32} \\
&\le C\|f\|_{L^{\infty}(M)}|\bx|\int_{B^n(0,r_1)} \frac{x_N^{2\ga}|\bw|^2}{\(|\bw|^2+x_N^2\)^{\frac{n+2\ga+2}{2}}} d\bw + C\|f\|_{C^{\beta_1}(M)}|\bx|^{\beta_1} \le C\|f\|_{C^{\beta_1}(M)}|\bx|^{\beta_1}. \nonumber
\end{align}
Consequently, we confirm that $|U_0(\xi)-U_0(\xi_0)| \le Cd_{\bh}(\pi(\xi),\pi(\xi_0))^{\beta_1}$, which is \eqref{eq:reg34}.

Suppose next that $\lfloor\beta_1\rfloor \ge 1$. By applying Remark \ref{remark:Poi}(2) with a large number $\ell \ge n+\lfloor\beta_1\rfloor+3$, we observe
\[\left|\(\nabla^j_{\pi(\xi)} \mck_{\bgg,\rhog}^{\mz}\)(\xi,\sigma)\right| \le C \quad \text{if } d_{\bgg}(\xi,\sigma) \ge \frac{r_1}{2} \text{ and } j = 0,\ldots,\lfloor\beta_1\rfloor+1.\]
Using this bound, we get $|\mcj_{33}^j|+|\mcj_{34}^j| \le Cd_{\bh}(\pi(\xi),\pi(\xi_0))$. Furthermore, 
Lemma \ref{lemma:ereg2}, \eqref{eq:ereg22}, and a closer inspection of the inhomogeneous term in equation \eqref{eq:aj} of $a_j$ yield
\begin{equation}\label{eq:reg33}
\left|\(\pa_{\theta_N}^{\ell_1}\nabla_{\bth}^{\ell_2}a_j\)(\theta)\right| \le C\theta_N^{2\ga-\ell_1} \quad \text{for each } \ell_1, \ell_2 \in \N \cup \{0\}.
\end{equation}
Meanwhile, for a multi-index $\alpha_1$ such that $|\alpha_1| = 0,\ldots,\lfloor\beta_1\rfloor$, it holds that
\begin{multline}\label{eq:reg31}
\nabla^{\alpha_1}_{\bx} \int_{\R^n} \left[\(K^{\mz}_{\bgg,\rhog}-K_0^{\mz}\)(x,\bw)\right] f_1(\bw) d\bw \\
= \sum_{\alpha_2+\alpha_3=\alpha_1} \frac{\alpha_1!}{\alpha_2!\alpha_3!} \int_{\R^n} \nabla^{\alpha_2}_{\bx} \left[\(K^{\mz}_{\bgg,\rhog}-K_0^{\mz}\)(x,\bw+\bx)\right] \nabla^{\alpha_3}_{\bx} f_1(\bw+\bx) d\bw.
\end{multline}
Employing \eqref{eq:reg33}, \eqref{eq:dist3}, and \eqref{eq:exp}, we examine each term consisting of $\mck_{\bgg,\rhog}^{\mz}$ (see \eqref{eq:Poi45}--\eqref{eq:Poi04}) as in Step 4 of the proof of Proposition \ref{prop:Poi}. Then we find
\begin{equation}\label{eq:reg35}
\left|\nabla^{\alpha_2}_{\bx} \left[\(K^{\mz}_{\bgg,\rhog}-K_0^{\mz}\)(x,\bw+\bx)\right]\right| \le \frac{Cx_N^{2\ga}}{(|\bw|^2+x_N^2)^{\frac{n+2\ga}{2}}}
\end{equation}
provided the multi-index $\alpha_2$ satisfies $|\alpha_2| \le \lfloor\beta_1\rfloor+1$.
Thus all the integrals on the right-hand side of \eqref{eq:reg31} are uniformly bounded for $x \in \overline{B^n(0,\frac{r_1}{4})} \times [0,\frac{r_1}{2}]$.
From this fact and the mean value theorem, we discover that $|\mcj_{32}^j| \le Cd_{\bh}(\pi(\xi),\pi(\xi_0))$ for $j = 0,\ldots,\lfloor\beta_1\rfloor-1$.
Besides, we can argue as in \eqref{eq:reg32} to derive that $|\mcj_{32}^j| \le Cd_{\bh}(\pi(\xi),\pi(\xi_0))^{\beta_1-\lfloor\beta_1\rfloor}$ for $j = \lfloor\beta_1\rfloor$.
Lastly, \eqref{eq:reg31} and \eqref{eq:reg35} remain valid when we substitute $K^{\mz}_{\bgg,\rhog}-K_0^{\mz}$ with $K_0^{\mz}$.
Therefore, the term $\mcj_{31}^j$ can be controlled as $\mcj_{32}^j$. This validates \eqref{eq:reg34}.

\medskip \noindent \textsc{Step 2.} We assume that \eqref{eq:reg34} holds for some $\beta_1 \in (0,\infty) \setminus \N$. We shall deduce
\begin{equation}\label{eq:reg36}
f \in C^{\beta_2}(M) \quad \text{for any } \beta_2 \in (0,\beta_1+\min\{1,2(1-\ga)\}).
\end{equation}

Let us adopt the notation from the proof of Lemma \ref{lemma:reg2}; particularly, we redefine $H_1(\bx,\bw) = d_{\bh}(\exp_{\sigma_0}\bx,\exp_{\sigma_0}\bw)^2$ and $H_2(\bx,\bw) = H_1(\bx,\bx+\bw)$.
Given $\sigma_0 \in M$, $\sigma \in B^n_{\bh}(\sigma_0,\frac{r_1}{4})$, and $j = 0,\ldots,\lfloor\beta_2\rfloor$, we consider
\begin{align*}
&\begin{medsize}
\displaystyle \ \(\nabla_{\sigma}^jf_0\)(\sigma) - \(\nabla_{\sigma}^jf_0\)(\sigma_0)
\end{medsize} \\
&\begin{medsize}
\displaystyle = \left[\nabla_{\bw}^j \int_{\R^N_+} \frac{\ka_{n,\ga} x_N}{(H_1(\bx,\bw)+x_N^2)^{\frac{n+2\ga}{2}}} U_2(x)dx
- \left. \nabla_{\bw}^j \int_{\R^N_+} \frac{\ka_{n,\ga} x_N}{(H_1(\bx,\bw)+x_N^2)^{\frac{n+2\ga}{2}}} U_2(x)dx \right|_{\bw = 0}\right]
\end{medsize} \\
&\begin{medsize}
\displaystyle \ + \left[\nabla_{\bw}^j \int_{\R^N_+} x_N^{\mz}\(K^{\mz}_{\bgg,\rhog}-K_0^{\mz}\)(x,\bw) U_2(x)dx - \left. \nabla_{\bw}^j \int_{\R^N_+} x_N^{\mz} \(K^{\mz}_{\bgg,\rhog}-K_0^{\mz}\)(x,\bw) U_2(x)dx \right|_{\bw=0}\right] + \mcj_{43}^j
\end{medsize} \\
&\begin{medsize}
=: \mcj_{41}^j + \mcj_{42}^j + \mcj_{43}^j.
\end{medsize}
\end{align*}
Here, $\chi_0 \in C^{\infty}_c(\mcc^N(0,r_1))$ is the cut-off function in \eqref{eq:chi0}, $U_2 := \chi_0U_1\sqrt{|\bgg|}$,
and $\mcj_{43}^j$ is a function such that $|\mcj_{43}^j| \le Cd_{\bh}(\sigma,\sigma_0)$ for all $j = 0,\ldots,\lfloor\beta_2\rfloor$ where $C > 0$ is determined by $n$, $\ga$, $(\ox,\bgg)$, $\rhog$, and $j$.

Suppose that $\lfloor\beta_1\rfloor = 0$. If $\beta_1 + \min\{1,2(1-\ga)\} \le 1$, then 
\begin{align*}
\mcj_{41}^0 &= \int_{\mcc^N(0,r_1)} \left[\frac{\ka_{n,\ga} x_N}{(H_1(\bx,\bw)+x_N^2)^{\frac{n+2\ga}{2}}} - \frac{\ka_{n,\ga} x_N}{(H_1(\bx,0)+x_N^2)^{\frac{n+2\ga}{2}}}\right] [U_2(x)-U_2(\bw,x_N)]dx \\
&\ + \int_{\mcc^N(-\bw,r_1)} \left[\frac{\ka_{n,\ga} x_N}{(H_1(\bx+\bw,\bw)+x_N^2)^{\frac{n+2\ga}{2}}} - \frac{\ka_{n,\ga} x_N}{(H_1(\bx,0)+x_N^2)^{\frac{n+2\ga}{2}}}\right] U_2(\bw,x_N)dx \\
&\ + \int_{\mcc^N(-\bw,r_1) \setminus \mcc^N(0,r_1)} \frac{\ka_{n,\ga} x_N}{(H_1(\bx,0)+x_N^2)^{\frac{n+2\ga}{2}}} U_2(\bw,x_N)dx.
\end{align*}
Applying \eqref{eq:reg34} and estimating as for $\mcj_{21}$ in the proof of Lemma \ref{lemma:reg2},
we see that the first integral is bounded by $C|\bw|^{\beta_2}$ for any $\beta_2 \in (0,\beta_1 + \min\{1,2(1-\ga)\})$ and $|\bw| < \frac{r_1}{4}$.
By \eqref{eq:dist1} and \eqref{eq:dist3}, the second and third integrals are bounded by $C|\bw|$. Therefore, $|\mcj_{41}^0| \le C|\bw|^{\beta_2}$.
Also, $|\mcj_{42}^0| \le C|\bw|$ owing to \eqref{eq:Poi26}. Combining the estimates for $\mcj_{41}^0$, $\mcj_{42}^0$, and $\mcj_{43}^0$,
we obtain \eqref{eq:reg36}. If $\beta_1 + \min\{1,2(1-\ga)\} > 1$, then we infer from the equalities
\[(\nabla_{\bw}H_1)(\bx,\bw) = (\nabla_{\bz}H_2)(\bx,\bz)|_{\bz=\bw-\bx} = -(\nabla_{\bx}H_1)(\bx,\bw) + (\nabla_{\bx}H_2)(\bx,\bz)|_{\bz=\bw-\bx}\]
and the divergence theorem that
\begin{equation}\label{eq:reg37}
\begin{aligned}
&\ \nabla_{\bw} \int_{\R^N_+} \frac{\ka_{n,\ga} x_N}{(H_1(\bx,\bw)+x_N^2)^{\frac{n+2\ga}{2}}} U_2(x)dx \\
&= -\frac{n+2\ga}{2} \int_{\mcc^N(0,r_1)} \frac{\ka_{n,\ga} x_N}{(H_1(\bx,\bw)+x_N^2)^{\frac{n+2\ga+2}{2}}} (\nabla_{\bw}H_1)(\bx,\bw) [U_2(x)-U_2(\bw,x_N)] dx \\
&\ - \int_0^{r_1} \int_{\pa B^n(0,r_1)} \frac{\ka_{n,\ga} x_N}{(H_1(\bx,\bw)+x_N^2)^{\frac{n+2\ga}{2}}} U_2(\bw,x_N) \nu(\bx) dS_{\bx} dx_N \\
&\ -\frac{n+2\ga}{2} \int_{\mcc^N(0,r_1)} \frac{\ka_{n,\ga} x_N}{(H_1(\bx,\bw)+x_N^2)^{\frac{n+2\ga+2}{2}}} (\nabla_{\bx} H_2)(\bx,\bw-\bx)U_2(\bw,x_N) dx
\end{aligned}
\end{equation}
where $dS_{\bx}$ is the surface measure and $\nu$ is the outward unit normal vector on the sphere $\pa B^n(0,r_1)$.
The first integral in the right-hand side of \eqref{eq:reg37} is finite due to \eqref{eq:dist2} and \eqref{eq:reg34}, whereas the last integral is finite due to \eqref{eq:dist3}.
Tedious computation using \eqref{eq:reg37} shows that $|\mcj_{41}^1| \le C|\bw|^{\beta_2-1}$ for $|\bw| < \frac{r_1}{4}$ and $\beta_2 \in (1,\beta_1 + \min\{1,2(1-\ga)\})$.
Also, it holds that $|\mcj_{42}^1| \le C|\bw|$. Hence \eqref{eq:reg36} is again valid.

If $\lfloor\beta_1\rfloor \ge 1$, one can suitably modify the proof for $\lfloor\beta_1\rfloor = 0$ as in Step 1. We omit the details.

\medskip
Having \eqref{eq:reg34} and \eqref{eq:reg36} in hand, one can employ a bootstrap argument to conclude that $f \in C^{\infty}(M)$.
\end{proof}

\subsection{Proof of Theorem \ref{thm:main21}}
Let $p = \frac{2n}{n-2\ga}$. By Lemma \ref{lemma:wir}, Theorem \ref{thm:main21g}, and Proposition \ref{prop:reg},
inequality \eqref{eq:thetaest} holds and there exists a positive maximizer $f \in C^{\infty}(M)$ of variational problem $\mcs_1$ in \eqref{eq:mcs} provided \eqref{eq:thetaest2}.
If we set $U = \mck_{\bgg,\rhog}^{\mz}f$, $\trh = U^{\frac{2}{n-2\ga}}\rhog$, $\tg = U^{\frac{4}{n-2\ga}} \bgg$ on $\ox$,
then \eqref{eq:wir0}, Lemma \ref{lemma:wir}, and its proof guarantee that $\Theta_{n,\ga}\(X,g^+,[\bh]\) = I_{n,\ga}(\ox,\tg,\trh)$ and $\tg \in \mcb_{g^+,\,[\bh]}$.
The proof of Theorem \ref{thm:main21} is completed.

\bigskip
{\small \noindent \textbf{Acknowledgement}
S. Jin was supported by the National Research Foundation of Korea (NRF) grant funded by the Korea government (MSIT) (No. RS-2023-00213407).
S. Kim was supported by Basic Science Research Program through the National Research Foundation of Korea (NRF) funded by the Ministry of Science and ICT (2020R1C1C1A01010133).}

\appendix
\section{}\label{sec:app}
Throughout the appendix, we assume that $n \in \N$, $\ga \in (0,1)$, and $\mz = 1-2\ga$.

\subsection{Embedding theorems}
Here, we present a collection of embedding theorems involving various weighted Sobolev spaces.

\begin{lemma}\label{lemma:embedSob}
Assume that $n > 2\ga$.

\medskip \noindent \textup{(a)} If $\ga \in (0,\frac{1}{2}]$, then there exists a constant $C > 0$ depending only on $n$ and $\ga$
such that the weighted Sobolev inequality \eqref{eq:Sob} on $\R^N_+$ is valid. If $\ga \in (\frac{1}{2},1)$, such $C > 0$ does not exist.

\medskip \noindent \textup{(b)} Suppose that $\ga \in (\frac{1}{2},1)$, $1 \le p < \frac{2(n+1)}{n-1}$, and $R > 0$.\footnote{For the case when $n = 1$, the second condition should be interpreted as $p \ge 1$.}
Then there exists a constant $C > 0$ depending only on $n$, $\ga$, and $p$ such that
\begin{equation}\label{eq:Sob3}
\|U\|_{L^p(B^N_+(0,R);x_N^{\mz})} \le CR^{\frac{n-2\ga+2}{p}-\frac{n-2\ga}{2}} \|\nabla U\|_{L^2(B^N_+(0,R);x_N^{\mz})}
\end{equation}
for $U \in C^{\infty}_c(B^N_+(0,R) \cap B^n(0,R))$.
\end{lemma}
\begin{proof}
(a) If $\ga \in (0,\frac{1}{2}]$, the proofs are found in \cite[Theorem 1.3]{CR}, \cite[(2.1.35)]{Ma}\footnote{The third condition for (2.1.35) must be corrected as
\[\beta = \frac{p}{(p-1)q+p}\left[p\,\alpha-1+\(\frac{1}{p}-\frac{1}{q}\)(n+1)\right] > -\frac{1}{q} \quad \textup{for } m = 1.\]},
and \cite[Proposition 1.1]{DSWZ}.

Assume that $\ga \in (\frac{1}{2},1)$. We pick a cut-off function $\chi_2$ in $C^{\infty}_c(B^N(0,2))$ satisfying $\chi_2 = 1$ in $B^N(0,1)$, and set $\chi_{2R}(x) = \chi_2(\bx,x_N-R)$ for $R > 2$. Then
\[\int_{\R^N_+} x_N^{\mz} |\chi_{2R}(x)|^{\frac{2(n-2\ga+2)}{n-2\ga}} dx \simeq R^{\mz} \quad \textup{and} \quad
\int_{\R^N_+} x_N^{\mz} |\nabla \chi_{2R}(x)|^2 dx \simeq R^{\mz}.\]
Here, the notation $a \simeq b$ for $a, b > 0$ means that $c^{-1} \le \frac{a}{b} \le c$ for some $c > 0$ determined by $N$, $\ga$, and $\chi_2$. As a consequence,
\[\frac{\|\chi_{2R}\|_{L^{\frac{2(n-2\ga+2)}{n-2\ga}}(\R^N_+;x_N^{\mz})}}{\|\chi_{2R}\|_{\dot{H}^{1,2}(\R^N_+;x_N^{\mz})}}
\simeq R^{\frac{\mz}{2} \(\frac{n-2\ga}{n-2\ga+2}-1\)} \to \infty \quad \textup{as } R \to \infty \quad \textup{for } \ga \in \(\tfrac{1}{2},1\).\]

\medskip \noindent \textup{(b)} We follow the argument in the proof of \cite[Proposition 3.1.1]{DMV} by replacing \cite[Theorem 1.2]{FKS} with \cite[Lemma 2.2]{TX}.
In particular, taking $q \in (1,2]$ close to $1$, we obtain \eqref{eq:Sob3} with $p$ arbitrarily close to $\frac{2(n+1)}{n-1}$. See also \cite{CF}.
\end{proof}

\begin{lemma}\label{lemma:embed}
Let $(X^N,g^+)$ be a CCE manifold with conformal infinity $(M^n,[\bh])$, $\rhog$ the geodesic defining function of $(M,\bh)$, and $\bgg = \rhog^2g^+$.
There exists a constant $C > 0$ depending only on $(\ox,\bgg)$, $\rhog$, and $\mz$ such that
\begin{equation}\label{eq:embed}
\int_X \rhog^{-\mz} U^2 dv_{\bgg} \le C \int_X \rhog^{\mz} \(|\nabla_{\bgg} U|_{\bgg}^2 + U^2\) dv_{\bgg} \quad \text{for } U \in H^{1,2}_0(X;\rhog^{\mz},\bgg).
\end{equation}
\end{lemma}
\begin{proof}
By density of $C_c^{\infty}(X)$ in $H^{1,2}_0(X;\rhog^{\mz},\bgg)$, we may assume that $U \in C_c^{\infty}(X)$.
Let $\delta \in (0,1)$ be a small number such that $\bgg = (d\rhog)^2 + h_{\rhog}$ on $\overline{\mct_{\delta}} \subset \ox$
where the metric $h_{\rhog}$ on $M$ varies continuously with respect to $\rhog$. 
Since $U = 0$ on $M$, it holds that
\[U(\sigma,\rhog) = \int_0^{\rhog} \pa_{\rhog} U(\sigma,s) ds \quad \textup{for } \sigma \in M \textup{ and } \rhog \in [0,\delta].\]
By the Cauchy-Schwarz inequality,
\[\rhog^{-\mz} U^2(\sigma,\rhog) \le \rhog^{-\mz} \|\pa_{\rhog} U(\sigma,\cdot)\|_{L^1(0,\rhog)}^2 \le C\rhog^{4\ga-1} \int_0^{\rhog} s^{\mz} |\pa_{\rhog} U(\sigma,s)|^2 ds.\]
Hence Fubini's theorem gives
\[\int_0^{\delta} \int_M \rhog^{-\mz} U^2 dv_{\bh_{\rhog}} d\rhog \le C \int_0^{\delta} \int_M \rhog^{\mz} |\pa_{\rhog} U|^2 dv_{\bh} d\rhog.\]
From this, we can easily derive \eqref{eq:embed}.
\end{proof}
\begin{remark}
If $\mz \le 0$, then $\rhog^{-\mz} \le \rhog^{\mz}$ for $\rhog \in (0,\delta]$, and so \eqref{eq:embed} holds trivially.
Hence the above argument is meaningful only if $\mz > 0$, or equivalently, $\ga \in (\frac{1}{2},1)$.
\end{remark}

\subsection{Elliptic regularity}\label{subsec:app2}
We write $\mcc^N(0,R) = B^n(0,R) \times (0,R) \subset \R^N_+$ for $R > 0$, and let $\nabla_{\bx} = (\pa_{x_1},\ldots,\pa_{x_n})$ be the tangential gradient.
\begin{lemma}\label{lemma:ereg}
Suppose that $r_0 > 0$, $g$ is a $C^3$-metric on $\overline{\mcc^N(0,2r_0)}$ such that
\begin{equation}\label{eq:ereg12}
g_{ij} = \delta_{ij} + O\(|x|^2\) \quad \text{and} \quad g_{iN} = 0 \quad \text{for } i,j = 1,\ldots,n, \quad g_{NN} = 1 \quad \text{on } \overline{\mcc^N(0,2r_0)},
\end{equation}
and $U$ is a solution to
\begin{equation}\label{eq:ereg10}
\begin{cases}
-\textup{div}_g\(x_N^{\mz} \nabla_g U\) + x_N^{\mz} AU = x_N^{\mz}Q &\textup{in } \mcc^N(0,2r_0),\\
U = 0 &\textup{on } B^n(0,2r_0),\\
U \in H^{1,2}(\mcc^N(0,2r_0);x_N^{\mz},g).
\end{cases}
\end{equation}

\medskip \noindent \textup{(a)} Fix any $\ell_0 \in \N \cup \{0\}$. If $|\nabla_{\bx}^{\ell} \pa_{x_N} \sqrt{|g|}(x)| \le Cx_N$ on $\mcc^N(0,2r_0)$ for all $\ell = 0,\ldots,\ell_0$,
then
\begin{equation}\label{eq:ereg1a}
\sum_{\ell=0}^{\ell_0} \left\|\nabla_{\bx}^{\ell} U\right\|_{C^{0,\beta}(\overline{\mcc^N(0,3r_0/2)})}
\le C\left[\|U\|_{L^2(\mcc^N(0,2r_0); x_N^{1-2\ga})} + \sum_{\ell=0}^{\ell_0} \left\|\nabla_{\bx}^{\ell} Q\right\|_{L^{\infty}(\mcc^N(0,2r_0))}\right]
\end{equation}
provided the right-hand side is finite. Here, $C > 0$ and $\beta \in (0,1)$ depend only on $n$, $\ga$, $r_0$, $g$, and $\sum_{\ell=0}^{\ell_0} \|\nabla_{\bx}^{\ell} A\|_{L^{\infty}(\mcc^N(0,2r_0); x_N^{1-2\ga})}$.

\medskip \noindent \textup{(b)} If $A \in C^{0,\beta}(\overline{\mcc^N(0,2r_0)})$, then
\begin{equation}\label{eq:ereg1b}
\left\|x_N^{\mz}\pa_{x_N} U\right\|_{C^{0,\min\{\beta,2(1-\ga)\}}(\overline{\mcc^N(0,r_0)})}
\le C\left[\sum_{\ell=0}^2 \left\|\nabla_{\bx}^{\ell} U\right\|_{C^{0,\beta}(\overline{\mcc^N(0,3r_0/2)})} + \sum_{\ell=0}^1 \left\|\nabla_{\bx}^{\ell} Q\right\|_{L^{\infty}(\mcc^N(0,2r_0))}\right]
\end{equation}
provided the right-hand side is finite. Here, $C > 0$ depends only on $n$, $\ga$, $r_0$, $g$, and $\|A\|_{C^{0,\beta}(\overline{\mcc^N(0,2r_0)})}$.
\end{lemma}
\begin{proof}
(a) By utilizing \eqref{eq:Sob} for $\ga \in (0,\frac{1}{2}]$, \eqref{eq:Sob3} for $\ga \in (\frac{1}{2},1)$, \eqref{eq:ereg12},
and the zero Dirichlet boundary condition of test functions, one can adapt the argument in the proof of \cite[Lemma A.3]{KMW2}
to establish \eqref{eq:ereg1a}.\footnote{The
term $\sum_{\ell=1}^{\ell_0} \left\|\nabla_{\bx}^{\ell} A\right\|_{L^{\infty}(\mcc^N(0,2r_0))}$ appears on the right-hand side of \cite[(A.6)]{KMW2}.
By slightly modifying the proof of \cite[Lemma A.3]{KMW2}, it is possible to make the constant $C > 0$ absorb this term.}

\medskip \noindent (b) We first claim that $x_N^{\mz} (\sqrt{|g|} \pa_{x_N}U)(\bx,x_N)$ converges to a finite value $v(\bx)$ as $x_N \to 0$ for each $\bx \in B^n(0,r_0)$.
The first equation in \eqref{eq:ereg10} can be rewritten as
\begin{equation}\label{eq:ereg11}
\pa_{x_N}\(x_N^{\mz} \sqrt{|g|} \pa_{x_N}U\) = x_N^{\mz} \left[-\sum_{i,j=1}^n\pa_{x_i}\(\sqrt{|g|}g^{ij}\pa_{x_j}U\) + \sqrt{|g|}(AU-Q)\right] =: x_N^{\mz} \mcq.
\end{equation}
Fixing $\bx$, we integrate the both sides of \eqref{eq:ereg11} from $x_N = \ep$ to $r_0$ to get
\begin{equation}\label{eq:ereg13}
\ep^{\mz}\(\sqrt{|g|}\pa_{x_N}U\)(\bx,\ep) = r_0^{\mz}\(\sqrt{|g|}\pa_{x_N}U\)(\bx,r_0) - \int_{\ep}^{r_0} t^{\mz} \mcq(\bx,t)dt.
\end{equation}
By classical elliptic regularity, 
the right-hand side of \eqref{eq:ereg13} is well-defined. On the other hand, there exists $C > 0$ depending only on $n$, $\ga$, $r_0$, $g$, and $\|A\|_{L^{\infty}(\mcc^N(0,2r_0))}$ such that
\[\int_0^{r_0} t^{\mz} |\mcq(\bx,t)|dt \le C\(\sum_{\ell=0}^2 \left\|\nabla_{\bx}^{\ell} U\right\|_{L^{\infty}(\mcc^N(0,\frac{3r_0}{2}))} + \|Q\|_{L^{\infty}(\mcc^N(0,2r_0))}\) < \infty.\]
Therefore, the right-hand side of \eqref{eq:ereg13} has a limit as $\ep \to 0$.

By integrating \eqref{eq:ereg11} with respect to the $x_N$-variable, we arrive at
\begin{equation}\label{eq:ereg14}
\begin{aligned}
x_N^{\mz} \(\sqrt{|g|}\pa_{x_N}U\)(\bx,x_N) &= v(\bx) + \int_0^{x_N} t^{\mz} \mcq(\bx,t) dt \\
&= r_0^{\mz}\(\sqrt{|g|}\pa_{x_N}U\)(\bx,r_0) - \int_{x_N}^{r_0} t^{\mz} \mcq(\bx,t) dt
\end{aligned}
\end{equation}
for $(\bx,x_N) \in \mcc^N(0,r_0)$. Direct calculations using \eqref{eq:ereg11} and \eqref{eq:ereg14} lead to the desired estimate \eqref{eq:ereg1b}.
\end{proof}
\begin{remark}
By integrating \eqref{eq:ereg14} with respect to the $x_N$-variable again, we obtain
\[U(\bx,x_N) = v(\bx) \int_0^{x_N} \frac{s^{-\mz} ds}{\sqrt{|g(\bx,s)|}} + \int_0^{x_N} \frac{s^{-\mz}}{\sqrt{|g(\bx,s)|}} \int_0^s t^{\mz} \mcq(\bx,t) dtds \quad \text{for } (\bx,x_N) \in \overline{\mcc^N(0,r_0)}.\]
Thus there exist functions $F \in C^0(\overline{\mcc^N(0,r_0)})$ and $G \in C^2(\overline{\mcc^N(0,r_0)})$ such that
\[U = F(\bx,x_N)x_N^2 + G(\bx,x_N)x_N^{2\ga} \quad \text{for } (\bx,x_N) \in \overline{\mcc^N(0,r_0)}.\]
This generalizes the conclusion of Proposition \ref{prop:ext2} to a wider class of equations.
\end{remark}

Exploiting Lemma \ref{lemma:ereg}, we derive \eqref{eq:degeq53} for any solution $Z$ to \eqref{eq:degeq54}.
\begin{proof}[Derivation of \eqref{eq:degeq53}]
Choosing functions $\chi_{31} \in C^{\infty}_c(B^n(0,\frac{3r_0}{2}))$ and $\chi_{32} \in C^{\infty}_c([0,\frac{3r_0}{2}))$ satisfying $\chi_{31} = 1$ in $B^n(0,r_0)$ and $\chi_{32} = 1$ in $[0,r_0)$,
we set $\chi_3(x) = \chi_{31}(\bx)\chi_{32}(x_N)$ for $x=(\bx,x_N) \in \mcc^N(0,2r_0)$. Given any $\ep \in (0,\ga)$, we observe
\begin{align*}
\textup{div}_{\bgg}\(x_N^{\mz} \nabla_{\bgg} \(x_N^{2\ga-\ep}\)\) 
&=-\ep(2\ga-\ep)x_N^{\mz} \left[x_N^{-2(1-\ga)-\ep}(1+O(x_N))\right] \\
&\le -\frac{\ep(2\ga-\ep)}{2}x_N^{\mz} x_N^{-2(1-\ga)-\ep} \quad \textup{in } \mcc^N(0,2r_0)
\end{align*}
for $r_0 > 0$ small. By \eqref{eq:degeq55}, Lemma \ref{lemma:ereg}(a), and the classical elliptic estimates,
\[|A|+|Z|+|\nabla_{\bx} Z| \le C \quad \textup{in } \mcc^N\(0,\tfrac{3r_0}{2}\) \quad \textup{and} \quad |\pa_{x_N}Z| \le C \quad \textup{in } B^n\(0,\tfrac{3r_0}{2}\) \times \(r_0,\tfrac{3r_0}{2}\).\]
Hence we see from \eqref{eq:degeq54} that
\[\begin{medsize}
\displaystyle \textup{div}_{\bgg}\(x_N^{\mz} \nabla_{\bgg} (\chi_3Z)\) 
= x_N^{\mz} \left[AZ\chi_3 + Z\Delta_{\bgg}\chi_3 + 2\la \nabla_{\bgg}Z, \nabla_{\bgg}\chi_3 \ra_{\bgg} + \mz x_N^{-1}Z\chi_{31} (\pa_{x_N}\chi_{32})\right] =: x_N^{\mz} \mcq_0
\end{medsize}\]
and $\mcq_0 \in L^{\infty}(\mcc^N(0,2r_0))$. Accordingly, if we choose
\[C_1 = \frac{2(2r_0)^{2(1-\ga)+\ep} (1+\|\mcq_0\|_{L^{\infty}(\mcc^N(0,2r_0))})}{\ep(2\ga-\ep)} > 0,\]
then $C_1 x_N^{2\ga-\ep} \pm \chi_3Z = C_1 x_N^{2\ga-\ep} \ge 0$ on $\pa \mcc^N(0,2r_0)$ and
\[-\textup{div}_{\bgg}\(x_N^{\mz} \nabla_{\bgg} \(C_1x_N^{2\ga-\ep} \pm \chi_3Z\)\) \ge x_N^{\mz} \left[\|\mcq_0\|_{L^{\infty}(\mcc^N(0,2r_0))} \(\frac{2r_0}{x_N}\)^{2(1-\ga)+\ep} \mp\mcq_0\right] \ge 0\]
in $\mcc^N(0,2r_0)$. Since $x_N^{2\ga-\ep}\in H^{1,2}(\mcc^N(0,2r_0);x_N^{\mz},\bgg)$, the weak maximum principle for degenerate elliptic equations yields
\[|(\chi_3Z)(x)| \le C_1x_N^{2\ga-\ep} \quad \textup{for } x=(\bx,x_N) \in \mcc^N(0,2r_0).\]
This completes the proof.
\end{proof}

We also need a variant of Lemma \ref{lemma:ereg} in the proof of Proposition \ref{prop:Poi}.
\begin{lemma}\label{lemma:ereg2}
Given $x=(\bx,x_N) \in \R^N \setminus \{0\}$, let $\bth = \frac{\bx}{|x|}$ and $\theta_N = \frac{x_N}{|x|}$.
Assume that $a$, $c$, and $q$ are functions on $\overline{\S^n_+}$ satisfying $c, \nabla_{\bth}c, q \in C^{0,\beta'}(\overline{\S^n_+})$ for some $\beta' \in (0,1)$ and
\begin{equation}\label{eq:ereg20}
\begin{cases}
-\textup{div}_{g_{\S^n}}\(\theta_N^{\mz}\nabla_{g_{\S^n}}a(\theta)\) + \theta_N^{\mz} c(\theta)a(\theta) = \theta_N^{\mz}q(\theta) &\text{in } \S^n_+,\\
a(\theta) = 0 &\text{on } \pa\S^n_+,\\
a \in H^{1,2}(\S^n_+;\theta_N^{\mz},g_{\S^n}).
\end{cases}
\end{equation}

\medskip \noindent \textup{(a)} It holds that
\begin{equation}\label{eq:ereg23}
\sum_{\ell=0}^2 \left\|\nabla_{\bth}^{\ell} a\right\|_{C^{0,\beta}(\overline{\S^n_+})} + \left\|\theta_N^{\mz} \pa_{\theta_N} a\right\|_{C^{0,\beta}(\overline{\S^n_+})}
\le C\left[\|a\|_{L^2(\S^n_+; \theta_N^{1-2\ga})} + \|q\|_{C^{0,\beta'}(\overline{\S^n_+})} + \sum_{\ell=1}^2 \left\|\nabla_{\bth}^{\ell} q\right\|_{L^{\infty}(\S^n_+)}\right]
\end{equation}
provided the right-hand side is finite. Here, $C > 0$ and $\beta \in (0,\beta')$ depend only on $n$, $\ga$,
and $\|c\|_{C^{0,\beta'}(\overline{\S^n_+})} + \sum_{\ell=1}^2 \left\|\nabla_{\bth}^{\ell} c\right\|_{L^{\infty}(\S^n_+)}$.

\medskip \noindent \textup{(b)} It holds that
\begin{multline}\label{eq:ereg24}
\sum_{\ell=0}^3 \left\|\nabla_{\bth}^{\ell} a\right\|_{C^{0,\beta}(\overline{\S^n_+})} + \sum_{\ell=0}^1 \left\|\theta_N^{\mz} \pa_{\theta_N} \nabla_{\bth}^{\ell}a\right\|_{C^{0,\beta}(\overline{\S^n_+})}
+ \left\|\theta_N^{\mz+1} \pa_{\theta_N}^2a\right\|_{C^{0,\beta}(\overline{\S^n_+})} \\
\le C\left[\|a\|_{L^2(\S^n_+; \theta_N^{1-2\ga})} + \|q\|_{C^{0,\beta'}(\overline{\S^n_+})} + \sum_{\ell=1}^3 \left\|\nabla_{\bth}^{\ell} q\right\|_{L^{\infty}(\S^n_+)}\right]
\end{multline}
provided the right-hand side is finite. Here, $C > 0$ and $\beta \in (0,\beta')$ depend only on $n$, $\ga$,
and $\sum_{\ell=0}^1 \|\nabla_{\bth}^{\ell} c\|_{C^{0,\beta'}(\overline{\S^n_+})} + \sum_{\ell=2}^3 \left\|\nabla_{\bth}^{\ell} c\right\|_{L^{\infty}(\S^n_+)}$.
\end{lemma}
\begin{proof}
Considering the metric tensor in spherical coordinates on $\S^n$, we observe that $\psi_n := \arcsin \theta_N$ is a geodesic defining function of $(\pa\S^n_+, g_{\S^{n-1}})$.
In other words, if $(\psi_1,\ldots,\psi_n)$ denotes Fermi coordinates on $\overline{\S^n_+}$ around a specific point on $\pa\S^n_+$, the resulting metric expression $g$ fulfills the orthogonal condition as in \eqref{eq:ereg12}.

\medskip \noindent \textup{(a)} Writing \eqref{eq:ereg20} with respect to $(\psi_1,\ldots,\psi_n)$-coordinates, we obtain
\begin{equation}\label{eq:ereg21}
\begin{cases}
-\textup{div}_g\((\sin\psi_n)^{\mz}\nabla_ga\) + (\sin\psi_n)^{\mz} ca = (\sin\psi_n)^{\mz}q &\text{in } \mcc^n(0,2r_2),\\
a = 0 &\text{on } B^{n-1}(0,2r_2),\\
a \in H^{1,2}(\mcc^n(0,2r_2);(\sin\psi_n)^{\mz},g),
\end{cases}
\end{equation}
for a small $r_2 > 0$ where $\mcc^n(0,2r_2) = B^{n-1}(0,2r_2) \times (0,2r_2)$. Because $\sin\psi_n/\psi_n = 1 + O(\psi_n^2)$ as $\psi_n \to 0$, we can reason as in the proof of Lemma \ref{lemma:ereg}
to derive H\"older estimates for $a$, $\pa_{\psi_i}a$, $\pa_{\psi_i\psi_j}a$, and $(\sin\psi_n)^{\mz} \pa_{\psi_n}a$ for $i,j = 1,\ldots,n-1$. Thus \eqref{eq:ereg23} holds.

\medskip \noindent \textup{(b)} By inspecting the equation of $\pa_{\psi_i}a$ and using $\nabla_{\bth}^3c,\, \nabla_{\bth}^3q \in L^{\infty}(\S^n_+)$ and $\nabla_{\bth}c \in C^{0,\beta'}(\overline{\S^n_+})$,
we deduce H\"older estimates for $\pa_{\psi_i\psi_j\psi_k}a$ and $(\sin\psi_n)^{\mz} \pa_{\psi_n\psi_i}a$ for $i,j,k = 1,\ldots,n-1$. Also, we infer from \eqref{eq:ereg21} that
\begin{multline}\label{eq:ereg22}
(\sin\psi_n)^{\mz+1}\pa_{\psi_n}^2a = -\mz(\sin\psi_n)^{\mz}\pa_{\psi_n}a \\
+ (\sin\psi_n)^{2(1-\ga)}\left[ca-q-\frac{1}{\sqrt|g|} \sum_{i,j=1}^{n-1} \pa_{\psi_i}(\sqrt{|g|}g^{ij}\pa_{\psi_j}a) - \frac{\pa_{\psi_n}\sqrt{|g|}}{\sqrt{|g|}} \pa_{\psi_n}a\right]
\end{multline}
in $\mcc^n(0,2r_2)$. From \eqref{eq:ereg22} and $c, q \in C^{0,\beta'}(\overline{\S^n_+})$, we observe that $(\sin\psi_n)^{\mz+1}\pa_{\psi_n}^2a \in C^{0,\beta}(\overline{\S^n_+})$. Thus \eqref{eq:ereg24} holds.
\end{proof}

\subsection{Weighted spherical harmonics}
As before, we write $\theta_N = \frac{x_N}{|x|}$ for $x \in \R^N$.
\begin{prop}\label{prop:spherical}
Given $n \in \N$, $\ga \in (0,1)$, and $\mz = 1-2\ga$, let $\textup{Spec}(-\textup{div}_{g_{\S^n}}(\theta_N^{\mz} \nabla_{g_{\S^n}} \cdot))$
be the set of all eigenvalues of the weighted Laplacian $-\textup{div}_{g_{\S^n}}(\theta_N^{\mz} \nabla_{g_{\S^n}} \cdot)$ with Dirichlet boundary condition on $\pa\S^n_+$. Then
\[\textup{Spec}\(-\textup{div}_{g_{\S^n}}\(\theta_N^{\mz} \nabla_{g_{\S^n}} \cdot\)\) = \{(\ell+2\ga)(\ell+n): \ell \in \N \cup \{0\}\}.\]
Also, if $Y_{\ell} \in L^2(\S^n_+;\theta_N^{\mz},g_{\S^n})$ is an eigenfunction corresponding to the eigenvalue $(\ell+2\ga)(\ell+n)$, then it is a restriction of a solution to
\begin{equation}\label{eq:har}
\begin{cases}
-\textup{div}\(x_N^{\mz} \nabla U\) = 0 &\textup{in } \R^N_+,\\
U = 0 &\textup{on } \R^n
\end{cases}
\end{equation}
such that
\[U(x) = \sum\limits_{\substack{k, m \in \N \cup \{0\} \\ k+2m=\ell}} x_N^{2(\ga+m)}p_k(\bx) \quad \text{for all } x=(\bx,x_N) \in \R^N_+\]
where $p_k(\bx)$ is a homogeneous polynomial in $\bx \in \R^n$ of degree $k$.
\end{prop}
\begin{proof}
The proposition and its proof have already been presented in \cite[Section 4]{MN}. Here, we provide a more detailed explanation based on the argument in \cite[Chapter III, Section 3]{St}.

\medskip \noindent \textsc{Step 1.} For $k, m \in \mathbb{Z}$, we define a vector space
\[\begin{medsize}
\mca_{k,m} = \begin{cases}
\textup{span}\left\{x_N^{2m} p(\bx): \textup{$p(\bx)$ is a homogeneous polynomial in $\bx \in \R^n$ of degree $k$}\right\} &\text{if } k, m \in \N \cup \{0\},\\
\emptyset &\text{otherwise}.
\end{cases}
\end{medsize}\]
Then $\mca := \sum_{k,m=0}^{\infty} \mca_{k,m}$ is a subalgebra of $C^0(\overline{B^N_+(0,1)})$ that contains nonzero constant functions and separates points.
Therefore the Stone-Weierstrass theorem implies that $\mca$ is dense in $C^0(\overline{B^N_+(0,1)})$.

Given any $f \in C^0(\overline{\S^n_+})$, we observe from the dominated convergence theorem that
\[f_{\ep} := \theta_N^{2\ga} \frac{f}{(\theta_N+\ep)^{2\ga}} \to f \quad \text{in } L^2(\S^n_+;\theta_N^{\mz},g_{\S^n}) \quad \text{as } \ep \to 0.\]
Since $C^0(\overline{\S^n_+})$ is dense in $L^2(\S^n_+;\theta_N^{\mz},g_{\S^n})$, so is $\theta_N^{2\ga} C^0(\overline{\S^n_+})$.
Applying the Weierstrass approximation theorem, we conclude that $\theta_N^{2\ga} \mca|_{\overline{\S^n_+}}$ is dense in $L^2(\S^n_+;\theta_N^{\mz},g_{\S^n})$
where $\mca|_{\S^n_+}$ is the space of the restrictions of polynomials in $\mca$ to $\overline{\S^n_+}$.

\medskip \noindent \textsc{Step 2.} Given any $k, m \in \N \cup \{0\}$, we define an inner product on the vector space $x_N^{2\ga} \mca_{k,m}$ as follows:
For $x_N^{2(\ga+m)} p_1(\bx), x_N^{2(\ga+m)} p_2(\bx) \in x_N^{2\ga} \mca_{k,m}$ (so that $\deg p_1 = \deg p_2 = k$), we set
\begin{multline*}
\la x_N^{2(\ga+m)} p_2(\bx), x_N^{2(\ga+m)} p_1(\bx)\ra \\
= \begin{cases}
4^m m! (m+\ga)(m-1+\ga) \cdots (2+\ga)(1+\ga) \left[p_2\(\dfrac{\pa}{\pa \bx}\)\right] p_1(\bx) &\text{for } m \ge 1,\\
\left[p_2\(\dfrac{\pa}{\pa \bx}\) \right]p_1(\bx) &\text{for } m = 0
\end{cases}
\end{multline*}
where $p_2(\frac{\pa}{\pa \bx})$ is the differential operator such that each component $x_i$ in $p_2$ is replaced with $\pa_{x_i}$. We extend this to the space $x_N^{2\ga} \mca$ by letting
\[\la x_N^{2(\ga+m_2)} p_2(\bx), x_N^{2(\ga+m_1)} p_1(\bx)\ra = 0 \quad \text{if } \deg p_1 = k_1,\, \deg p_2 = k_2,\, (k_1,m_1) \ne (k_2,m_2).\]
Then, by conducting straightforward but long computations, one sees that if $q_1 \in x_N^{2\ga} (\mca_{k+2,m} + \mca_{k,m+1})$ and $q_2 \in x_N^{2\ga} \mca_{k,m}$,
then $x_N^{-\mz} \textup{div}(x_N^{\mz} \nabla q_1) \in x_N^{2\ga} (\mca_{k,m}+\mca_{k+2,m-1}+\mca_{k-2,m+1})$ and
\[\la|x|^2 q_2, q_1\ra = \la q_2, x_N^{-\mz} \textup{div}\(x_N^{\mz} \nabla q_1\)\ra.\]
Therefore, if we set
\[x_N^{2\ga}\mca_{\ell} = x_N^{2\ga} \sum_{\substack{k,m \in \N \cup \{0\} \\ k+2m=\ell}} \mca_{k,m}
\quad \text{and} \quad x_N^{2\ga}\mch_{\ell} = \left\{U \in x_N^{2\ga}\mca_{\ell}: U \textup{ is a solution to \eqref{eq:har}}\right\}\]
for $\ell \in \N \cup \{0\}$, then $x_N^{2\ga}\mca_{\ell} = x_N^{2\ga}\mch_{\ell} \oplus |x|^2 x_N^{2\ga} \mca_{\ell-2}$.

\medskip \noindent \textsc{Step 3.} Let $\theta_N^{2\ga} H_{\ell}$ be the space of the restriction of functions in $x_N^{2\ga}\mch_{\ell}$ to $\S^n_+$.
The elements of this space are referred to as weighted spherical harmonics.
An induction argument with Step 2 shows that every element of $\theta_N^{2\ga}\mca|_{\overline{\S^n_+}}$ can be written as a finite linear combination of weighted spherical harmonics.
In view of Step 1, it holds that $L^2(\S^n_+;\theta_N^{\mz},g_{\S^n}) = \sum_{\ell=0}^{\infty} \theta_N^{2\ga} H_{\ell}$.

By Lemma \ref{lemma:Poi2},
\[-\textup{div}\(x_N^{\mz} \nabla \(|x|^{2\ga+\ell}a(\theta)\)\) = |x|^{\ell-1} \left[\text{div}_{g_{\S^n}}\(\theta_N^{\mz}\nabla_{g_{\S^n}}a(\theta)\) + (\ell+2\ga)(\ell+n) \theta_N^{\mz}a(\theta)\right].\]
As a result, $\theta_N^{2\ga} H_{\ell}$ is the eigenspace of $-\text{div}_{g_{\S^n}}(\theta_N^{\mz}\nabla_{g_{\S^n}} \cdot)$ associated with the eigenvalue $(\ell+2\ga)(\ell+n)$. The proof is finished.
\end{proof}

\subsection{Expansions of the squares of Riemannian distances}
Let $\Psi = d_{\bh}^2$ be the square of the Riemannian distance on a smooth closed manifold $(M,\bh)$.
By appealing the smoothness of the exponential map defined in an open neighborhood of the zero section of $TM$ and the inverse function theorem, one can easily check that $\Psi$ is smooth near the diagonal of $M \times M$.
In the following proposition, we study the expansion of $\Psi$, referring to \cite{Ni, DZ}.
\begin{lemma}
Fix any $\sigma_0 \in M$ and $r_3 \in (0,\frac{1}{2}\textup{inj}(M,\bh))$.

\medskip \noindent \textup{(a)} Let $H_1(\bx,\bw) = \Psi(\exp_{\sigma_0}\bx,\exp_{\sigma_0}\bw)$ for $\bx, \bw \in \R^n$ such that $|\bx|, |\bw| < r_3$. Then
\begin{equation}\label{eq:dist1}
H_1(\bx,\bw) = |\bx-\bw|^2 + O\(\(|\bx|^2+|\bw|^2\)|\bx-\bw|^2\)
\end{equation}
and
\begin{equation}\label{eq:dist2}
(\nabla_{\bw}H_1)(\bx,\bw) = 2(\bw-\bx) + O\(\(|\bx|^2+|\bw|^2\)|\bx-\bw|\).
\end{equation}

\medskip \noindent \textup{(b)} Let $H_2(\bx,\bz) = \Psi(\exp_{\sigma_0}\bx,\exp_{\sigma_0}(\bx+\bz))$ for $\bx, \bz \in \R^n$ such that $|\bx|, |\bz| < r_3$. Then
\begin{equation}\label{eq:dist3}
\left|(\nabla_{\bx}^{\ell} H_2)(\bx,\bz)\right| = O\(|\bz|^2\) \quad \text{for any } \ell \in \N \cup \{0\}.
\end{equation}
\end{lemma}
\begin{proof}
(a) Note that $H_2(\bx,\bz) 
= H_1(\bx,\bw)$ for $\bz = \bw-\bx = (z_1,\ldots,z_n)$. In the proof of Lemma 5.6 in \cite{DZ}, it was shown that
\begin{equation}\label{eq:dist11}
H_2(\bx,\bz) = \sum_{i,j=1}^n \bh_{ij}(\bx)z_iz_j + \frac{1}{3!} \sum_{i,j,k=1}^n\pa_{z_iz_jz_k}H_2(\bx,0)z_iz_jz_k + O\(|z|^4\).
\end{equation}
Since $|(\nabla_{\bx} d_{\bh})(\exp_{\sigma_0}\bx,\exp_{\sigma_0}\bw)|_{\bh} = 1$ for each fixed $\bw$, it holds that
\begin{equation}\label{eq:dist12}
|(\nabla_{\bx} H_1)(\bx,\bw)|_{\bh}^2 = \sum_{i,j=1}^n \bh^{ij}(\bx) (\pa_{x_i}H_1\pa_{x_j}H_1)(\bx,\bw) = 4H_1(\bx,\bw).
\end{equation}
Taking $\pa_{x_kx_lx_m}$ on the both sides of \eqref{eq:dist12} and letting $\bx = \bw$, we find
\[(\pa_{x_kx_lx_m} H_1)(\bx,\bx) = \(\pa_{x_k}\bh_{lm} + \pa_{x_l}\bh_{mk} + \pa_{x_m}\bh_{kl}\)(\bx) \quad \text{for } k,l,m = 1,\ldots,n,\]
which together with Lemma \ref{lemma:Poi1} yields
\begin{equation}\label{eq:dist13}
\pa_{z_iz_jz_k}H_2(\bx,0) = (\pa_{x_ix_jx_k} H_1)(\bx,\bx) = O(|\bx|).
\end{equation}
Plugging \eqref{eq:dist13} and the equality $\bh_{ij}(\bx) = \delta_{ij} + O(|\bx|^2)$ into \eqref{eq:dist11}, we obtain \eqref{eq:dist1}.

Equality \eqref{eq:dist2} is a simple consequence of the relation $(\nabla_{\bw}H_1)(\bx,\bw) = (\nabla_{\bz}H_2)(\bx,\bz)$ and the above calculations.

\medskip \noindent \textup{(b)} By taking the $\ell$-th derivative of \eqref{eq:dist11} with respect to $\bx$, we readily derive \eqref{eq:dist3}.
\end{proof}
\begin{remark}
Following the derivation of \eqref{eq:dist3}, one can verify
\begin{equation}\label{eq:exp}
\left|\nabla_{\bx}^{\ell} \left[\exp^{-1}_{\exp_{\pi(\xi_0)}(\bx+\bz)} \exp_{\pi(\xi_0)} \bx\right]\right| = O(|\bz|)
\end{equation}
for $\ell \in \N \cup \{0\}$ and $\bx, \bz \in \R^n$ such that $|\bx|,|\bz| < r_3$.
\end{remark}

\end{document}